\documentclass[12pt,reqno]{amsart}
\usepackage{amssymb,longtable}
\numberwithin{equation}{section}
\newtheorem{thm}{Theorem}[section]
\newtheorem{prop}[thm]{Proposition}
\newtheorem{lem}[thm]{Lemma}
\newtheorem{cor}[thm]{Corollary}

\theoremstyle{remark}
\newtheorem{rem}[thm]{Remark}

\newcommand{\R}{\mathbb{R}}
\newcommand{\Z}{\mathbb{Z}}
\newcommand{\N}{\mathbb{N}}
\newcommand{\C}{\mathbb{C}}
\newcommand{\F}{\mathcal{F}}

\renewcommand{\L}{\mathcal{L}}
\newcommand{\K}{\mathcal{K}}
\newcommand{\m}{m}
\newcommand{\mm}{\mathfrak{M}}

\newcommand{\Dab}{\L_{\al,\be}}
\newcommand{\U}{\underline}

\newcommand{\al}{\alpha}
\newcommand{\be}{\beta}
\newcommand{\ga}{\gamma}

\newcommand{\X}{\mathcal{X}}
\newcommand{\Y}{\mathcal{Y}}
\newcommand{\De}{\Delta}
\newcommand{\Th}{\Theta}
\newcommand{\La}{\Lambda}
\newcommand{\ka}{\kappa}
\newcommand{\si}{\sigma}
\newcommand{\fy}{\varphi}
\newcommand{\e}{\varepsilon}
\newcommand{\p}{\partial}
\newcommand{\la}{\lambda}
\newcommand{\de}{\delta}

\newcommand{\s}{\sigma}
\renewcommand{\th}{\theta}
\newcommand{\x}{\xi}
\newcommand{\y}{\eta}

\newcommand{\na}{\nabla}
\newcommand{\I}{\infty}
\newcommand{\LR}[1]{\langle {#1} \rangle}
\newcommand{\lec}{{\, \lesssim \, }}
\newcommand{\gec}{{\, \gtrsim \, }}

\renewcommand{\bar}{\overline}
\renewcommand{\hat}{\widehat}
\newcommand{\ti}{\widetilde}

\renewcommand{\S}{\mathcal{S}}

\newcommand{\om}{\omega}

\newcommand{\supp}{\operatorname{supp}}

\newcommand{\EQ}[1]{\begin{equation} \begin{split} #1
 \end{split} \end{equation}}
\newcommand{\pr}{\\ &}
\newcommand{\pt}{&}
\newcommand{\pq}{\quad }
\newcommand{\pn}{}

\newcommand{\CAS}[1]{\begin{cases} #1 \end{cases}}
\newcommand{\tabhead}[1]{\hline \quad {\it #1} & \\} 
\newcommand{\Del}[1]{}

\newcommand{\IN}[1]{\text{ in }#1}

\newcommand{\V}{\vec}
\newcommand{\Da}{\LR{\na}}

\newcommand{\paren}[1]{\left( #1 \right)}
\newcommand{\reg}{\operatorname{reg}}
\newcommand{\str}{\operatorname{str}}
\newcommand{\dec}{\operatorname{dec}}

\newcommand{\iD}{i\LR{\na}}
\renewcommand{\limsup}{\varlimsup}
\renewcommand{\liminf}{\varliminf}
\newcommand{\emf}[1]{{\bf #1}}
\newcommand{\CTM}{C_{\scriptscriptstyle{\operatorname{TM}}}}
\newcommand{\CTMS}{\CTM^\star}
\newcommand{\bul}{\bullet}
\newcommand{\MP}{\mathcal{MC}}

\newcommand{\prq}{\\ &\quad}
\newcommand{\prqq}{\\ &\qquad}

\newcommand{\mat}[1]{\begin{pmatrix} #1 \end{pmatrix}}
\newcommand{\NLKG}{{\sf NLKG}}
\newcommand{\NLS}{{\sf NLS}}
\newcommand{\NLW}{{\sf NLW}}

\author{S.~Ibrahim}
\address{Department of Mathematics and Statistics,\\University of Victoria\\
 PO Box 3060 STN CSC\\   Victoria, BC, V8P 5C3\\ Canada}
\email{ibrahim@math.uvic.ca}
\urladdr{http://www.math.uvic.ca/~ibrahim/}

\author{N.~Masmoudi}
\address{The Courant Institute for Mathematical Sciences,\\New York University} 
\email{masmoudi@courant.nyu.edu} 
\urladdr{http://www.math.nyu.edu/faculty/masmoudi} 

\author{K.~Nakanishi}
\address{Department of Mathematics, Kyoto University}
\email{n-kenji@math.kyoto-u.ac.jp}

\title[Scattering threshold for NLKG]{Scattering threshold for \\ the focusing nonlinear Klein-Gordon equation}

\begin{document}
\begin{abstract} 
We show scattering versus blow-up dichotomy below the ground state energy for the focusing nonlinear Klein-Gordon equation, in the spirit of Kenig-Merle for the $H^1$ critical wave and Schr\"odinger equations. 
Our result includes the $H^1$ critical case, where the threshold is given by the ground state for the massless equation, and the 2D square-exponential case, where the mass for the ground state may be modified, depending on the constant in the sharp Trudinger-Moser inequality. 
The main difficulty is the lack of scaling invariance in both the linear and the nonlinear terms. 
\end{abstract}

%@@@@@@@@@@@@@@@@@@@@@@@@@@@@@@@@@@@@

\subjclass[2010]{35L70, 35B40, 35B44, 47J30} 
\keywords{nonlinear Klein-Gordon equation, scattering theory, blow-up solution, ground state, Sobolev critical exponent, Trudinger-Moser inequality} 

%@@@@@@@@@@@@@@@@@@@@@@@@@@@@@@@@@@@@

\maketitle
\tableofcontents

%@@@@@@@@@@@@@@@@@@@@@@@@@@@@@@@@@@@@

\section{Introduction}
\subsection{The problem and overview}
We study global and asymptotic behavior of solutions in the energy space for the nonlinear Klein-Gordon equation (\NLKG):
\EQ{ \label{NLKG} 
  \ddot u - \Delta u + u = f'(u), \quad u:\R^{1+d}\to\R, \pq(d\in\N)}
where $f:\R\to\R$ is a given function. 
Typical examples that we can treat are the power nonlinearities in any dimensions 
\EQ{ \label{power f}
 f(u) = \la |u|^{p+2}, \pq (2_\star< {p+2} \le 2^\star, \pq \la\ge 0),}
where $2_\star$ and $2^\star$ respectively denote the $L^2$ and $H^1$ critical powers
\EQ{ \label{def 2star}
 2_\star= 2 + \frac{4}{d}, \pq 2^\star=\CAS{2+\frac{4}{d-2} &(d\ge 3) \\ \I &(d\le 2)},}
and the square-exponential nonlinearity in two spatial dimensions
\EQ{
 f(u) = \la |u|^p e^{\ka|u|^2}, \pq (d=2,\ p>4,\ \la\ge 0,\ \ka>0 ),}
which is related to the critical case for the Trudinger-Moser inequality. 
The equation conserves (at least formally) the energy 
\EQ{ \label{def E}
 E(u;t)=E(u(t),\dot u(t)):=\int_{\R^d}\frac{|\dot u|^2+|\na u|^2+|u|^2}{2}-f(u)dx.}

The main goal in this paper is to give necessary and sufficient conditions for the solution $u$ to scatter, which means that $u$ is asymptotic to some free solutions as $t\to\pm\I$, under the condition that $u$ has less energy than the least energy static solution, namely the ground state. In the defocusing case, where $f$ has the opposite sign, one has the scattering result for all finite energy solutions, see \cite{Brenner,GV1,3Dcrit,2Dsubcrit,Nak-rem,2DcNLKG}. 
In the focusing case, it turns out that the solutions below the ground energy split into the scattering solutions and the blow-up solutions (in both time directions in both cases). 
Such results have been recently established for many other equations including the nonlinear wave equation (\NLW), the nonlinear Schr\"odinger equation (\NLS), the Yang-Mills system and the wave maps, since Kenig-Merle's \cite{KM2} on NLS with the $H^1$ critical power (i.e. $p+2=2^\star$ in \eqref{power f}), see \cite{AN,CKM,DHR,KM1,KVZ1,KVZ2,KS,ST,Tao} and the references therein. 

To be more precise, let us recall the result by Kenig-Merle for the critical nonlinear wave equation 
\EQ{ \label{CNW}
 \ddot u - \Delta u = f'(u), \pq f(u)=|u|^{2^\star}.}
Let $E^{(0)}(u)$ be the conserved energy, and $Q$ be a static solution with the least energy: 
\EQ{
 E^{(0)}(u):=\int_{\R^d}\frac{|\dot u|^2+|\na u|^2}{2} + f(u) dx, 
 \pq Q(x):=\left[1+\frac{|x|^2}{d(d-2)}\right]^{-(d-2)/2}.} 
Kenig-Merle \cite{KM1} proved that every solution with $E^{(0)}(u)<E^{(0)}(Q)$ scatters in the energy space as $t\to\pm\I$, provided that $\|\na u(0)\|_{L^2}<\|\na Q\|_{L^2}$, and otherwise it blows up in finite time both for $t>0$ and for $t<0$. 
The idea of their proof is to bring the concentration compactness argument into the scattering problem by using space-time norms and the concept of ``critical element", that is the minimal non-scattering solution.

The equations in those papers following Kenig-Merle have a common important property---the scaling invariance. It is further shared with the solution space (either the energy space or $L^2$, i.e. the critical case), except for the NLS with a subcritical power \cite{DHR,AN}. The scaling invariance brings significant difficulties for the analysis, but also a lot of algebraic or geometric structures and simplifications. Hence it is a natural question what happens if the invariance is broken in the linear and the nonlinear parts of the equation. This is the main technical challenge in this paper. 

The dichotomy into the global existence and the blow-up has been known \cite{PS} long before the scattering result of Kenig-Merle, under the name of ``potential well", which is defined by derivatives of the static energy functional. More precisely, Payne-Sattinger \cite{PS} proved on bounded domains the dichotomy into blow-up and global existence for solutions below the ground energy, by the sign of the functional 
\EQ{
 K_{1,0}(u) := \int |\na u|^2+|u|^2-uf'(u) dx.}
It is easy to observe that their argument applies to the whole space $\R^d$ as soon as one has the local wellposedness in the energy space. Hence our primary task is to prove the scattering result in the region of global existence. Then our first problem due to the inhomogeneity is that the above functional $K_{1,0}$ is not suited for the scattering proof, though it is useful for the blow-up and global existence. More specifically, we want to use the functional 
\EQ{
 K_{d,-2}(u) := \int 2|\na u|^2+d[uf'(u)-2f(u)] dx,}
which is related to the virial identity. There is actually a one-parameter family of functionals, corresponding to various scalings, each of which defines a splitting of the solutions below the ground energy by its sign. For example, Shatah \cite{Shatah} used another functional 
\EQ{
 K_{0,1}(u) := \int \frac{d-2}{2}|\na u|^2 + \frac{d}{2} |u|^2 - d f(u) dx,}
to prove the instability of the standing waves. Note that in his proof the instability is not given by blow-up in the region $K_{0,1}(u)<0$. More recently, Ohta-Todorova \cite{OT} proved blow-up in the region $K_{d,-2}(u)<0$, but they need radial symmetry for the powers $p$ close to $2^\star$. 

The special feature of the critical wave equation \eqref{CNW} is that those functionals are the same modulo constant multiples, which is exactly due to the scaling invariance. For the NLS with a subcritical power \cite{DHR,AN}, the functionals are different from each other, but the situation is much better than NLKG, because they contain only two terms (without the $L^2$ norm), the $L^2$ is another conserved quantity, and the virial identity is used both for the blow-up and for the scattering, while $K_{1,0}$ is not so useful for NLS. 

It turns out, however, that those algebraically different functionals for NLKG define the same splitting below the threshold energy. This observation does not seem to be well recognized, but it is indeed crucial for the proof of the dichotomy, since we need different functionals for the blow-up and for the scattering. 

One interesting feature resulting from the breakdown of the scaling is that, for some nonlinearity, the energy threshold is not given by the ground state of the original NLKG, but by that of a modified equation. More precisely, for the $H^1$ critical power ($p+2=2^\star$) in three dimensions or higher, the threshold is given by that of the critical wave equation, or massless Klein-Gordon equation. This can be expected because the concentration by the critical scaling makes 
  the $L^2$ norm vanish  while preserving other components, namely the massless energy. However the transition from the Klein-Gordon to the wave requires non-trivial amount of effort in the scattering proof. 

We find another instance of mass modification, which is more surprising. That is in two dimensions and for nonlinearities  which grow slightly slower than the square exponential $e^{|u|^2}$, where the mass for the threshold ground energy can change to any number between $0$ and $1$, depending on the constant in the sharp ($L^2$) Trudinger-Moser inequality. Thus we prove the existence of extremizers as well as the ground states with mass less than or equal to the sharp constant, which also seems new for general nonlinearity on the whole plane. For the existence of the ground state on bounded domains, we refer to \cite{FMR,A,AS}. One should be warned, however, that the situation on the whole plane is different from that on disks, unlike the higher dimensional Sobolev critical case, since here the concentration compactness has to be accompanied with a leak of $L^2$ norm to the spatial infinity. This will be discussed separately in a forthcoming paper \cite{TM}. 

It is worth noting that the scattering result in the focusing exponential case is actually easier to obtain than in the defocusing case, concerning the global Strichartz estimate.  This is because the (mass-modified) ground energy threshold implies that our solutions are in the subcritical regime for the Trudinger-Moser inequality.  Hence concentration of energy is a priori precluded, and so we do not need the concentration radius or the localized Strichartz estimate used in \cite{2DcNLKG} on the Trudinger-Moser threshold in the defocusing case. This is another striking difference from the power case, where the analysis for the focusing case essentially contains that for the defocusing case. 

\subsection{Main result}
To state the main results of this paper, we need to introduce some notation and assumptions for the variational setting and the nonlinear setting of the problem. 
\subsubsection{Variational setting} 
To specify our class of solutions, we need the static energy 
\EQ{ \label{def J}
 J(\fy) := \frac{1}{2}\int_{\R^d} [|\na \fy|^2 + |\fy|^2] dx - F(\fy), \pq F(\fy):=\int_{\R^d} f(\fy) dx,}
and its derivatives with respect to different scalings. In the critical/exponential cases, we also need the energy with a modified mass $c\ge 0$, 
\EQ{ \label{def Jc}
 J^{(c)}(\fy) = \frac{1}{2}\int_{\R^d} [|\na \fy|^2 + c|\fy|^2] dx - F(\fy).}
For any $\al,\be,\la\in\R$ and $\fy:\R^d\to\R$, we define the two-parameter rescaling family
\EQ{ \label{def scaling}
 \fy^\la_{\al,\be}(x) = e^{\al\la}\fy(e^{-\be\la}x),}
and the differential operator $\Dab$ acting on any functional $S:H^1(\R^d)\to\R$ by 
\EQ{ \label{def Dab}
 \Dab S(\fy) = \left.\frac{d}{d\la}\right|_{\la=0} S(\fy_{\al,\be}^\la).}
The scaling derivative of the static energy is denoted by 
\EQ{ \label{def K}
 \pt K_{\al,\be}(\fy) := \Dab J(\fy)
 \pr= \int_{\R^d} \left[\frac{2\al+(d-2)\be}{2}|\na \fy|^2 + \frac{2\al+d\be}{2}|\fy|^2 - \al \fy f'(\fy) - d\be f(\fy)\right] dx, 
 \pr K_{\al,\be}^{(c)}(\fy) := \Dab J^{(c)}(\fy).}
For each $(\al,\be)\in\R^2$ in the range
\EQ{ \label{range albe}
 \al \ge 0, \pq 2\al+d\be \ge 0, \pq 2\al +(d-2)\be \ge 0, \pq (\al,\be)\not=(0,0),}
we consider the constrained minimization problem 
\EQ{ \label{min J}
 \m_{\al,\be} = \inf\{J(\fy) \mid \fy \in H^1(\R^d),\ \fy\not=0,\ K_{\al,\be}(\fy)=0\}.}
We will prove that it is attained, (after a modification of the mass in some cases), provided that $(\al,\be)$ is in the above range \eqref{range albe}. 
The condition on $(\al,\be)$ is also necessary in general (see Proposition \ref{need range}). 

Our solutions start from the following subsets of the energy space
\EQ{ \label{def Kpm}
 \pt \K_{\al,\be}^+ = \{(u_0,u_1)\in H^1(\R^d)\times L^2(\R^d) \mid E(u_0,u_1)<\m_{\al,\be},\ K_{\al,\be}(u_0)\ge 0\},
 \pr \K_{\al,\be}^- = \{(u_0,u_1)\in H^1(\R^d)\times L^2(\R^d) \mid E(u_0,u_1)<\m_{\al,\be},\ K_{\al,\be}(u_0)< 0\}.}

\subsubsection{Nonlinear setting} 
For the nonlinearity $f$, we consider the following three cases: the $H^1$ subcritical ($d\ge 1$), the 2D exponential case, and the $H^1$ critical ($d\ge 3$) cases. First we assume that $f:\R\to\R$ is $C^2$ and 
\EQ{ \label{f sym} 
 f(0)=f'(0)=f''(0)=0.}

Secondly for the variational arguments, we need some monotonicity and convexity conditions. 
Let $D$ denote the linear operator defined by 
\EQ{ \label{def D}
 Df(u):=uf'(u).} 
We assume that $f$ satisfies for some $\e>0$, 
\EQ{ \label{f conv}
 (D-2_\star-\e) f \ge 0, \pq (D-2)(D-2_\star-\e) f \ge 0,}
which implies in particular that 
\EQ{
 D^2 f \ge (2_\star+\e)Df \ge (2_\star+\e)^2 f \ge 0.}

Finally we need regularity and growth conditions, which can differ for small $|u|$ and large $|u|$. Fix a cut-off function $\chi\in C_0^\I(\R)$ satisfying $\chi(r)=1$ for $|r|\le 1$ and $\chi(r)=0$ for $|r|\ge 2$, and denote 
\EQ{ \label{def chiR}
 \chi_R(x):=\chi(|x|/R),}
for arbitrary vector $x$ and $R>0$. Decompose the nonlinearity by 
\EQ{ \label{def fSL}
 f_S(u) := \chi_1(u) f(u), \pq f_L(u) = f(u)-f_S(u).}
We assume that for some $p_1>2_\star-2$ 
\EQ{ \label{f_S}
 \CAS{
 |f_S''(u)|\lec |u|^{p_1} & (d\le 4),\\ 
 |f_S''(u_1)-f_S''(u_2)| \lec |u_1-u_2|^{p_1} &(d\ge 5),}}
where we should choose $p_1<1$ for $d\ge 5$. 

For the behavior of $f$ for large $|u|$, we distinguish three cases: 

(1) $H^1$ subcritical case: We assume that for some $p_2< 2^\star-2$ 
\EQ{ \label{f sub}
 \CAS{|f_L''(u)|\lec |u|^{p_2} &(2 \le d \le 4)\\
 |f_L''(u_1)-f_L''(u_2)| \lec (|u_1|+|u_2|)^{p_2-1}|u_1-u_2| &(d\ge 5\text{ and }p_2\ge 1)\\
 |f_L''(u_1)-f_L''(u_2)| \lec |u_1-u_2|^{p_2} &(d\ge 5 \text{ and } p_2<1).}}
$p_2=2^\star-2$ will be allowed in some of the later arguments. There is no growth restriction for $d=1$. A typical example is 
\EQ{
 f(u) = \la_1 |u|^{q_1} + \cdots \la_k |u|^{q_k},}
where $\la_j>0$ and $2_\star<q_j<2^\star$ for all $j$, which satisfies \eqref{f sub} as well as \eqref{f sym}, \eqref{f conv} and \eqref{f_S}.

(2) $H^1$ critical case. We assume 
\EQ{ \label{f crit}
 d\ge 3, \pq   f(u) = |u|^{2^\star}/2^\star.}
In this case, we do not include lower powers in order to avoid their nontrivial effects in the variational characterization. The absence of lower powers will 
only be used in section \ref{sect:var}. In particular the Strichartz  spaces 
we use in section \ref{sect:str} can handle the sum of a critical power 
with a subcritical function. 
For  the variational  characterization,   the case where lower  powers
  are included will be treated in 
a forthcoming work.

(3) $2D$ exponential case: We assume that 
\EQ{ \label{f exp}
 d=2,\pq \pt\exists\ka_0\ge 0,\text{ s.t. } 
  \CAS{ \forall\ka>\ka_0,\pq \lim_{|u|\to\I} f_L''(u)e^{-\ka |u|^2}=0, \\ 
  \forall\ka<\ka_0,\pq \lim_{|u|\to\I} f_L(u)e^{-\ka|u|^2}=\I,}
 \pr\text{and if $\ka_0>0$ then}  \lim_{|u|\to\I} f_L(u)/Df_L(u) = 0.}
Then we define $\CTMS$ by 
\EQ{ \label{def C*F}
 \CTMS(F) = \sup\{2F(\fy)\|\fy\|_{L^2(\R^2)}^{-2} \mid 0\not=\fy\in H^1(\R^2),\ \ka_0\|\na\fy\|_{L^2(\R^2)}^2\le 4\pi\}.}
For example, all the conditions are satisfied by 
\EQ{
 f(u) = e^{\ka_0 |u|^2 } - 1 -\ka_0 |u|^2 - \frac{\ka_0^2}2 |u|^4}
and by 
\EQ{ \label{exp example}
 \pt f(u) = |u|^p e^{\ka_0 |u|^2+\ga|u|},}
where $p>4$, $\ka_0\ge 0$, and $\max(-\ga,0)\ll 1$ (depending on $\ka_0(p-4)$).  More specifically, it suffices to have for all $u\in[0,\I)$ that 
\EQ{ \label{exp cond}
 8\ka_0 u^2 + 3\ga u + 2(p-4) > 0,}
since, putting $g:=Df/f=2\ka_0u^2+\ga u+p$, we have $2_\star=4$ and 
\EQ{
 \pt (D-2)(D-4)f=[(g-4)^2+Dg+2(g-4)]f, 
 \pr Dg + 2(g-4) = 8\ka_0 u^2 + 3\ga u + 2(p-4) = 2[g(3u/2)-4]-u^2/2.}
In addition, one can easily observe that $\CTMS(F)=\I$ if $\ga\ge 0$ and $\CTMS(F)<\I$ if $\ga<0$, using Moser's sequence of functions for the former, and by the spherical symmetrization for the latter (cf. \cite{Moser,AT,Ruf}).\footnote{ 
Actually, the optimal (fastest) growth to have $\CTMS(F)<\I$ is given by 
\EQ{
 f(u) \sim e^{\ka_0|u|^2}/|u|^2 \pq (|u|\to\I),}
which will be shown in a forthcoming paper \cite{TM}. The results in this paper do not rely on this observation, though it seems to have its own interest.}

In short, our assumption on $f$ is that  
\EQ{ \label{asm f}
 \pt \eqref{f sym},\ \eqref{f conv},\ \eqref{f_S}, 
 \text{ and }[\eqref{f sub} \text{ or } \eqref{f crit} \text{ or } \eqref{f exp}].}
Then by Sobolev or Trudinger-Moser, we observe that $F$, $\Dab F$ and $\Dab^2 F$ are continuous functionals on $H^1(\R^d)$. 

Now we can state our main result. Denote the quadratic part of the energy (i.e. the linear energy) by 
\EQ{ \label{def E^Q}
 E^Q(u;t)=E^Q(u(t),\dot u(t)):=\int_{\R^d}\frac{|\dot u|^2+|\na u|^2+|u|^2}{2} dx.}
\begin{thm} \label{main thm} 
Assume \eqref{asm f} for $f$. Then for all $(\al,\be)$ in \eqref{range albe}, 
both $\m_{\al,\be}$ and $\K_{\al,\be}^\pm$ are independent of $(\al,\be)$. 
Moreover \eqref{NLKG} is locally wellposed in the energy space $H^1\times L^2$, and  
\begin{enumerate}
\item If $(u(0),\dot u(0))\in\K_{\al,\be}^-$, then $u$ extends neither for $t\to\I$ nor for $t\to-\I$ as the unique strong solution in $H^1\times L^2$. 
\item If $(u(0),\dot u(0))\in\K_{\al,\be}^+$, then $u$ scatters both in $t\to\pm\I$ in the energy space. In other words, $u$ is a global solution and there are $v_\pm$ satisfying 
\EQ{
 \ddot v_\pm - \De v_\pm + v_\pm = 0,\pq E^Q(u-v_\pm,\dot u-\dot v_\pm) \to 0 \pq (t\to\pm\I).}
\end{enumerate}
\end{thm} 
The dichotomy of global existence versus blow-up in the subcritical case was essentially given by Payne-Sattinger \cite{PS}, using $K_{1,0}$, on bounded domains. 
Hence our main contribution is the scattering part, and the parameter independence of $\K_{\al,\be}^\pm$. The corresponding result in the defocusing case (hence only the scattering) has been shown by \cite{Brenner,GV2} for the subcritical $f$ in three dimensions and higher, by \cite{2Dsubcrit} in lower dimensions, by \cite{3Dcrit} for the $H^1$ critical $f$, and by \cite{2DcNLKG} for the 2D exponential nonlinearity. The massless $H^1$ critical case (the other powers cannot be controlled by the massless energy) was solved by \cite{BS,BG} for the defocusing $f$ and by \cite{KM1} for the focusing nonlinearity. 

The parameter independence of $m_{\al,\be}$ seems to be known in the study of stability of standing waves, but the authors could not find an available result as general as the above one. See \cite{OT,Z} for partial results. 
We quote a recent paper \cite{JL} for a pure power nonlinearity, but unfortunately their range of $(\al,\be)$ was not correct (the condition $\al\ge 0$ was overlooked; its necessity is shown by Proposition \ref{need range}). 

The parameter independence of $\K^\pm_{\al,\be}$, on the other hand, does not seem to have got much attention from the stability analysis, but it is essential in our proof of the scattering, since the monotonicity is given for the blow-up and for the scattering in terms of different $K_{\al,\be}$, respectively $K_{1,0}$ and $K_{d,-2}$. 

Thanks to the parameter independence, we may write $\m=\m_{\al,\be}$ and $\K^\pm=\K_{\al,\be}^\pm$. We will also show the following important properties of the energy threshold. 
\begin{prop} \label{Gstate}
Under the assumptions of the above theorem, 
\begin{enumerate}
\item In the subcritical case \eqref{f sub}, the threshold energy $m$ is attained by some $Q\in H^1(\R^d)$, independent of $(\al,\be)$, solving the static equation 
\EQ{ \label{static NLKG} 
 -\De Q + Q = f'(Q),}
with the least energy $J(Q)=m$ among the solutions in $H^1(\R^d)$. In other words, $m$ is attained by the ground states. 
\item In the critical case \eqref{f crit}, there is no minimizer for \eqref{min J}, but we have 
\EQ{
 m = J^{(0)}(Q),}
for a static solution $Q\in\dot H^1(\R^d)$ of the massless equation 
\EQ{ \label{static NLW}
 -\De Q = f'(Q),}
with the least massless energy $J^{(0)}$. In other words, $m$ equals to the massless ground energy. 
\item In the exponential case \eqref{f exp}, let $c:=\min(1,\CTMS(F))$, where $\CTMS(F)$ is as in \eqref{def C*F}. Then we have 
\EQ{
 m = J^{(c)}(Q),}
for a static solution $Q\in H^1(\R^2)$ of the mass-modified equation 
\EQ{ \label{static mmKG}
 -\De Q + c Q = f'(Q),}
with the least energy $J^{(c)}(Q)$. Moreover we have 
\EQ{
 m \le 2\pi/\ka_0,}
where the equality holds if and only if $\CTMS(F)\le 1$, and $m=\m_{\al,\be}$ is attained in \eqref{min J} if and only if $\CTMS(F)\ge 1$.
\end{enumerate}
\end{prop}
Again this is well known in the subcritical case. 
Hence the main novelty is in the mass change in the critical/exponential cases. Note that the ground state $Q$ with a different mass $c\in[0,1)$ yields standing wave solutions $e^{\pm it\om}Q(x)$ with $1-\om^2=c$. 
But it is not a true obstruction for the scattering, because its dynamical energy is above $m$, although $m$ is the right threshold in the sense that for higher energy level $E>m$ the sets $\K^\pm$ are no longer separated from each other, that is, $\p\K^+\cap\p\K^-\not=\emptyset$.

\subsection{Some notation}
Here we recall some standard notation. $\F$ denotes the Fourier transform on $\R^d$, and 
\EQ{
 \LR{\na}:=\sqrt{1-\De}=\F^{-1}\sqrt{1+|\x|^2}\F.}
$L^p$, $H^s$, $B^s_{p,q}$ and $\dot B^s_{p,q}$ respectively denote the Lebesgue, Sobolev, inhomogeneous and homogeneous Besov spaces on $\R^d$. 
For later use we recall the most used functionals $K_{\alpha,\beta}$ and 
$H_{\al,\be}$: 
\EQ{ K_{1,0} (\fy) & = \int_{\R^d} \left[ |\na \fy|^2  +| \fy|^2  
 -  \fy f'(\fy) \right] dx , \\ 
 K_{0,1} (\fy) & = \int_{\R^d} \left[\frac{d-2}{2}|\na \fy|^2 + \frac{d}{2}|\fy|^2
  - d f(\fy)\right]  dx, \\ 
 K_{d,-2} (\fy) & = \int_{\R^d}  \left[{2}|\na \fy|^2  - d  (D-2 ) f(\fy) \right]  dx, }
\EQ{ H_{1,0} (\fy) & =\frac{1}{2}  \int_{\R^d} \left[  (D-2) f(\fy)\right] dx , \\ 
 H_{0,1} (\fy) & = \int_{\R^d} \left[\frac{1}{d}|\na \fy|^2 
\right]  dx, \\ 
 H_{d,-2} (\fy) & = \int_{\R^d}  \left[\frac12|  \fy|^2   +  
 \frac{d}4  (D-2_* ) f(\fy) \right]  dx.  }
We give a table of notation in Appendix B. 

\section{Variational characterizations} \label{sect:var}
In this section, we prove Proposition \ref{Gstate}. In particular we prove 
 the existence of ground states as constrained minimizers, the $(\al,\be)$-independence of the splittings, together with various estimates for solutions below the threshold by variational arguments, which will be used for the scattering and blowup. 

Throughout this section, we assume that $(\al,\be)$ is in the range \eqref{range albe}. 
For ease of presentation, we often omit $(\al,\be)$ from the subscript. 
We associate with it the following two numbers: 
\EQ{ \label{def mu}
 \bar\mu=\max(2\al+d\be,2\al+(d-2)\be), \pq \U\mu=\min(2\al+d\be,2\al+(d-2)\be),}
which come from the scaling exponents for $\dot H^1$ and $L^2$ in \eqref{def scaling}. 
Notice that in the range \eqref{range albe}, we have $\bar\mu>0$, $\U\mu\ge 0$, and that $\al=\U\mu=0$ if and only if $(d,\al)=(2,0)$, which will often be an exceptional case in the following arguments. 

We decompose $K_{\al,\be}=\Dab J$ into the quadratic and the nonlinear parts:
\EQ{ \label{def KQN}
 K_{\al,\be}=K^Q_{\al,\be}+K^N_{\al,\be}, \pq K^Q_{\al,\be}(\fy)=\Dab \|\fy\|_{H^1}^2/2, \pq K^N_{\al,\be}(\fy)=-\Dab F(\fy).} 
Then $K^Q_{\al,\be}(\fy_{\al,\be}^\la)$ is non-negative and non-decreasing with respect to $\la\in\R$, and 
\EQ{ \label{KL conc}
 \lim_{\la\to-\I} K^Q_{\al,\be}(\fy_{\al,\be}^\la) = 0,}
from its explicit form. 

\subsection{Energy landscape in various scales} 
First we investigate how $J$ and its derivatives behave with respect to the scaling $\fy_{\al,\be}^\la$, in order to get $\m_{\al,\be}$ as a minimax value. 
The results  of  this subsection are essentially known, at least under more restrictions on the nonlinearity and $(\al,\be)$. 

We start from the origin of the energy space. 
\begin{lem}[Positivity of $K$ near $0$] \label{K>0 near 0}
Assume that $f$ satisfies \eqref{asm f}, and that $(\al,\be)$ satisfies \eqref{range albe} and $(d,\al)\not=(2,0)$. 
Then for any bounded sequence $\fy_n\in H^1(\R^d)\setminus\{0\}$ such that $K_{\al,\be}^Q(\fy_n)\to 0$, we have for large $n$, 
\EQ{
 K_{\al,\be}(\fy_n)>0.}
\end{lem} 
Note that if $(d,\al)=(2,0)$ then the conclusion is false, since in that case $K^Q(\fy^\la)=e^{d\be\la}K^Q(\fy)\to 0$ as $\la\to-\I$, but $K(\fy^\la)=e^{d\be\la}K(\fy)$ can be negative. 
\begin{proof}
First we consider the $H^1$ subcritical/critical cases. If $d\ge 2$ then
\EQ{
 |Df(\fy)|+|f(\fy)| \lec |\fy|^{p_1+2} + |\fy|^{p_2+2},}
for some $2_\star<p_1+2<p_2+2\le 2^\star$, hence by the Gagliardo-Nirenberg inequality
\EQ{ \label{GN} 
 \|\fy\|_{L^q_x}^q \lec \|\na\fy\|_{L^2_x}^{d(q/2-1)} \|\fy\|_{L^2_x}^{d-q(d-2)/2},\pq (2\le q\le 2^\star)}
we obtain
\EQ{ \label{GN bd F}
 |F(\fy)|+|\L F(\fy)| \lec \sum_{q=p_1+2,\ p_2+2} \|\na\fy\|_{L^2_x}^{d(q/2-1)} \|\fy\|_{L^2_x}^{d-q(d-2)/2}.}
If $d=1$ then we can dispose of $f_L$ by Sobolev $H^1(\R)\subset L^\I(\R)$. Then we get 
\EQ{ \label{GN bd 1D}
 |F(\fy)|+|\L F(\fy)| \lec \|\na\fy\|_{L^2_x}^{p_1/2+1} \|\fy\|_{L^2_x}^{p_1/2+1}C(\|\fy\|_{H^1}),}
for some function $C$ determined by $f_L$. 

Hence if $2\al+(d-2)\be>0$ then for any $d$ we have 
\EQ{
 |K^N(\fy)| = o(\|\na\fy\|_{L^2_x}^2) = o(K^Q(\fy)).}
Under the assumption, $2\al+(d-2)\be=0$ is possible only for $d=1$, then using \eqref{GN bd 1D}, 
\EQ{
 |K^N(\fy)| = o(\|\fy\|_{L^2_x}^2) = o(K^Q(\fy)).}

Finally we consider the 2D exponential case \eqref{f exp}. Then we have 
\EQ{
 |Df(\fy)| + |f(\fy)| \lec |\fy|^p (e^{\ka|\fy|^2}-1),}
for some $p>2$ and any $\ka>\ka_0$. 
Since $\al>0$, we have $K^Q(\fy_n)\gec \|\na \fy_n\|_{L^2}^2 \to 0$, so it suffices to consider $\fy\in H^1$ satisfying for some $q>1$ satisfying $(4-p)q<2$, 
\EQ{
 q\ka\|\na\fy\|_{L^2}^2 \le 2\pi.} 
Let $q'=q/(q-1)$ be the H\"older conjugate. 
Then by H\"older, Gagliardo-Nirenberg \eqref{GN} and the Trudinger-Moser inequality:
\EQ{ \label{TM}
  \|\na\fy\|_{L^2(\R^2)}< \sqrt{4\pi} \implies \int_{\R^2} (e^{|\fy|^2}-1) dx \lec \frac{\|\fy\|_{L^2(\R^2)}^2}{4\pi-\|\na\fy\|_{L^2(\R^2)}^2},}
we obtain
\EQ{ \label{TM bd F}
 |\L F(\fy)|+|F(\fy)| \pt\lec \|\fy\|_{L^{pq'}}^{p} \|e^{q\ka|\fy|^2}-1\|_{L^1}^{1/q},
 \pr\lec \|\fy\|_{L^2}^{2/q'} \|\na\fy\|_{L^2}^{p-2/q'}\left[\frac{\|\fy\|_{L^2}^2}{4\pi-q\ka\|\na\fy\|_{L^2}^2}\right]^{1/q}
 \pr\lec \|\fy\|_{L^2}^2 \|\na\fy\|_{L^2}^{p-2/q'}.}
Since $p-2/q'>2$ by the choice of $q$, we get 
\EQ{
 |K^N(\fy)|=o(\|\na\fy\|_{L^2}^2)=o(K^Q(\fy)).} 

Thus in all cases $K(\fy)\sim K^Q(\fy)>0$ when $0<K^Q(\fy)\ll 1$. 
\end{proof}

The following inequalities describe the graph of $J$, and will play the central role in the succeeding arguments. 
\begin{lem}[Mountain-pass structure] \label{mount}
Assume that $f$ satisfies \eqref{asm f} and $(\al,\be)$ satisfies \eqref{range albe}. Then for any $\fy\in H^1(\R^d)$ we have 
\EQ{ \label{K F mono} 
 \pt (\Dab-\bar\mu)\|\fy\|_{H^1}^2 \le -2|\be|\min(\|\fy\|_{L^2}^2,\|\na\fy\|_{L^2}^2),
 \pr (\Dab-\bar\mu)F(\fy) \ge \al\e F(\fy), }
where $\e>0$ is given in \eqref{f conv}. Hence 
\EQ{ \label{J mono}
 (\bar\mu-\Dab)J(\fy) \ge \al\e F(\fy)  + |\be|\min(\|\fy\|_{L^2}^2,\|\na\fy\|_{L^2}^2).}
Moreover we have
\EQ{ \label{J conv}
 -(\Dab-\bar\mu)(\Dab-\U\mu)J(\fy) \pt= (\Dab-\bar\mu)(\Dab-\U\mu)F(\fy) 
 \pr\ge \frac{2\al\e}{d+1}\Dab F(\fy) \ge \frac{2\al\e\bar\mu}{d+1}F(\fy).}
\end{lem}
\begin{proof}
First we observe that 
\EQ{
 \pt (\L-2\al-(d-2)\be)\|\na\fy\|_{L^2_x}^2 = 0,
 \pq (\L-2\al-d\be)\|\fy\|_{L^2_x}^2 = 0,}
and for any functional $S$ of the form $S(\fy)=\int_{\R^d}s(\fy)dx$, 
\EQ{ \label{Dab & D}
 \Dab S(\fy) = \int_{\R^d}[(\al D+\be d) s](\fy) dx,}
where $Df(\fy)=\fy f'(\fy)$ as defined in \eqref{def D}. 
Using them, we obtain 
\EQ{
 \pt (\L-\bar\mu)\|\fy\|_{H^1}^2 = -2 |\be| \times \CAS{\|\na\fy\|_{L^2}^2 &(\be\ge 0) \\ \|\fy\|_{L^2}^2 &(\be\le 0)},}
and also 
\EQ{ \label{LF expand}
 \L F(\fy) \pt= \int [\al(D-2)+2\al+d\be] f(\fy) dx
 \pr= \int [(\al D - 2\al + 2\be) + 2\al+(d-2)\be]f(\fy) dx.}
Since 
\EQ{ \label{D2 F}
 \al D - 2 \al + 2\be = \al(D-2_\star) + \frac{2}{d}(2\al+d\be),}
using \eqref{f conv}, we obtain
\EQ{ \label{low bd DF}
 \L F \ge (\bar\mu + \al\e) F.}
 
Using the above computations, we have 
\EQ{
 -(\L-\bar\mu)(\L-\U\mu)J(\fy) \pt= (\L-\bar\mu)(\L-\U\mu) F(\fy) 
 \pr= \al\int (\al D -2\al+2\be)(D-2) f(\fy)dx
 \pr\ge \al\e \int\left[\al(D-2)+\frac{2}{d}(2\al+d\be)\right]f(\fy)dx
 \pr\ge \frac{2}{d+1}\al\e\L F(\fy) \ge \frac{2\al\e\bar\mu}{d+1}F(\fy),}
where we used \eqref{D2 F} and \eqref{f conv} in the first inequality, 
$\min(1,2/d)\ge 2/(d+1)$ in the second, and \eqref{low bd DF} in the last. 
\end{proof}

Using the above inequalities, we can replace the minimized quantity in \eqref{min J} with a positive definite one, while extending the minimizing region from ``the mountain ridge'' to ``the mountain flank''. Let 
\EQ{ \label{def H}
 H_{\al,\be}:=(1-\Dab/\bar\mu)J.}
Then the above lemma implies that $H_{\al,\be}>0$ and 
\EQ{ \label{DH>0}
 \Dab H_{\al,\be} \pt= -(\L-\U\mu)(\L-\bar\mu)J/\bar\mu + \U\mu(1-\L/\bar\mu)J
 \pr\ge \frac{2\al\e}{d+1} F + \U\mu H_{\al,\be} \ge 0.}
We can rewrite the minimization problem \eqref{min J} by using $H$:
\begin{lem}[Minimization of $H$]
Assume that $f$ satisfies \eqref{asm f} and $(\al,\be)$ satisfies \eqref{range albe}. Then $\m_{\al,\be}$ in \eqref{min J} equals 
\EQ{ \label{min H}
 \m_{\al,\be} =  \inf\{H_{\al,\be}(\fy) \mid \fy\in H^1(\R^d),\ \fy\not=0,\ K_{\al,\be}(\fy) \le 0\}.}
\end{lem}
\begin{proof}
Let $m'$ denote the right hand side of \eqref{min H}. 
Then $m\ge m'$ is trivial because $J=H$ if $K=0$, so it suffices to show $m\le m'$. Take $\fy\in H^1$ such that $K(\fy)<0$. 

If $(d,\al)\not=(2,0)$, then from Lemma \ref{K>0 near 0} together with \eqref{KL conc}, we deduce that 
\EQ{ \label{K0 by scal}
 (d,\al)\not=(2,0),\ K(\fy)<0 \implies \exists\la<0,\ K(\fy^\la)=0,\ H(\fy^\la)\le H(\fy),}
where the latter inequality follows from \eqref{DH>0} since $H(\fy^\la)$ is nondecreasing in $\la$. 
Hence $m\le m'$. 

If $(d,\al)=(2,0)$, then we use another scaling $\nu u$ with $\nu\in(0,1)$. 
We have $K^Q(\nu\fy)=\nu^2K^Q(\fy)$ and $|K^N(\nu\fy)|=o(\nu^4)$ by \eqref{GN bd F} or \eqref{TM bd F}. 
Hence $K(\nu\fy)>0$ for small $\nu>0$, and so we deduce  
\EQ{ \label{K0 by amp}
 (d,\al)=(2,0),\ K(\fy)<0 \implies \exists\nu\in(0,1),\ K(\nu\fy)=0,\ H(\nu\fy)\le H(\fy),} 
where the inequality follows from $H(\fy)= \|\na\fy\|_{L^2_x}^2/2$ in this case. Hence $m\le m'$.

\end{proof}

\subsection{Ground state as common minimizer}
Now we can prove the parameter independence of $\m_{\al,\be}$ via its characterization by the ground states. 
First we consider the $H^1$ subcritical case. 
\begin{lem}[Ground state in the subcritical case] \label{Q sub} 
Assume that $f$ satisfies \eqref{asm f} and \eqref{f sub}, and that $(\al,\be)$ satisfies \eqref{range albe}. Then $\m_{\al,\be}$ in \eqref{min J} is positive and independent of $(\al,\be)$. Moreover $\m_{\al,\be}=J(Q)$ for some $Q\in H^1(\R^d)$ solving the static NLKG \eqref{static NLKG} with the minimal $J(Q)$ among the solutions in $H^1(\R^d)$. 
\end{lem}
\begin{proof}
Let $\fy_n\in H^1$ be a minimizing sequence for \eqref{min H}, namely $K(\fy_n)\le 0$, $\fy_n\not=0$ and $H(\fy_n)\searrow m$. 

First we consider the case $(d,\al)\not=(2,0)$. Let $\fy_n^*$ be the Schwartz symmetrization of $\fy_n$, i.e.~the radial decreasing rearrangement. 
Since the symmetrization preserves the nonlinear parts and does not increase the $\dot H^1$ part, we have $\fy_n^*\not=0$, $K(\fy_n^*)\le K(\fy_n)\le 0$ and $H(\fy_n)\ge H(\fy_n^*)\to m$. 
Then using \eqref{K0 by scal}, we may replace it by symmetric $\psi_n\in H^1$ such that 
\EQ{ \label{def psin}
 \psi_n\not=0,\pq K(\psi_n)=0, \pq J(\psi_n)=H(\psi_n) \to m.}

If $\al>0$, then \eqref{J mono} implies 
\EQ{
 (\bar\mu + \al\e)J(\psi_n) \ge \al\e\|\psi_n\|_{H^1}^2/2,}
hence $\psi_n$ is bounded in $H^1$. 

If $\al=0$ (and $d>2$), then $H(\psi_n)=\|\na\psi_n\|_{L^2_x}^2/d$ is bounded, 
and if $\|\psi_n\|_{L^2}\to \I$, then by \eqref{GN bd F} 
\EQ{
 d\be \|\psi_n\|_{L^2}^2 \le 2K^Q(\psi_n) = -2K^N(\psi_n) \le o(\|\psi_n\|_{L^2}^{d-2_\star(d-2)/2}),}
but since $d-2_\star(d-2)/2<2$, this is a contradiction. Hence $\psi_n$ is bounded in $H^1$. 

Since $\psi_n$ is bounded in $H^1$, after replacing with some appropriate subsequence, it converges to some $\psi$ weakly in $H^1$. 
By the radial symmetry, it also converges strongly in $L^p$ for all $2<p<2^\star$. Then in the subcritical case \eqref{f sub}, the nonlinear parts converge, and so $K(\psi)\le 0$ and $H(\psi)\le m$. 

If $\psi=0$, then $K(\psi_n)=0$ implies that $K^Q(\psi_n)=-K^N(\psi_n)\to 0$, and by Lemma \ref{K>0 near 0} we have $K(\psi_n)>0$ for large $n$, a contradiction. Hence $\psi\not=0$. 

By \eqref{K0 by scal}, we may replace $\psi$ by its rescaling, so that $K(\psi)=0$, $J(\psi)=H(\psi)\le m$ and $\psi\not=0$. 
Then $\psi$ is a minimizer and $m=H(\psi)>0$. 

Since $\psi$ is a minimizer for \eqref{min J}, there is a Lagrange multiplier $\y\in\R$ such that 
\EQ{ \label{Lagrange}
 J'(\psi) = \y K'(\psi).}
Then denoting $\L\psi=\p_\la \psi^\la_{\al,\be}|_{\la=0}$, we get 
\EQ{
 0 = K(\psi) = \L J(\psi) = \LR{J'(\psi)|\L \psi} = \y\LR{K'(\psi)|\L\psi}
 = \y \L^2 J(\psi).}
By \eqref{J conv} and $\L J(\psi)=0$, we have  
\EQ{
 \L^2 J(\psi) \le -\bar\mu\U\mu J(\psi) - \frac{2\al\e\bar\mu}{d+1}F(\psi)<0,}
since $\U\mu>0$ or $\al>0$. 
 
Therefore $\y=0$ and $\psi$ is a solution to \eqref{static NLKG}. The minimality of $J(\psi)$ among the solutions is clear from \eqref{min J}, since every solution $Q$ in $H^1$ of \eqref{static NLKG} satisfies $K(Q)=\LR{J'(Q)|\L Q}=0$. 
This implies that $\m_{\al,\be}$ is independent of $(\al,\be)$. 

In the exceptional case $(d,\al)=(2,0)$, the above argument needs considerable modifications, due to the scaling invariance 
\EQ{
 H(\fy)=\|\na\fy\|_{L^2}^2/2 = H(\fy^\la), \pq K(\fy^\la)=e^{d\be\la}K(\fy).} 
First, we should use \eqref{K0 by amp} instead of \eqref{K0 by scal} to get $\psi_n$ satisfying \eqref{def psin}. 
Next, the invariance breaks the $H^1$ boundedness of $\psi_n$. 
But we are free to replace each $\psi_n$ by its rescaling so that $\|\psi_n\|_{L^2}=1$, without losing its properties \eqref{def psin}. 
Then $1=\|\psi_n\|_{L^2}^2=2F(\psi_n)\to 2F(\psi)$, which clearly implies that the limit $\psi\not=0$. 
By \eqref{K0 by amp}, we may replace $\psi$ by its constant multiple, so that $K(\psi)=0$, $J(\psi)=H(\psi)\le m$ and $\psi\not=0$. 
Then $\psi$ is a minimizer and $m=H(\psi)>0$. 

Finally, the invariance gives us $\L^2J(\psi)=0$ and the Lagrange multiplier $\y$ may be nonzero. The equation \eqref{Lagrange} is written in this case 
\EQ{
 -\De \psi = (\y d \be -1)[\psi-f'(\psi)].}
Since $\LR{-\De\psi|\psi}_{L^2_x}>0$ and 
\EQ{
 \LR{\psi-f'(\psi)|\psi}_{L^2_x} = K_{0,2/d}(\psi)-\int(D-2)f(\psi)dx < 0,}
we have $(\y d \be -1)<0$. Hence there exists $\la>0$ such that $\psi^\la$ solves the static NLKG \eqref{static NLKG}, and it is also a minimizer. 
\end{proof}

\subsection{$H^1$ critical case; massless threshold} 
In the $H^1$ critical case \eqref{f crit}, we still have the $(\al,\be)$ independence, but $\m_{\al,\be}$ is equal to the massless energy of the massless ground state. This is a consequence of the invariance of the massless energy with respect to the $\dot H^1$ scaling. 
\begin{lem}[Ground state in $H^1$ critical case] \label{Q crit}
Assume that $f$ satisfies \eqref{f crit}, and that $(\al,\be)$ satisfies \eqref{range albe}. Then $\m_{\al,\be}$ in \eqref{min J} is positive and independent of $(\al,\be)$. Moreover $\m_{\al,\be} = J^{(0)}(Q)$ for some $Q\in\dot H^1(\R^d)$ solving the static massless NLKG \eqref{static NLW}, with the minimal $J^{(0)}(Q)$ among the solutions in $\dot H^1(\R^d)$. 
\end{lem}
\begin{proof}
Let $H^w$ and $K^w$ be the massless versions of $H$ and $K$, respectively.
Then 
\EQ{ \label{massless H/K}
 m = m^w:= \inf\{H^w(\fy) \mid \fy\in H^1,\ K^w(\fy)< 0.\}.}
Indeed, comparing the above with \eqref{min H}, we easily get $m\ge m^w$ from $H^w\le H$ and $K^w<K$ if $2\al+d\be>0$. 
If $2\al+d\be=0$, then we may replace $K\le 0$ in \eqref{min H} by $K<0$, because for any nonzero $\fy\in H^1$ satisfying $K(\fy)\le 0$, we have by \eqref{J conv}
\EQ{
 \L K(\fy) \le \bar\mu K(\fy)-\frac{2\al\e\bar\mu}{d+1}F(\fy)<0,}
which implies that $K(\fy^\la)<0$ for all $\la>0$, and so the set $K<0$ is dense in the minimization set of \eqref{min H}. Hence $m\ge m^w$ in this case too. 

To prove $m\le m^w$, let 
\EQ{
 \fy^\nu=\fy^\nu_{d/2-1,-1}} 
denote the $\dot H^1$ invariant scaling. Then $K(\fy^\nu)\to K^w(\fy)$ and $H(\fy^\nu)\to H^w(\fy)$ as $\nu\to\I$. Hence if $K^w(\fy)<0$ then $K(\fy^\nu)<0$ for large $\nu$, and so $m\le m^w$. 

Due to the $\dot H^1$ scale invariance, $K^w_{\al,\be}$ for all $(\al,\be)$ are constant multiples of the same functional, and $H^w$ is independent of $(\al,\be)$, so is the minimization for $m^w$. In fact we have 
\EQ{
 m^w = \inf\{\|\na\fy\|_{L^2}^2/d \mid \fy\in H^1,\ \|\na\fy\|_{L^2}^2 < \|\fy\|_{L^{2^\star}}^{2^\star}\}.}
By the homogeneity and the scaling $\fy\mapsto\nu\fy$, it is equal to 
\EQ{
 \inf_{0\not=\fy\in H^1}\frac{1}{d}\|\na\fy\|_{L^2}^2\left[\frac{\|\na\fy\|_{L^2}^2}{\|\fy\|_{L^{2^\star}}^{2^\star}}\right]^{\frac{d-2}{2}} 
 = \inf_{0\not=\fy\in H^1}\frac{1}{d}\left[\frac{\|\na\fy\|_{L^2}}{\|\fy\|_{L^{2^\star}}}\right]^d
 =\frac{(C_S^\star)^{-d}}{d},}
where $C_S^\star$ denotes the best constant for the Sobolev inequality
\EQ{
 \|\fy\|_{L^{2^\star}} \le C_S^\star \|\na\fy\|_{L^2},}
which is well known to be attained by the following explicit $Q\in\dot H^1$ 
\EQ{
 Q(x) = \left[1+\frac{|x|^2}{d(d-2)}\right]^{-\frac{d-2}{2}},}
which solves \eqref{static NLW}.
\end{proof}

\subsection{Exponential case; mass-modified threshold} 
In the 2D exponential case \eqref{f exp}, the conclusion is somewhat intermediate between the above two cases. 
If $\CTMS(F)\ge 1$ then $m_{\al,\be}$ is achieved by a ground state, but if $\CTMS(F)<1$ then we can still see  $m_{\al,\be}$ as the energy of a ground state to an equation \eqref{static mmKG} where  the mass is changed to $c=\min(1,\CTMS(F))\in(0,1)$. 
\begin{lem}[Ground state in the exponential case] \label{Q exp}
Assume that $f$ satisfies \eqref{asm f} and \eqref{f exp}, and that $(\al,\be)$ satisfies \eqref{range albe}. 
Then $\m_{\al,\be}$ in \eqref{min J} is independent of $(\al,\be)$ and $0<m_{\al,\be}\le 2\pi/\ka_0$, where the second inequality is strict if and only if $\CTMS(F)>1$. Moreover $\m_{\al,\be}=J^{(c)}(Q)$ with $c=\min(1,\CTMS(F))$ for some $Q\in H^1(\R^2)$ solving the modified static NLKG \eqref{static mmKG} with the minimal $J^{(c)}(Q)$ among the solutions in $H^1(\R^2)$. 
\end{lem}
For the proof, we prepare some notations and lemmas. For any functional $G$ on $H^1(\R^2)$ and any $A> 0$, we introduce the Trudinger-Moser ratio 
\EQ{ \label{def CTMA}
 \CTM^A(G):= \sup\{ 2G(\fy)\|\fy\|_{L^2}^{-2} \mid 0\not=\fy\in H^1(\R^2),\ \|\na\fy\|_{L^2}\le A \},}
the Trudinger-Moser threshold on the $\dot H^1$ norm:
\EQ{ \label{def mm}
 \mm(G) := \sup \{A>0 \mid \CTM^A(G)<\I \},}
and the ratio on the threshold: 
\EQ{ \label{def CTMS}
 \CTMS(G) := \CTM^{\mm(G)}(G).}
The growth condition \eqref{f exp} together with \eqref{f conv} implies 
\EQ{
 \mm(\L_{\al,\be} F)=\mm(F)=\sqrt{4\pi/\ka_0}}
for any $(\al,\be)$ satisfying \eqref{range albe}, by the Trudinger-Moser inequality \eqref{TM}. Hence the above definition of $\CTMS$ is consistent with \eqref{def C*F}. 

For any functional $G$ of the form $G(\fy)=\int g(\fy)dx$, and for any sequence $(\fy_n)_{n \in \N}\in H^1(\R^2)^\N$, we define its concentration (at $x=0$) 
 $conc.G((\fy_n)_{n \in \N} )$ by 
\EQ{ \label{def conc}
 conc.G((\fy_n)_{n \in \N} ) := \limsup_{\e\to+0} \limsup_{n\to\I} \int_{|x|<\e}g(\fy_n)dx.}
We will use the following compactness by dominated convergence. 
\begin{lem} \label{TM compact}
Let $g,h:\R\to\R$ be continuous functions satisfying 
\EQ{ \label{LS dominance}
 \lim_{ u \to \pm \I}  \frac{|g(u)|}{h(u)}=0, \pq 
 \lim_{ u \to 0} \frac{|g(u)|}{|u|^2}=0.}
Let $\fy_n$ be a sequence of radial functions, weakly convergent to $\fy$ in $H^1(\R^2)$ such that $\{h(\fy_n)\}_n$ is bounded in $L^1(\R^2)$. Then $g(\fy_n)\to g(\fy)$ strongly in $L^1(\R^2)$. 
\end{lem}
\begin{proof}
By assumption \eqref{LS dominance}, for any $\e>0$ there exist $\de>0$ such that 
\EQ{
 |u|>1/(2\de)  \text{ or } |u|<2\de \implies |g(u)|<\e(h(u)+|u|^2). }
Then we have  
\EQ{ \label{fyn LS}
 \int_{|\fy_n|>1/(2\de) \text{ or } |\fy_n|<2 \de} |g(\fy_n)| dx \lec \e \int h(\fy_n)+|\fy_n|^2 dx \lec \e.}
The radial Sobolev inequality $\|r^{1/2}\fy_n\|_{L^\I}\lec\|\fy_n\|_{L^2}^{1/2}\|\na\fy_n\|_{L^2}^{1/2}$ implies that $\fy_n(x)$ are uniformly small for large $x$. Then the weak convergence together with 
\EQ{
 \fy_n(R_1)-\fy_n(R_2)=\int_{R_1}^{R_2}\p_r\fy_n(r)dr}
implies that $\fy_n(x)\to\fy(x)$ for $x\not=0$. Then Fatou's lemma implies 
\EQ{ \label{fy LS}
  \int_{|\fy|>1/(2\de) \text{ or } |\fy|<2\de} |g(\fy)| dx \lec \e,}
and the dominated convergence theorem implies 
\EQ{ \label{conv med}
 \|g^{(\de)}(\fy_n)-g^{(\de)}(\fy)\|_{L^1} \to 0,\pq(n\to\I)}
where $g^{(\de)}$ is defined by $g^{(\de)}(u) = (1-\chi_\de(u))\chi_{1/\de}(u)g(u)$ using the cut-off defined in \eqref{def chiR}. 
Combining \eqref{fyn LS}, \eqref{fy LS} and \eqref{conv med}, we deduce the desired convergence. 
\end{proof}

\begin{proof}[Proof of Lemma \ref{Q exp}]
We start with the exceptional case $(d,\al)=(2,0)$. 
First, let $A>0$ and assume $\CTM^A(F)>1$. Then there exists $0\not=\fy\in H^1$ such that $\|\na\fy\|_{L^2}\le A$ and $F(\fy)>\|\fy\|_{L^2}^2/2$. 
For small $\e>0$ we have $K_{0,1}((1-\e)\fy)<0$, 
and hence $m_{0,1}\le\|\na(1-\e)\fy_n\|_{L^2}^2/2< A^2/2$. 
Hence $m_{0,1}\le \mm(F)^2/2$. 

Consider the case $\CTMS(F)>1$. Then by choosing $A=\mm(F)$ in the above argument, we get $m_{0,1}<\mm(F)^2/2$. Now we take a minimizing sequence for $m_{0,1}$. 
By the Schwartz symmetrization and rescalings as in the proof of Lemma \ref{Q sub} for $(d,\al)=(2,0)$, we get a sequence of radial functions $\psi_n\in H^1$ such that 
\EQ{
 \|\psi_n\|_{L^2}=1, \pq H_{0,1}(\psi_n)\to m_{0,1}, \pq K_{0,1}(\psi_n)=1-2F(\psi_n)=0,}
and $\psi_n\to\psi$ in $H^1$. 
Because of $m_{0,1}<\mm(F)^2/2$, we can choose some $\ka\in(\ka_0,2\pi/m_{0,1})$, so that $e^{\ka|\psi_n|^2}-1$ is bounded in $L^1$ by the Trudinger-Moser inequality \eqref{TM}. Then we can use Lemma \ref{TM compact} with $\fy_n:=\psi_n$, $g:=f$ and $h(u):=e^{\ka|u|^2}-1$, which implies $F(\psi_n)\to F(\psi)$. Hence $\psi$ attains $m_{0,1}$. 
After appropriate rescalings, we obtain a ground state $Q$, as in the proof of Lemma \ref{Q sub}.

Next consider the case $\CTMS(F)\le 1$. Then for any $\psi\in H^1$ satisfying $\|\na\psi\|_{L^2}\le \mm(F)$ we have $K_{0,1}(\psi)\ge 0$. Hence 
\EQ{
 m_{0,1}=\inf\{\|\na\fy\|_{L^2}^2/2\mid K_{0,1}(\fy)<0\} \ge \mm(F)^2/2,}
and so $m_{0,1}=\mm(F)^2/2$. Now we show that there exists $\fy\in H^1$ satisfying 
\EQ{
 \|\na\fy\|_{L^2} = \mm(F), \pq F(\fy)=\CTMS(F)/2, \pq \|\fy\|_{L^2}=1.}
After rescaling this $\fy$, we obtain a ground state $Q$. However, due to the 
criticality, we have to approximate the problem by a subcritical one, namely 
we first prove the existence of 
$\fy_n \in H^1$ satisfying 
\EQ{
 \|\na\fy_n \|_{L^2} \leq  \mm(F)-\frac1n, \pq F(\fy_n)=c_n/2, \pq \|\fy_n\|_{L^2}=1}
where   $c_n:=\CTM^{\mm(F)-\frac1n}(F)$, then $0<c_n\nearrow \CTMS(F)\le 1$. 
Fix $n\gg 1$ and let $\fy^k\in H^1(\R^2)$ be a maximizing sequence for $c_n$ (see 
\eqref{def CTMA}): 
\EQ{
 \|\na\fy^k\|_{L^2} \le \mm(F)-\frac1n, \pq F(\fy^k)\nearrow c_n/2, \pq \|\fy^k\|_{L^2}=1,}
where the $L^2$ norm is normalized by the rescaling $\fy_{0,1}^\la$. 
The Schwartz symmetrization enables us to assume that $\fy^k$ are radial functions, and convergent to some $\fy_n$ weakly in $H^1$, by extracting a subsequence. 
Moreover, we have $F(\fy^k)\to F(\fy_n)=c_n/2$, by Lemma \ref{TM compact} with $g:=f$ and $h=e^{\ka|u|^2}-1$ for some $\ka\in(\ka_0,4\pi/(\mm(F)-1/n)^2)$. 

Thus $\fy_n$ is a maximizer, which implies that  $ \|\fy_n\|_{L^2}=1 $ and  
\EQ{ 
 -\y \De\fy_n = f'(\fy_n)-c_n\fy_n,}
for a Lagrange multiplier $\y(n)\in\R$. Multiplying it with $\fy_n$, we obtain  
\EQ{
 \y\|\na\fy_n\|_{L^2}^2 = \int Df(\fy_n)dx - c_n\|\fy_n\|_{L^2}^2
 =  \int (D-2)f(\fy_n)dx >0,}
since  $(D-2)f>0$. 
Hence $\y>0$, and so $Q_n(x):=\fy_n(\y^{1/2}x)\in H^1$ satisfies 
\EQ{
 \|\na Q_n\|_{L^2} \le \mm(F) - \frac1n, \pq -\De Q_n + c_n Q_n = f'(Q_n).}
 
Now consider the limit $n\to\I$. The equation for $Q_n$ implies that $0=K_{0,1}^{(c_n)}(Q_n)=K_{1,-1}^{(c_n)}(Q_n)$, that is 
\EQ{
 c_n\|Q_n\|_{L^2}^2=2F(Q_n), \pq \|\na Q_n\|_{L^2}^2=2\int(D-2)f(Q_n)dx \ge 4F(Q_n),}
where the last inequality follows from $(D-4)f\ge 0$. 
Since $\|\na Q_n\|_{L^2}$ is bounded and $c_n$ is positive non-decreasing, we deduce that $\|Q_n\|_{L^2}$ and $\int Df(Q_n)dx$ are bounded as $n\to\I$. 
Hence we may extract a subsequence so that $Q_n$ converges to some $Q$ weakly in $H^1$, and then apply Lemma \ref{TM compact} with $\fy_n:=Q_n$, $g=f'$ and $h:=Df$. 
Then $f'(Q_n)\to f'(Q)$ strongly in $L^1$, and so $Q$ solves
\EQ{ \label{stmmKG}
 -\De Q + c Q = f'(Q), \pq c:=\CTMS(F).}
This implies that 
\EQ{
 K_{0,1}^{(c)}(Q)=\LR{{J^{(c)}}'(Q)|\L_{0,1}Q}=0,}
namely $2F(Q)=c\|Q\|_{L^2}^2$. Hence $Q$ is a maximizer for $\CTM^{\mm(F)}(F)$ with a non-zero Lagrange multiplier, which implies that $\|\na Q\|_{L^2}=\mm(F)$. 
Thus $J^{(c)}(Q)=\mm(F)^2/2$ is unique for any solution $Q$ of \eqref{stmmKG}. 

Next we consider $m_{\al,\be}$ with $\al>0$. 
If $m_{0,1}<\mm(F)^2/2$, then there exists a ground state $Q$, which satisfies $K_{\al,\be}(Q)=0$ for all $(\al,\be)$. 
Hence $m_{\al,\be}\le J(Q)=m_{0,1}$. 

Otherwise, $m_{0,1}=\mm(F)^2/2=\mm(\L F)^2/2$. 
For any $A>\mm(\L F)$, there exists a sequence $\fy_n\in H^1$ satisfying 
\EQ{
 \|\na\fy_n\|_{L^2} \le A, \pq \|\fy_n\|_{L^2} \to 0, \pq \L F(\fy_n)\to \I.}
Since $K(\fy) = \al\|\na\fy\|_{L^2}^2 + (\al+\be)\|\fy\|_{L^2}^2 - \L F(\fy)$ 
and $\al>0$, we can replace each $\fy_n$ with $\fy_n(x/\nu_n)$ with some $\nu_n\to+0$, so that we have after the rescaling 
\EQ{
 \|\na\fy_n\|_{L^2} \le A, \pq K(\fy_n)=0, \pq \|\fy_n\|_{L^2}\to 0.}
Hence $m_{\al,\be}\le\liminf_{n\to\I}J(\fy_n)\le A^2/2$, and so $m_{\al,\be}\le \mm(\L F)^2/2=m_{0,1}$. 
Thus in both cases we have $m_{\al,\be}\le m_{0,1}\le \mm(F)^2/2$. 

Now suppose that $m_{\al,\be}<m_{0,1}\le \mm(F)^2/2$. As in the proof of Lemma \ref{Q sub} for $(d,\al)\not=(2,0)$, we may find a sequence of radial $\fy_n \in H^1$ such that  
\EQ{
 K(\fy_n)=0, \pq H(\fy_n)\searrow m,}
and $\fy_n\to\exists\fy$ weakly in $H^1$, and pointwise for $x\not=0$. 

Let $\psi_n=\fy_n-\fy$. Then $\psi_n\to 0$ weakly in $H^1$, and so
\EQ{ \label{KL split}
 \lim_{n\to\I} K^Q(\fy_n) \pt = \lim_{n\to\I} K^Q(\psi_n) + K^Q(\fy)
 \pr = \lim_{n\to\I} \L F(\fy_n) = conc.\L F((\fy_n)_n) + \L F(\fy),}
where the second identity is because $K(\fy_n)=0$, and the last one follows from $\fy_n(x)\to \fy(x)$ for $x\not=0$ and the radial Sobolev inequality $\|r^{1/2}\fy_n\|_{L^\I}\lec\|\fy_n\|_{H^1}$. 
Since $H(\fy)\le m$ by Fatou's lemma, we have $K(\fy)\ge 0$, otherwise there would be some $\la<0$ such that $K(\fy^\la)=0$ and $H(\fy^\la)<H(\fy)\le m$, a contradiction. 
Thus $K^Q(\fy)\ge \L F(\fy)$, and so from \eqref{KL split}, we deduce 
\EQ{ \label{vn bd}
 \lim_{n\to\I} K^Q(\psi_n) \le conc.\L F((\fy_n)_n).}

Since $\L F(\fy_n)$ is bounded by \eqref{KL split}, Lemma \ref{TM compact} with $h_n:=(\al D+\be d)f$ implies that $conc.F((\fy_n)_n)=0$. Hence by \eqref{vn bd} and $(\L-\bar\mu)F\ge 0$, we get 
\EQ{
  \lim_{n\to\I} K^Q(\psi_n) \le conc.(\L -\bar\mu)F((\fy_n)_n) \le \lim_{n\to\I}(\L-\bar\mu)F(\fy_n).}
On the other hand we have 
\EQ{
 m = \lim_{n\to\I} H(\fy_n) = \lim_{n\to\I} H^Q(\psi_n) + H^Q(\fy) + \lim_{n\to\I}(\L-\bar\mu)F(\fy_n)/\bar\mu,}
where $H^Q(\psi):=(1-\L/\bar\mu)\|\psi\|_{H^1}^2/2$ denotes the quadratic part of $H$. 
Combining the above two, and discarding $H^Q(\fy)\ge 0$, we obtain
\EQ{
 \lim_{n\to\I} \|\psi_n\|_{H^1}^2/2 \le m < \mm(F)^2/2 = 2\pi/\ka_0. }

Hence applying Lemma \ref{TM compact} to $\fy_n$ with $h (u):=e^{\ka|u|^2}-1$ for some $\ka\in(\ka_0,2\pi/m)$, we get $\L F(\fy_n)\to \L F(\fy)$, and so $\fy$ is a minimizer for $m_{\al,\be}$. Indeed, we have 
\EQ{ e^{\ka |\fy_n|^2} -1  \leq  e^{  C_{\ka, \ka'}  |\fy|^2} -1 +  e^{\ka' |\psi_n|^2} -1      } 
for some $\ka' \in  (\ka,2\pi/m)$ and constant  $C_{\ka, \ka'} > 0  $. 
   Hence  $h (\fy_n) $ is uniformly 
 bounded in $L^1$. Recall that for a fixed $\fy \in H^1$, 
$ e^{  C_{\ka, \ka'}  |\fy|^2}-1  \in L^1 $.

Then as in the proof of Lemma \ref{Q sub}, we obtain a ground state $Q$ with $J(Q)=m_{\al,\be}<m_{0,1}$, which is a contradiction since $K_{0,1}(Q)=0$. 
Hence $m_{\al,\be}=m_{0,1}$ for all $(\al,\be)$ in the range \eqref{range albe}. 
\end{proof}
\begin{rem}
In the above argument for $(\al,\be)=(0,1)$ in the case $\CTMS(F)\le 1$, we used a priori bounds on the ground state to get the compactness. 
For general sequences, we can have concentrating loss of compactness on the kinetic threshold $\|\na \fy\|_{L^2}=\mm(F)$ if and only if $f$ satisfies 
\EQ{
 \limsup_{|u|\to\I} e^{-\ka_0|u|^2}|u|^2f(u) \in (0,\I).}
The above result implies that the concentration requires more energy than the (mass-modified) ground state. 
Similar phenomena have been observed in slightly different settings (either on a bounded domain or on the $H^1(\R^2)$ threshold, where $e^{\ka_0|u|^2}$ appears as the critical growth instead of $e^{\ka_0|u|^2}/|u|^2$, see \cite{CC,Flu,Ruf}). 
More details about this issue, including the above concentration compactness, will be addressed in a forthcoming paper \cite{TM}. 
\end{rem}

\subsection{Parameter independence of the splitting}
The $(\al,\be)$-independence of $\K^\pm_{\al,\be}$ follows from that of $\m_{\al,\be}$ and contractivity of $\K^+_{\al,\be}$. 
\begin{lem}[Parameter independence of $\K^\pm$] \label{K indp}
Assume that $f$ satisfies \eqref{asm f}, and that $(\al,\be)$ satisfies \eqref{range albe}. Then $\K^\pm_{\al,\be}$ in \eqref{def Kpm} are independent of $(\al,\be)$. 
\end{lem}
\begin{proof}
Since $\m_{\al,\be}$ is independent of $(\al,\be)$, we only need to see that the sign of $K$ is independent under the threshold $m$. 
Moreover, we may restrict to the first component. For any $\de\ge 0$, we define $\K_{\al,\be}^{\pm\de}\subset H^1$ by  
\EQ{
 \pt \K_{\al,\be}^{+\de} = \{\fy \in H^1 \mid J(\fy)<m-\de,\ K_{\al,\be}(\fy)\ge 0\},
 \pr \K_{\al,\be}^{-\de} = \{\fy \in H^1 \mid J(\fy)<m-\de,\ K_{\al,\be}(\fy)<0\}.}
Then $(u_0,u_1)\in\K_{\al,\be}^\pm$ if and only if $u_0\in\K_{\al,\be}^{\pm\de}$ with $\de=\|u_1\|_{L^2}^2/2$. 
In addition, the disjoint union $\K_{\al,\be}^{+\de}\cup\K_{\al,\be}^{-\de}$ is already independent of $\alpha$ and $\beta$. Hence it suffices to show the independence of $\K_{\al,\be}^{+\de}$. 

First we consider the interior exponents satisfying $2\al+d\be>0$ and $2\al+(d-2)\be>0$. Then $\K^{+\de}_{\al,\be}$ is contracted to $\{0\}$ by the rescaling $\fy\mapsto\fy^\la$ with $0\ge\la\to-\I$. This is due to the following facts
\begin{enumerate}
\item $K(\fy^\la)>0$ is preserved as long as $J(\fy^\la)<m$, by the definition of $m$. 
\item $J(\fy^\la)$ does not increase as $\la$ decreases, as long as $\L J(\fy^\la)= K(\fy^\la)>0$. 
\item $\fy^\la\to 0$ in $H^1$ as $\la\to-\I$, since $2\al+d\be>0$ and $2\al+(d-2)\be>0$.
\end{enumerate}
In particular, $J$ cannot be negative on $\K^+_{\al,\be}$, and so $\K^{+\de}_{\al,\be}=\emptyset$ for $\de\ge m$. 
For $0\le\de<m$, both $\K^{\pm \de}_{\al,\be}$ are open in $H^1$. 
It follows for $\K^{-\de}$ from the definition, and for $\K^{+\de}$ from the facts that $J(\fy)<m$ and $K(\fy)=0$ imply $\fy=0$, and that a neighborhood of $0$ is contained in $\K^{+\de}$, which follows from \eqref{GN bd F}, \eqref{GN bd 1D} or \eqref{TM bd F}. Then the above argument of the scaling contraction shows that $\K^{+\de}_{\al,\be}$ is connected. 
Hence each $\K^{+\de}_{\al,\be}$ cannot be separated by $\K^{+\de}_{\al',\be'}$ and $\K^{-\de}_{\al',\be'}$ with any other $(\al',\be')$ in the interior range. Since $\K^{+\de}_{\al,\be}\cap \K^{+\de}_{\al',\be'}\ni 0$, we conclude that $\K^{+\de}_{\al,\be}=\K^{+\de}_{\al',\be'}$. 

Finally for $(\al,\be)$ on the boundary $2\al+d\be=0$ or $2\al+(d-2)\be=0$, 
take a sequence $(\al_n,\be_n)$ in the interior converging to $(\al,\be)$. Then $K_{\al_n,\be_n}\to K_{\al,\be}$, and so 
\EQ{
 \K^{\pm\de}_{\al,\be}\subset\bigcup_n \K^{\pm\de}_{\al_n,\be_n}.} 
Since the right hand side is independent of the parameter, so is the left. 
\end{proof}
 
\subsection{Variational estimates}
We conclude this section with a few estimates on the energy-type functionals, which will be important in the proof of the blow-up and the scattering. 
We start with the easy observation that the free energy and the nonlinear energy are equivalent in the set $\K^+$. 
\begin{lem}[Free energy equivalence in $\K^+$] \label{E0 equiv K+}
Assume that $f$ satisfies \eqref{asm f}. Then for any $(u_0,u_1)\in H^1(\R^d)\times L^2(\R^d)$ we have 
\EQ{ \label{J0 equiv}
 K_{1,0}(u_0)\ge 0 \implies 
 \CAS{ J(u_0) \le \|u_0\|_{H^1_x}^2/2 \le (1+d/2)J(u_0), \\
   E(u_0,u_1) \le E^Q(u_0,u_1) \le (1+d/2) E(u_0,u_1).}}
\end{lem}
\begin{proof}
Since $(D-2-c)f(u)\ge 0$ with $c:=4/d>0$ by \eqref{f conv}, we have for any $(u_0,u_1)\in H^1\times L^2$, 
\EQ{ \label{K0 E bd} 
 \pt K_{1,0}(u_0) \pn= \|u_0\|_{H^1_x}^2 - (2+c) F(u_0) - \int (D-2-c)f(u_0) dx 
 \pr\le (2+c)J(u_0) - c\|u_0\|_{H^1_x}^2/2  
 \pn= (2+c)E(u_0,u_1) - c E^Q(u_0,u_1) - \|\dot u\|_{L^2_x}^2,}
and hence we obtain the desired estimate.
\end{proof}

In the 2D exponential case, we have a sharper bound on the derivatives, which implies that $\K^+$ is in the subcritical regime for the Trudinger-Moser inequality. 
\begin{lem}[Subcritical bound in $\K^+$ in the 2D exponential case] \label{E bd exp}
Assume that $f$ satisfies \eqref{asm f} and \eqref{f exp}. Then for any $(u_0,u_1)\in\K^+$ we have 
\EQ{
 \|\na u_0\|_{L^2}^2 + \|u_1\|_{L^2}^2 < 2m \le \mm(F)^2=4\pi/\ka_0.}
\end{lem}
\begin{proof}
Since $K_{0,1}(u_0)\ge 0$, we have 
\EQ{
 \|\na u_0\|_{L^2}^2 + \|u_1\|_{L^2}^2
 \pt\le \|\na u_0\|_{L^2}^2 + \|u_1\|_{L^2}^2 + K_{0,1}(u_0) 
 \pn= 2E(u_0,u_1) < 2m.}
\end{proof}

The next estimate gives a lower bound on $|K|$ under the threshold $m$, which will be important both for the blow-up and for the scattering. 
\begin{lem}[Uniform bounds on $K$] \label{lem:K bd}
Assume that $f$ satisfies \eqref{f conv}, and that $(\al,\be)$ satisfies \eqref{range albe} and $(d,\al)\not=(2,0)$. 
Then there exists $\de>0$ determined by $(\al,\be)$, $d$ and $\e$ in \eqref{f conv}, such that for any $\fy\in H^1$ with $J(\fy)<m$ we have
\EQ{
 K_{\al,\be}(\fy) \ge \min(\bar\mu(m-J(\fy)),\de K_{\al,\be}^Q(\fy)) 
 \text{ or }
 K_{\al,\be}(\fy)
  \le -\bar\mu(m-J(\fy)).}
\end{lem}
Note that if $(d,\al)=(2,0)$ then the conclusion is false, since in that case $K(\fy^\la_{\al,\be})=e^{d\be\la}K(\fy)\to 0$ as $\la\to-\I$, while $J(\fy^\la)$ is away from $m$, since it is decreasing if $K(\fy)>0$ and $J(\fy^\la)\nearrow H(\fy)<m$ if $K(\fy)<0$. 
\begin{proof}
We may assume $\fy\not=0$. 
Let $j(\la)=J(\fy^\la)$ and $n(\la)=F(\fy^\la)$, where $\fy^\la_{\al,\be}=\fy^\la$ is the rescaling \eqref{def scaling}. 
Then $j(0)=J(\fy)$ and $j'(0)=K(\fy)$, and \eqref{J conv} implies 
\EQ{ \label{dineq j}
 j'' \le (\bar\mu+\U\mu) j' - \bar\mu\U\mu j - \frac{2\al\e}{d+1} n'.}

First we consider the case $K(\fy)<0$. 
By Lemma \ref{K>0 near 0} together with \eqref{KL conc}, there exists $\la_0<0$ such that $j'(\la)<0$ for $\la_0<\la\le 0$ and $j'(\la_0)=0$. For $\la_0\le\la\le 0$ we have from \eqref{K F mono}, 
\EQ{
 (\bar\mu+\U\mu)j'-\bar\mu\U\mu j \le \bar\mu j'.}
Inserting this in \eqref{dineq j} and integrating it, we get 
\EQ{
 \int_{\la_0}^0 j''(\la) d\la \le \bar\mu \int_{\la_0}^0 j'(\la) d\la,}
and hence 
\EQ{
 K(\fy) = j'(0) \le \bar\mu(j(0)-j(\la_0)).}
Since $K(\fy^{\la_0})=0$ and $\fy^{\la_0}\not=0$, we have 
$j(\la_0)=J(\fy^{\la_0})\ge m$. Thus we obtain
\EQ{
 K(\fy) \le - \bar\mu(m-J(\fy)).}

Next we consider the case $K(\fy)>0$. 
If 
\EQ{ \label{K dominance}
 (2\bar\mu+\U\mu)K(\fy) \ge \bar\mu\U\mu J(\fy) + \frac{2\al\e}{d+1}\L F(\fy),}
then applying \eqref{J0 equiv} to the first term on the right hand side, and $K=K^Q-\L F$ to the second one, we get 
\EQ{
 \left[2\bar\mu+\U\mu+\frac{2\al\e}{d+1}\right]K(\fy)  
 \ge \frac{\bar\mu\U\mu}{2+d}\|\fy\|_{H^1}^2 + \frac{2\al\e}{d+1}K^Q(\fy),}
and so $K(\fy)\ge\de K^Q(\fy)$ for some $\de>0$, since $\U\mu>0$ or $\al>0$. 
If \eqref{K dominance} fails, then 
\EQ{ \label{j' bd}
 (2\bar\mu+\U\mu)j' < \bar\mu\U\mu j + \frac{2\al\e}{d+1}n',}
at $\la=0$, and so from \eqref{dineq j}, 
\EQ{ \label{bd j''}
 j'' < - \bar\mu j'.}
Now let $\la$ increase. 
As long as \eqref{j' bd} holds and $j'>0$, we have $j''<0$ and so $j'$ decreases and $j$ increases. Also by \eqref{J conv} and \eqref{K F mono} we have 
\EQ{
 n'' \ge (\bar\mu+\U\mu)n' - \bar\mu\U\mu n \ge \bar\mu n' \ge \bar\mu^2 n > 0.}
Hence \eqref{j' bd} is preserved until $j'$ reaches $0$. It does reach at finite $\la_0>0$, because the right hand side of \eqref{dineq j} is negative and decreasing as long as $j'>0$. Now integrating \eqref{bd j''} we obtain
\EQ{
 K(\fy)= j'(0) \ge \bar\mu(j(\la_0)-j(0)) \ge \bar\mu(m-J(\fy)),}
where we used that $J(\fy^{\la_0})\ge m$ which follows from $K(\fy^{\la_0})=0$ and $\fy^{\la_0}\not=0$. 
\end{proof}

\section{Blow-up}
Here we prove the blow-up part of Theorem \ref{main thm}. 
The idea is essentially due to Payne-Sattinger \cite{PS}, but we give a full proof for convenience. We will use that $\K^-$ is stable under the flow. 

By contradiction we assume that the solution $u$ exists for all $t>0$. 
The proof for $t<0$ is the same and omitted.  
Let 
\EQ{
 y(t):=\|u(t,x)\|_{L^2_x(\R^d)}^2.} 
Multiplying the equation with $u$, and using \eqref{K0 E bd}, we get  
\EQ{ \label{convex1}
 \ddot y = 2\|\dot u\|_{L^2}^2 - 2K_{1,0}(u) 
  \ge (4+c)\|\dot u\|_{L^2}^2-2(2+c)E(u) + c \|u\|_{H^1}^2,}
for some $c>0$. Sine $u(t)\in\K^-$, Lemma \ref{lem:K bd} implies that there is some positive $\de\le-K_{1,0}(u(t))$. 
Thus for all $t>0$ we have 
\EQ{ 
 \ddot y(t)\ge 2\de >0,}
and so $y(t)=\|u(t)\|_{L^2}^2\to\I$ as $t\to\I$. 
Going back to \eqref{convex1}, and using Schwarz, we deduce that for large $t$
\EQ{
 {\ddot y} \geq (4+c)\|\dot u\|_{L^2}^2 > \frac{4+c}{4}\frac{{\dot y}^2}{y},} 
therefore 
\EQ{
 (y^{-c/4})_{tt} = -\frac{c}{4}y^{-c/4-2}\left[y\ddot y-\frac{4+c}{4}\dot y^2\right] < 0,}
which contradicts that $y\to\I$. 

\section{Global space-time norm}\label{sect:str} 
In this section we introduce Strichartz-type estimates and a perturbation lemma for global space-time bounds of the solution. 

The inhomogeneity of the Klein-Gordon equation makes the exponents a bit more complicated than the case of wave or Schr\"odinger equation. 
In the $H^1$ critical case, we get another complication in higher dimensions, due to the fact that we have to estimate the difference of solutions in some Sobolev (or Besov) spaces with positive regularity but the nonlinearity is not twice differentiable\footnote{The problem is not on the local regularity of the nonlinearity (at $u=0$), but rather on the global H\"older continuity for $f_L$.}. 
This is not a problem in the subcritical case, where we are allowed to lose small regularity, so that we can estimate the difference in some $L^p$ spaces and then interpolate. 
This technical issue was solved in the pure critical case in \cite{3Dcrit} by using space-time norms with exponents away from the admissible region for the standard Strichartz estimate, which was later called ``exotic Strichartz estimates" in the Schr\"odinger case \cite{TaoVisan}. 

Here we have a further complication by the presence of lower powers, for which we need the exotic Strichartz for the Klein-Gordon equation. 
Note that it is not a big trouble in the Schr\"odinger case (see \cite{TaoVisanZhang}), because the same Strichartz estimate is used both for higher and lower powers. 
In the Klein-Gordon case, in contrast, we have to use different Strichartz norms, with better regularity for higher powers and with better decay for lower powers. 
It is easy in the standard Strichartz estimate, where we can freely mix different norms by the duality argument, but this does not work for the exotic Strichartz estimate, which uses exponents away from the duality. 
Hence we are forced to use a common exponent for different powers, which makes our estimates much more involved. 
In particular, when we have both the $H^1$ critical and the $L^2$ critical powers, we need three steps to close our estimates. 

\subsection{Reduction to the first order equation} \label{rdc 1st}
To simplify the notation, we rewrite NLKG in the first order equation. With any real-valued function $u(t,x)$, we associate the complex-valued function $\V u(t,x)$ by 
\EQ{ \label{def vec}
 \V{u} = \LR{\na}u - i \dot u, \pq u = \LR{\na}^{-1}\Re\V u.}
This relation $u\leftrightarrow\V u$ will be assumed for any space-time function $u$ throughout this paper. 
Here we use $i$ purely for notational convenience, and we could use a vector form instead\footnote{We chose the complex form rather than the vector one, in order to avoid adding a subscript, for this notation will be applied mostly to sequences.}, especially if $u$ is originally complex-valued. 
The free and nonlinear Klein-Gordon equations are given by 
\EQ{
  \pt(\square+1)u=0 \iff (i\p_t + \Da)\V{u}=0, 
  \pr(\square+1)u = f'(u) \iff (i\p_t + \Da)\V{u} = f'(\Da^{-1}\Re\V{u}),}
and the free energy is given by $E^Q(u)=\|\V{u}\|_{L^2_x}^2/2$. We denote 
\EQ{ \label{def tiE}
 \pt \ti E(\fy) := \|\fy\|_{L^2_x}^2/2 - F(\LR{\na}^{-1}\Re\fy), 
 \pr \ti K_{\al,\be}(\fy):= K^Q_{\al,\be}(\LR{\na}^{-1}\fy) + K^N_{\al,\be}(\LR{\na}^{-1}\Re\fy).}
Remark that 
\EQ{
 \ti E(\V u(t))=E(u;t), \pq \ti K(\V u(t))\ge K(u(t)),} 
where the equality in the latter holds if and only if $\dot u(t)=0$. 
Nevertheless, the invariant set $\K^+=\K^+_{\al,\be}$ for $\V u$ is given by 
\EQ{ \label{def tiK}
 \ti\K^+ \pt:= \{\fy\in L^2(\R^d) \mid \ti E(\fy)<m,\ K(\Re\LR{\na}^{-1}\fy)\ge 0\} \pr=\{\fy\in L^2(\R^d) \mid \ti E(\fy)<m,\ \ti K(\fy)\ge 0\}.}
The second identity (the first one is definition) is proved as follows. 
Let $\fy\in L^2(\R^d)$ satisfy $\ti E(\fy)<m$ and $K(\Re\LR{\na}^{-1}\fy)<0$. 
Let $\psi_1=\Re\LR{\na}^{-1}\fy$ and $\psi_2=\Im\LR{\na}^{-1}\fy$. 
Then Lemma \ref{lem:K bd} implies that 
\EQ{
 K(\psi_1) \le -\bar\mu(m-J(\psi_1)) < -\bar\mu\|\psi_2\|_{H^1_x}^2/2
 \le -K^Q(\psi_2),} 
so $\ti K(\fy)=K(\psi_1)+K^Q(\psi_2)<0$. 
Hence under the condition $\ti E(\fy)<m$, the signs of $K(\psi_1)$ and $\ti K(\fy)$ are the same, which proves \eqref{def tiK}. 

\subsection{Strichartz-type estimates and exponents} \label{Stz notation}
Here we recall the Strichartz estimate for the free Klein-Gordon equation, introducing some notation for the space-time norms and special exponents. 

With any triplet $(b,c,\s)\in [0,1]^2\times\R$ and any $q\in(0,\I]$, we associate the following Banach function spaces on $I\times\R^d$ for any interval $I$: 
\EQ{ \label{def Yq}
 \pt [(b,c,\s)]_q(I) := L^{1/b}_t (I;B^{\s}_{1/c,q}(\R^d)), 
 \pq [(b,c,\s)]_0(I) := L^{1/b}_t (I;L^{1/c}(\R^d)),
 \pr [(b,c,\s)]_q^\bul(I) := L^{1/b}_t (I;\dot B^{\s}_{1/c,q}(\R^d)),}
where $B^s_{p,q}$ and $\dot B^s_{p,q}$ respectively denote the inhomogeneous and homogeneous Besov spaces, 
and the following characteristic numbers with a parameter $\th\in[0,1]$:
\EQ{ \label{def regstr}
 \pt \reg^\th(b,c,\s) := \si - (1-2\th/d)b - d(c-1/2),
 \pr \str^\th(b,c,\s) := 2b + (d-1+\th)(c-1/2),
 \pr \dec^\th(b,c,\s) := b + (d-1+\th)(c-1/2).}
$\th=0,1$ correspond respectively to the wave and the Klein-Gordon equations.  
$\reg^\th$ indicates the regularity of the space, while $\str^\th$ and $\dec^\th$ indicate the space-time decay, corresponding respectively to the Strichartz and the $L^p-L^q$ decay estimates. 
We denote the regularity change and the duality in $H^{s-1/2}$ (here $-1/2$ takes account of one regularity gain in the wave equation) respectively by 
\EQ{ \label{exp change}
 (b,c,\s)^s := (b,c,s), \pq (b,c,\s)^{*(s)} := (1-b,1-c,-\s+2s-1).}

Given a real number $s$, we say $Z=(Z_1,Z_2,Z_3)$ is Strichartz $s$-admissible 
if for some $\th\in[0,1]$ we have 
\EQ{
 \pt 0\le Z_1\le 1/2, \pq 0\le Z_2< 1/2, \pq \reg^\th(Z) \le s, 
 \pq \str^\th(Z) \le 0.}
We avoid the endpoint $Z_2=1/2$ to mix different $\th$. The Strichartz estimates read 
\begin{lem}[see \cite{Brenner,GV1,MNO}] 
For any $s\in\R$, let $Z$ and $T$ be $s$-admissible. 
Then for any space-time function $u(t,x)$, any interval $I\subset\R$, and any $t_0\in I$, we have 
\EQ{
 \|u\|_{[Z]_2(I)} \lec \|u(t_0)\|_{H^s} + \|\dot u(t_0)\|_{H^{s-1}} + \|\ddot u - \De u + u\|_{[T^{*(s)}]_2(I)},}
where the implicit constant does not depend on $I$ or $t_0$. 
\end{lem}
The ``exotic Strichartz estimate" is given for the Klein-Gordon equation by 
\begin{lem}
Let $Z,T\in\R^3$ satisfy for some $\th\in[0,1]$ 
\EQ{
 \pt \reg^\th(Z) \le \reg^\th(T)+2, \pq \str^\th(Z) \le \str^\th(T)-2, \pq 0<Z_1,T_1<1, 
 \pr \dec^\th(Z) < 0 < \dec^\th(T)-1,  \pq 0< \frac 12 -  Z_2,  T_2 - \frac 12 < \frac{1}{d-1+\th}.}
Then we have for any interval $I\subset\R$, $t_0\in I$, and $u(t,x)$ satisfying $u(t_0)=\dot u(t_0)=0$, 
\EQ{
 \|u\|_{[Z]_2(I)} \lec \|\ddot u - \De u + u\|_{[T]_2(I)}.}
\end{lem}
\begin{proof}
The wave case $\th=0$ was essentially proved in \cite[Lemma 7.4]{3Dcrit}, where the borderline case $\str^0(Z)=\str^0(T)-2$ was excluded for the real interpolation to improve the Besov exponent $2$. 
Here we discard that improvement, restoring the borderline case, which is needed for the lower critical power $p_1=4/d$. 

The proof is rather immediate from the standard Strichartz estimate and the $L^p$ decay estimate. 
Indeed, if $\str^\th(Z)=0=\str^\th(T)-2$ and $\reg^\th(Z)=\reg^\th(T)+2$, then the above estimate is nothing but Strichartz. 
If moreover $Z_2+T_2=1$, then the estimate directly follows from the $L^p$ decay and Hardy-Littlewood-Sobolev
\EQ{
 \pt\|\int_{t_0}^t \LR{\na}^{-1}e^{\pm i(t-s)\LR{\na}}h(s)ds\|_{[Z]_2(I)}
  \pr\lec \|\int_{t_0}^t |t-s|^{-2Z_1}\|h(s)\|_{B^{T_3}_{1/T_2,2}}ds\|_{L^{1/T_1}(I)}
  \pn\lec \|h\|_{[T]_2(I)}.}
This estimate can be translated in the time and the regularity exponents as 
\EQ{
 Z \mapsto Z'=Z+(b,0,s), \pq T \mapsto T'=T+(b,0,s)}
for any $s\in\R$ and $b\in(-1/2,1/2)$, as long as $0<Z'_1,T'_1<1$. 
By the complex interpolation for those estimates and the standard Strichartz estimate, we obtain the desired estimate in the case $\str^\th(Z)=\str^\th(T)-2$ and $\reg^\th(Z)=\reg^\th(T)+2$. 
It is extended to the remaining cases (with inequality in these relations) by the Sobolev embedding. 
\end{proof}

The following interpolation is convenient to switch from some exponents to others, 
\begin{lem} \label{lem:trinterpole}
Let $Z,A,B,C\in[0,1]\times\R$ and $\th\in[0,1]$. Assume that $A_1<Z_1<B_1$ and one of the followings
\begin{enumerate}
\item $\min(\str^{\th}(A),\str^{\th}(B),\str^{\th}(C)) \ge \str^{\th}(Z)$ and $\min(\reg^{\th}(A),\reg^{\th}(B)) > \reg^{\th}(Z)$
\item $\min(\str^{\th}(A),\str^{\th}(B)) > \str^{\th}(Z)$ and $\min(\reg^{\th}(A),\reg^{\th}(B),\reg^\th(C)) \ge  \reg^{\th}(Z)$.
\end{enumerate}
Then there exist $\al,\be,\ga\in(0,1)$ satisfying $\al+\be+\ga=1$ and for all $q\in(0,\I]$ we have the interpolation inequality 
\EQ{
 \|u\|_{[Z]_q} \lec \|u\|_{[A]_\I}^\al \|u\|_{[B]_\I}^\be \|u\|_{[C]_\I}^\ga.}
\end{lem}
\begin{proof}
Since $A_1<Z_1<B_1$, for any $0<\th_2\ll 1$ there exists $\th_1\in(0,1)$ such that 
\EQ{
 (1-\th_2)((1-\th_1)A_1+\th_1 B_1) + \th_2 C_1 = Z_1.}
Let $\ti Z:=(1-\th_2)((1-\th_1)A+\th_1 B) + \th_2 C$. Then from the assumption we have 
\EQ{ \label{ti Z and Z}
 \str^{\th}(\ti Z)\ge \str^{\th}(Z), \pq \reg^{\th}(\ti Z)\ge \reg^{\th}(Z),}
which imply $\ti Z_2\ge Z_2$ and $\ti Z_3-d\ti Z_2 \ge Z_3-d Z_2$, 
and so we have the Sobolev embedding $[\ti Z]_q\subset [Z]_q$. 
In the first case, we have $\reg^\th(\ti Z)>\reg^\th(Z)$ and so 
\EQ{
 [[[A]_\I,[B]_\I]_{\th_1},[C]_\I]_{\th_2} = [\ti Z]_\I \subset [Z]_q.}
The desired inequality follows from that for the complex interpolation. 

It remains to prove in the second case. 
By the real interpolation in the Besov space in $x$ and H\"older in $t$, we have for all $0<\de\ll 1$, 
\EQ{
 \pt \|u\|_{[Z]_q} \lec \|u\|_{[Z+]_\I}^{1/2} \|u\|_{[Z-]_\I}^{1/2}, 
 \pq Z^\pm := Z \pm \de(1,0,1-2\th/d).}
Let $0<\e\ll 1$ satisfy $\e(B_1-A_1)(1-\th_2)=\de$ and 
\EQ{
 \ti Z^\pm:=(1-\th_2)((1-\th_1\mp\e)A+(\th_1\pm\e)B) + \th_2 C.}
Then from the assumption and the definition of $Z^\pm$ and $\e$, we have 
\EQ{
 \str^{\th}(\ti Z^\pm)>\str^{\th}(Z^\pm), \pq \reg^{\th}(\ti Z^\pm)\ge \reg^{\th}(Z^\pm)=\reg^{\th}(Z),}
when $\e>0$ is small. Hence we have the Sobolev embedding 
\EQ{
 [[[A]_\I,[B]_\I]_{\th_1\pm\e},[C]_\I]_{\th_2} = [\ti Z^\pm]_\I \subset [Z^\pm]_\I,}
where the left hand side is a nested complex interpolation space. 
Now the conclusion follows from the interpolation inequality.   
\end{proof}

\subsection{Global perturbation of Strichartz norms}
Now we fix a few particular exponents. Define $H,W,K$ by 
\EQ{ \label{def WK}
 \pt H:=\paren{0,\frac 12,1}, \pq W := \left(\frac{d-1}{2(d+1)},W_1,\frac 12\right), \pq K := \left(\frac{d}{2(d+2)},K_1,\frac 12\right).} 
Then $[H]_2=L^\I_t H^1_x$ is the energy space, while $W$ and $K$ are $1$-admissible, diagonal and boundary exponents respectively for the wave ($\th=0$) and  the Klein-Gordon ($\th=1$) equations:
\EQ{
 \pt 1=\reg^0(H)=\reg^1(H)=\reg^0(W)=\reg^1(K), 
 \pr 0=\str^0(H)=\str^1(H)=\str^0(W)=\str^1(K).}
Let $eq(u)$ denote the left hand side of NLKG
\EQ{
 eq(u) := u_{tt} - \De u + u - f'(u).}
Recall the convention $u\leftrightarrow\V u$ in Section \ref{rdc 1st} to switch to the first order equations. We will treat the $H^1$ critical case \eqref{f crit} together with the subcritical case. 
Since $f_S(u)$ is for small $|u|$ and $f_L(u)$ for large $|u|$, we may freely lower $p_1$ in \eqref{f_S} and raise $p_2$ in \eqref{f sub}. 
Hence we assume \eqref{f_S} with 
\EQ{ \label{range p1}
 2_\star-2 = \frac{4}{d} < p_1 < \frac{4(d+1)}{(d+2)(d-1)},}
and we assume either $d=1$, \eqref{f exp} or \eqref{f sub} with 
\EQ{ \label{range p2}
 \frac{4d-2}{d(d-2)}< p_2 \le 2^\star-2.}

Before the main perturbation lemma, we see that $[H]_2\cap[W]_2\cap[K]_2$ is enough to bound the full Strichartz norms of the solutions. 
\begin{lem} \label{Stz extend}
Assume that $f$ satisfies \eqref{asm f}. Let $Z$, $T$ and $U$ be $1$-admissible. In the 2D exponential case \eqref{f exp}, let $\Th\in(0,1)$. 
Then there exist a constant $C_1>0$ and a continuous function $C_2:(0,\I)\to(0,\I)$ such that for any interval $I$, any $t_0\in I$ and any $w(t,x)$, we have 
\EQ{
 \|w\|_{[Z]_2(I)}  
  \pt \le C_1\|\V w(t_0)\|_{L^2_x} + C_1\|eq(w)\|_{([T^{*(1)}]_2+[U^{*(1)}]_2)(I)}
  \prqq + C_2(\|w\|_{([H]_2\cap[W]_2\cap[K]_2)(I)}),}
provided, in the exponential case, that 
\EQ{ \label{exp sub cond}
 \sup_{t\in I} \ka_0 \|\na w\|_{L^2_x}^2 \le 4\pi\Th.}
\end{lem}
We remark that \eqref{exp sub cond} is needed only in the exponential case. 
\begin{proof}
We may assume $\Th>1/2$ without losing any generality. 
We introduce the new exponents $M^\sharp$ and $X$ by 
\EQ{ \label{def MsharpX}
 M^\sharp:=\frac{2}{p_2(d+1)}(1,1,0), \pq X:=(\nu,0,\nu-\nu^2),}
with some $\nu\in(0,1/10)$ satisfying $\Th < (1-\nu)^2$, where 
$M^\sharp$ is used only if $d\ge 2$ and $X$ only in the exponential case. 
In either case we have 
\EQ{
 0>\str^0(M^\sharp),\ \str^0(X), \pq 1\ge \reg^0(M^\sharp),\ 1>\reg^0(X), 
 \pq 0<M^\sharp_1,X_1<W_1.}
Hence by Lemma \ref{lem:trinterpole}(1), we have 
\EQ{
 \|w\|_{[M^\sharp]_2(I)} + \|w\|_{[X]_2(I)} \lec \|w\|_{([H]_2\cap[W]_2\cap[K]_2)(I)}.}
The Strichartz estimate gives 
\EQ{
 \|w\|_{[Z]_2(I)} \pt\lec \|\V{w}(t_0)\|_{L^2_x} + \|eq(w)\|_{([T^{*(1)}]_2+[U^{*(1)}]_2)(I)} 
 \prq + \|f'(w)\|_{([K^{*(1)}]_2+[W^{*(1)}]_2+L^1_tL^2_x)(I)}.}
By the standard nonlinear estimate we have
\EQ{
 \|f_S'(w)\|_{[K^{*(1)}]_2(I)} \lec \|w\|_{[K]_2(I)} \|w\|_{[K]_0(I)}^{4/d},}
and in the subcritical/critical cases
\EQ{ \label{fL hd}
 \|f_L'(w)\|_{[W^{*(1)}]_2(I)} \lec \|w\|_{[W]_2(I)} \|w\|_{[M^\sharp]_0(I)}^{p_2}.}
In the exponential case, there are $\ka>\ka_0$ and $\mu>0$ such that 
\EQ{ \label{Th' bd}
 \sup_{t\in I} \ka \|w\|_{H^1_\mu}^2 \le 4\pi\Th',}
where $\Th':=(1+\Th)/2<1$ and  
\EQ{ \label{def H1mu}
 \|\fy\|_{H^1_\mu} := \|\na\fy\|_{L^2_x}^2 + \mu \|\fy\|_{L^2_x}^2.}
Then we have 
\EQ{ \label{f_L bd exp}
 \|f_L'(w)\|_{L^2_x}
 \lec \||w|(e^{\ka|w|^2}-1)\|_{L^2_x}
 \lec \|w\|_{L^\I_x} \|e^{\ka|w|^2}-1\|_{L^1_x}^{1/2} \|e^{\ka|w|^2}\|_{L^\I_x}^{1/2},}
where the second factor is bounded by Trudinger-Moser
\EQ{
 \|e^{\ka|w|^2}-1\|_{L^1_x} \lec \|w\|_{L^2}^2/(1-\Th'),}
and the third factor is bounded by the following log-interpolation inequality \cite[Theorem 1.3]{IMM2}: for any $\al\in(0,1)$, $\la>1/(2\pi\al)$ and $\mu>0$, there is $C>0$ such that 
\EQ{ \label{log ineq} 
 \|\fy\|^2_{L^\infty(\R^2)}
  \le \la \|\fy\|_{H^1_\mu(\R^2)}^2 
   \left[C+\log(1+\|\fy\|_{C^{\al}(\R^2)}/\|\fy\|_{H^1_\mu(\R^2)})\right],}
for any $\fy\in H^1\cap C^\al(\R^2)$, where $C^\al=B^\al_{\I,\I}$ denotes the H\"older space. 
Plugging this with $\al:=\nu-\nu^2$ into the exponential, we get 
\EQ{
 \|e^{\ka|w|^2}\|_{L^\I_x} \pt\lec (1+\|w\|_{C^\al_x}/\|w\|_{H^1_\mu})^{\la\ka\|w\|_{H^1_\mu}^2}
 \pn\lec (1 + \ka\|w\|_{C^\al_x}^2/\Th')^{2\pi\la\Th'},}
where $\la>0$ is chosen so that 
\EQ{ 
 1 < 2\pi\la\al, \pq (2\pi\la\Th' +1)\nu =1.}
Since $f_L$ vanishes for small $|u|$, we may assume $\|w\|_{C^\al_x}\gec\|w\|_{L^\I_x}\gec 1$. Hence 
\EQ{
 \|e^{\ka|w|^2}\|_{L^\I_x} \lec \|w\|_{C^\al_x}^{4\pi\la\Th'} = \|w\|_{C^\al_x}^{2(1/\nu-1)},}
and plugging this into \eqref{f_L bd exp}, we get 
\EQ{ \label{f_L est exp}
 \|f_L'(w)\|_{L^1_t L^2_x}
 \lec \|w\|_{L^{1/\nu}_t L^\I_x} \|w\|_{L^\I_t L^2_x} \|w\|_{L^{1/\nu}_t C^\al_x}^{1/\nu-1} \lec \|w\|_{[X]_2}^{1/\nu} \|w\|_{[H]_2}.}
\end{proof}

\begin{lem} \label{PerturLem} 
Assume that $f$ satisfies \eqref{asm f}.
Let $Z$, $T$, $U$ and $V$ be $1$-admissible and $\reg^0(V)=1$. 
In the exponential case \eqref{f exp}, let $\Th\in(0,1)$. 
Then there are continuous functions $\e_0,C_0:(0,\I)^2\to(0,\I)$ such that the following holds: 
Let $I\subset \R$ be an interval, $t_0\in I$ and $\V{u},\V{w}\in C(I;L^2(\R^d))$. Let $\V{\ga_0}=e^{\iD(t-t_0)}(\V{u}-\V{w})(t_0)$ and 
assume that for some $A,B>0$ we have
\EQ{ \label{asm ebd}
  \|\V{u}\|_{L^\I_t(I;L^2_x)} + \|\V{w}\|_{L^\I_t(I;L^2_x)}  \le A,}
\EQ{ 
  \|w\|_{[W]_2(I) \cap [K]_2(I)} \le B,} 
\EQ{
 \pt\|(eq(u),eq(w))\|_{([T^{*(1)}]_2+ [U^{*(1)}]_2)(I)} 
 + \|\ga_0\|_{[V]_\I(I)} \le \e_0(A,B), }
and in the exponential case, 
\EQ{ \label{exp cond sub uw}
 \sup_{t\in I} \ka_0 \max(\|\na u\|_{L^2_x}^2,\|\na w\|_{L^2_x}^2) \le 4\pi\Th.}
Then we have 
\EQ{ 
  \|u\|_{[Z]_2(I)} \le C_0(A,B).} 
\end{lem}
\begin{rem}
\eqref{exp cond sub uw} is needed only in the exponential case. The above lemma remains valid in the lower critical case $p_1=4/d=2_\star-2$, if we assume in addition that 
\EQ{
 \|\ga_0\|_{[K]_0(I)} \le \e_0(A,B).}
We will indicate the necessary modifications in the proof. 
\end{rem}

\begin{proof}[Proof of Lemma \ref{PerturLem}]
We restrict $p_1,p_2$ as in \eqref{range p1} and \eqref{range p2}, without losing any generality. 
In the following, $C(\cdot,\dots)$ denotes arbitrary positive constants which may depend continuously on the indicated parameters. Let $\de\in(0,1)$ be a fixed small number, whose smallness will be specified by the following arguments. 
Let 
\EQ{
 e:=eq(u)-eq(w), \pq \ga:=u-w.} 
Then we have the equation for the difference 
\EQ{
 \ddot \ga - \De \ga + \ga = f'(w+\ga)-f'(w)-e, \pq \V{\ga}(t_0)=\V{\ga_0}(t_0).} 
First note that by Lemma \ref{Stz extend}, we have the full Strichartz norms on $w$. 

Next we estimate the difference $u-w$ in the easier case $d\le 4$. 
We define new exponents $S,L$ and a space $\X$ by 
\EQ{ \label{def SL}
 \pt [S]_0 := L^{p_1+1}_t L^{2(p_1+1)}_x, \pq [L]_0 := L^{p_2+1}_t L^{2(p_2+1)}_x,
 \pr \X:=\CAS{[S]_0 &(d=1) \\ [S]_0\cap[X]_2 &\eqref{f exp}, \\ [S]_0\cap[L]_0 &\text{(otherwise).}}}
Thanks to the restrictions \eqref{range p1} and \eqref{range p2}, we have 
\EQ{
 0>\str^1(S),\ \str^0(L), \pq 1>\reg^1(S),\ \reg^0(L).}
Hence by Lemma \ref{lem:trinterpole}(2) with $C:=V$, we get for some $\th_1,\th_2\in(0,1)$, 
\EQ{ \label{S' start}
 \|\ga_0\|_{\X(I)} \lec A^{1-\th_1}\e_0^{\th_1} + A^{1-\th_2}\e_0^{\th_2}.}
If $p_1\to 4/d$, then $\str^0(S)\to 0$, and we would need the smallness in $[K]_0(I)$. 

Since $w\in \X(I)$ by Lemma \ref{Stz extend}, there exists a partition of the right half of $I$: 
\EQ{
 \pt t_0<t_1<\cdots<t_n,\pq I_j=(t_j,t_{j+1}),\pq I\cap(t_0,\I)=(t_0,t_n)}
such that $n\le C(A,B,\de)$ and 
\EQ{ \label{w bd lowD} 
 \|w\|_{\X(I_j)} \le \de \pq(j=0,\dots,n-1).}
We omit the estimate on $I\cap(-\I,t_0)$ since it is the same by symmetry. 

Let $\ga_j$ be the free solution defined by 
\EQ{
 \V \ga_j := e^{\iD(t-t_j)}\V \ga(t_j).}
Then the Strichartz estimate applied to the equations of $\ga$ and $\ga_{j+1}$ implies 
\EQ{ \label{est S'}
 \pn\|\ga-\ga_j\|_{\X(I_j)} + \|\ga_{j+1}-\ga_j\|_{\X(\R)}
 \pt\lec \|f'(w+\ga)-f'(w)\|_{L^1_tL^2_x(I_j)} 
  \prq+ \|e\|_{([U^{*(1)}]_2+[T^{*(1)}]_2)(I_j)}.}
The nonlinear difference is estimated as follows. 
For smaller $|u|$, we have by H\"older 
\EQ{ \label{fS diff lD}
 \|f_S'(w+\ga)-f_S'(w)\|_{L^1_t L^2_x} 
  \lec \|(w,\ga)\|_{[S]_0}^{p_1}\|\ga\|_{[S]_0},}
and for larger $|u|$ for $d\ge 2$ in the subcritical/critical cases,
\EQ{ \label{fL diff lD}
 \|f_L'(w+\ga)-f_L'(w)\|_{L^1_t L^2_x}
  \lec \|(w,\ga)\|_{[L]_0}^{p_2}\|\ga\|_{[L]_0}.}
If $d=1$, let $C(\nu)=\sup_{|u|\le\nu}|f_L''(u)|/|u|^{p_1}$. Then we have 
\EQ{ 
 \|f_L'(w+\ga)-f_L'(w)\|_{L^1_t L^2_x}
 \pt\lec C(\|w\|_{L^\I_{t,x}}+\|\ga\|_{L^\I_{t,x}})\|(w,\ga)\|_{[S]_0}^{p_1}\|\ga\|_{[S]_0}
 \pr\lec C(\|(w,\ga)\|_{L^\I_t H^1_x})\|(w,\ga)\|_{[S]_0}^{p_1}\|\ga\|_{[S]_0}. }
In the exponential case, there exist $\ka>\ka_0$ and $\mu>0$ such that \eqref{Th' bd}. 
Let $w_\th=w+\th\ga=(1-\th)w+\th u$ for $\th\in[0,1]$. Then we have $\ka \|w_\th\|_{H^1_\mu}^2\le 4\pi\Th'$, where $\Th'=(1+\Th)/2$ and $H^1_\mu$ is defined in \eqref{def H1mu}. In the same way as for \eqref{f_L est exp}, we obtain 
\EQ{ \label{fL diff exp}
 \pt\|f_L'(w+\ga)-f_L'(w)\|_{L^1_tL^2_x}
 \pn\le\int_0^1\|f''_L(w_\th)\ga\|_{L^1_tL^2_x} d\th
 \pr\lec \sup_{\th\in[0,1]}\|w_\th\|_{[H]_2}\|w_\th\|_{[X]_2}^{1/\nu-1}\|\ga\|_{[X]_2} \lec A\|(w,\ga)\|_{[X]_2}^{1/\nu-1}\|\ga\|_{[X]_2}.}

Thus in all cases, assuming  
\EQ{ \label{asm ga S'}
 \|\ga\|_{\X(I_j)} \le \de \ll 1,\pq (j=0,\dots,n-1),}
where the smallness depends on $A$ (and $\Th$), we get
\EQ{ \label{S' iterate}
 \|\ga\|_{\X(I_j)} + \|\ga_{j+1}\|_{\X(t_{j+1},t_n)}
 \le C\|\ga_j\|_{\X(t_j,t_n)}+\e_0,}
for some absolute constant $C\ge 2$. Then by \eqref{S' start} and iteration in $j$ we get 
\EQ{ \label{S' conclude}
 \|\ga\|_{\X(I)} \lec (2C)^n(A^{1-\th_1}\e_0^{\th_1}+A^{1-\th_2}\e_0^{\th_2}) \le C(A,B)(\e_0^{\th_1}+\e_0^{\th_2}).}
Choosing $\e_0(A,B)$ sufficiently small, we can make the last bound much smaller than $\de$, and thus the assumption \eqref{asm ga S'} is justified by continuity in $t$ and induction on $j$. 
Then repeating the estimate \eqref{est S'} once more, we can estimate the full Strichartz norms on $\ga$, which implies also the bound on $u$. 

Next we estimate the difference $u-w$ in the harder case $d\ge 5$, where we need the new exponents $\ti M$, $M$, $\ti N$, $N$, $R$, $Q$, $P$, and $Y$ defined by 
\EQ{ \label{def M2X}
 \pt M = \frac{2}{d+1}\left[\frac{1}{p_2}(1-d,2,0) + \frac{d-2}{4}(d,-1,0)\right],
 \pr \ti N = \frac{2}{d+1}\left[\paren{\frac{1}{2},\frac{d-1}{4},1}+\left(1-\frac{d-2}{4}p_2\right)(-d,1,0)\right],
 \pr \ti M = M + \frac{2}{p_2(d+1)}(0,1/d,1),
 \pq N = \ti N - \frac{2}{d+1}(0,1/d,1), 
 \pr Q = \frac{\paren{1,2,2}}{p_1(d+1)}, 
 \pq P = \frac{\paren{4,d-1,4}}{2(d+1)},
 \pq Y = \frac{\paren{6,d+3,4}}{2(d+1)},
 \pr R = \paren{\frac{(d+4)}{2(d+2)(p_1+1)},R_1,\frac{1}{2}}.}
In the case $p_2>1$, we need another exponent
\EQ{ \label{def hatM}
 \hat M := \ti M + \frac{2(p_2-1)}{p_2(d+1)}(0,1/d,1),}
and if $p_2\le 1$ then we put $\hat M=\ti M$. 
Note that $p_1<1$ under \eqref{range p1} for $d\le 5$. 
Then we have the sharp Sobolev embedding 
\EQ{
 [\hat M]_q \subset [\ti M]_q \subset [M]_q, \pq [\ti N]_q \subset [N]_q,}
and nonlinear and interpolation relations 
\EQ{ \label{interpolation relation}
 R + p_1 R^0 = K^{*(1)}, \pq R=(1-\al)W+\al K, \pq M^\sharp = (1-\be)W^0+\be R^0,}
for some $\al,\be\in(0,1)$, thanks to \eqref{range p1} and \eqref{range p2}. 
 $Y$ is a non-admissible exponent satisfying   
\EQ{
 Y \pt= \ti N + p_2 M = N + p_2 \ti M = P + p_1 Q^{0} = P^0 + p_1 Q,}
where the second and the last identities follow from $P_3=p_1Q_3$, $\ti N_3=p_2\ti M_3$, and the above sharp embeddings. 
If $p_2>1$, we have in addition
\EQ{
 Y = N + \hat M + (p_2-1)M.}
Moreover, these exponents satisfy (when $d\ge 5$) 
\begin{equation} \begin{gathered}
  1 = \reg^0(\ti N) = - \reg^0(Y) \ge \reg^0(\hat M),
 \pq 1 > \reg^1(Q), \reg^1(P),  
 -\reg^1(Y),
 \\ 0> \str^0(\hat M), \str^0(\ti N), \str^1(Q), \str^1(P), 
 \\  \str^0(\ti N) \le \str^0(Y)-2, \pq \str^1(P)=\str^1(Y)-2, 
 \\ 0\le \hat M_1, \hat M_2, Q_1, Q_2, R_1 < 1/2, 
 \pq 1< \dec^0(Y), 
  \dec^1(Y), 
 \\ Y_2<\frac 12 + \frac 1{d}, \pq \ti N_2>\frac 12-\frac 1{d-1}, \pq P_2>\frac 12- \frac 1{d}.
\end{gathered} \end{equation}
$\reg^0(\hat M)=1$ only if $p_2=2^\star-2=4/(d-2)$. 
Lemma \ref{lem:trinterpole}(1) implies that 
\EQ{
 \pt\|w\|_{([Q]_{2p_1} \cap [\hat M]_{2} \cap [\ti M]_{2p_2})(I)} \lec \|w\|_{([H]_2\cap[K]_2\cap[W]_2)(I)} \lec A+B.}
As before, we divide $I\cap(t_0,\I)$ into $t_0<\cdots<t_n$, $n\le C(A,B)$ such that 
\EQ{ \label{w bd highD}
 \|w\|_{([Q]_{2p_1} \cap [\hat M]_{2} \cap [\ti M]_{2p_2} \cap [K]_2 \cap [W]_2)(I_j)} \le \de \ll 1, \pq(j=0,\dots,n-1).}
We also introduce the following spaces: 
\EQ{ \label{def YY}
 \pt \Y_0:=[W]_0\cap[R]_0, \pq \ti\Y :=[\ti N]_2 \cap [P]_2, \pq \Y :=[W]_2\cap[K]_2, 
 \pr \Y_0^*:=[W^{*(1)}]_0+[K^{*(1)}]_0, \pq \Y^*:=[W^{*(1)}]_2 + [K^{*(1)}]_2.}
Our proof for $d\ge 5$ consists of three steps:
\begin{enumerate}
\item We estimate $\ga$ in $\Y_0$, assuming it is bounded in some norm similar to \eqref{w bd highD}. Here we can use the standard Strichartz because the estimates do not contain spatial derivative. 
\item We estimate $\ga$ in $\ti\Y$, under the same assumption on $\ga$. Here we use the exotic Strichartz. 
\item We estimate $u$ in $\Y$ by using the bounds in $[\ti N]_2\cap[R]_0$. 
The assumption in the previous steps is justified once we get a better bound. 
\end{enumerate}
Actually we could skip the first step, by using interpolation in the last step to bound $[R]_0$ by the other norms. 
However, if $p_1=4/d$ the lower critical power, then $R=K$ and the first step becomes necessary. 

Assuming that 
\EQ{ \label{asm ga bd}
 \|\ga\|_{([Q]_{2p_1} \cap [\hat M]_{2p_2} \cap [R]_0 \cap [M^\sharp]_0)(I_j)} \le \de \pq(j=0,\dots,n-1),}
we have by Strichartz and H\"older (since $W^0$ and $R^0$ are $1/2$-admissible)
\EQ{
 \pt\|\ga-\ga_j\|_{\Y_0(I_j)} + \|\ga_{j+1}-\ga_j\|_{\Y_0(\R)}
 \pr\lec \|f'(w+\ga)-f'(w)\|_{\Y_0^*(I_j)}+\|e\|_{\Y^*(I_j)}
 \pr\lec \|(w,\ga)\|_{[R]_0(I_j)}^{p_1}\|\ga\|_{[R]_0(I_j)} + \|(w,\ga)\|_{[M^\sharp]_0(I_j)}^{p_2}\|\ga\|_{[W]_0(I_j)} + \e_0
 \pr\lec \de^{p_1}\|\ga\|_{\Y_0(I_j)} + \e_0,}
where we used \eqref{w bd highD} and \eqref{asm ga bd}. 
By Lemma \ref{lem:trinterpole}(2), we have 
\EQ{ \label{S0 start} 
 \|\ga_0\|_{\Y_0(I)} \lec A^{1-\th_3}\e_0^{\th_3} + A^{1-\th_4}\e_0^{\th_4},}
for some $\th_3,\th_4\in(0,1)$. 
Note that $\str^1(R)\to 0$ as $p_1\to 4/d$, hence in the lower critical case we would need $\ga_0$ to be small in $[K]_0$. 
By the same argument as for \eqref{S' conclude}, we obtain 
\EQ{
 \|\ga\|_{\Y_0(I)} \le C(A,B)(\e_0^{\th_3}+\e_0^{\th_4}) \ll \de.}

Next, still assuming \eqref{asm ga bd}, we have by the exotic Strichartz estimate, 
\EQ{
 \pt\|\ga-\ga_j\|_{\ti \Y(I_j)} + \|\ga_{j+1}-\ga_j\|_{\ti \Y(\R)}
 \pn\lec \|f'(w+\ga)-f'(w)\|_{[Y]_2(I_j)} + \|e\|_{\Y^*(I_j)},}
where the nonlinear difference is estimated by 
\EQ{ \label{fL diff hD}
 \pn\|f_L'(w+\ga)-f_L'(w)\|_{[Y]_2} 
 \pt\lec  \|(w,\ga)\|_{[M]_0}^{p_2}\|\ga\|_{[\ti N]_2} 
   + \|(w,\ga)\|_{[\ti M]_{2p_2}}^{p_2}\|\ga\|_{[N]_0}
   \prq+ \|(w,\ga)\|_{[M]_0}^{p_2-1}\|(w,\ga)\|_{[\hat M]_2}\|\ga\|_{[N]_0}
,}
where the last term is for $p_2> 1$ while the second last is for $p_2\le 1$, and similarly
\EQ{ \label{fS diff hD}
 \pn\|f_S'(w+\ga)-f_S'(w)\|_{[Y]_2} 
 \pt\lec \|(w,\ga)\|_{[Q]_0}^{p_1}\|\ga\|_{[P]_2} + \|(w,\ga)\|_{[Q]_{2p_1}}^{p_1}\|\ga\|_{[P]_0}.} 
Thus we obtain
\EQ{ 
 \|\ga-\ga_j\|_{\ti \Y(I_j)} + \|\ga_{j+1}-\ga_j\|_{\ti \Y(\R)}\lec \de^{p_1} \|\ga\|_{\ti \Y(I_j)} + \e_0,}
where we used \eqref{w bd highD}, \eqref{asm ga bd}, and the following embeddings in $x$
\EQ{
 [Q]_{2p_1}\subset[Q]_0, \pq [P]_2\subset[P]_0, \pq [\hat M]_{2}+[\ti M]_{2p_2}\subset[M]_0, \pq [\ti N]_2\subset[N]_0.} 
By Lemma \ref{lem:trinterpole} and Strichartz, we have 
\EQ{
 \pt \|\ga_0\|_{[\ti N]_2(I)} \lec \|\ga_0\|_{[H]_2(I)\cap[W]_2(I)}^{1-\th_5}\|\ga_0\|_{[M]_0(I)}^{\th_5} \lec A^{1-\th_5}\e_0^{\th_5},
 \pr \|\ga_0\|_{[P]_2(I)} \lec \|\ga_0\|_{[H]_2(I)\cap[K]_2(I)}^{1-\th_6}\|\ga_0\|_{[M]_0(I)}^{\th_6}\lec A^{1-\th_6}\e_0^{\th_6},}
for some $\th_5,\th_6\in(0,1)$. 
Note that $\str^{1}(P)$ is away from $0$ as $p_1\to 4/d$, and so $\th_5,\th_6$ are uniformly bounded from below. 
Thus by the same argument as for \eqref{S' conclude}, 
\EQ{
 \|\ga\|_{\ti \Y(I)} \le C(A,B)(\e_0^{\th_5}+\e_0^{\th_6}) \ll \de.}
Hence under the assumption \eqref{asm ga bd} we have obtained  
\EQ{ \label{ga small}
 \|\ga\|_{[W]_0(I)\cap [R]_0(I)\cap[\ti N]_2(I)\cap[P]_2(I)} \lec C(A,B)\sum_{k=3}^6\e_0^{\th_k} \ll\de.}

Finally by Strichartz, \eqref{w bd highD} and \eqref{asm ga bd}, we have 
\EQ{ \label{est u in S}
 \|u\|_{\Y(I_j)} \pt\lec \|\V{u}(t_j)\|_{L^2_x} + \|eq(u)+f'(u)\|_{\Y^*(I_j)}
 \pr\lec A + \e_0 + \|u\|_{[R]_0(I_j)}^{p_1}\|u\|_{[R]_2(I_j)} + \|u\|_{[M^\sharp]_0(I_j)}^{p_2}\|u\|_{[W]_2(I_j)}
 \pr\lec A + \e_0 + \de^{p_1}\|u\|_{\Y(I_j)}.}
Hence we obtain 
\EQ{ \label{bd u in S}
 \|u\|_{\Y(I_j)} \lec A+\e_0,}
and so
\EQ{
 \|u\|_{\Y(I)} \lec n(A+\e_0)\le C(A,B),}
which is extended to the full Strichartz norms by Lemma \ref{Stz extend}. 

It remains to justify \eqref{asm ga bd}. By Lemma \ref{lem:trinterpole}(2), we have 
\EQ{
 \pt \|\ga\|_{[Q]_{2p_1}\cap [\hat M]_2 \cap [\ti M]_{2p_2}} \lec \sum_{k=7,8}\|\ga\|_{[H]_2\cap[K]_2\cap[W]_2}^{1-\th_k}\|\ga\|_{[P]_2\cap[\ti N]_2}^{\th_k},}
for some $\th_7,\th_8\in(0,1)$. If $p_1=4/d$, then we need to add $[K]_0$ to the last factor. 

In either case, by \eqref{bd u in S}, \eqref{w bd highD}, \eqref{ga small}, and \eqref{interpolation relation}, we obtain  
\EQ{
 \|\ga\|_{([Q]_{2p_1}\cap [\hat M]_2 \cap [\ti M]_{2p_2} \cap [R]_0 \cap [M^\sharp]_0)(I_j)} \lec C(A,B)\e_0^{\th},}
for some $\th\in(0,1)$. 
By choosing $\e_0(A,B)$ sufficiently small, the last bound can be made much smaller than $\de$. 
Then the assumption \eqref{asm ga bd} is justified by continuity in $t$ and induction in $j$. 
Thus we have obtained the desired estimates. 
\end{proof} 

\section{Profile decomposition}
In this section, following Bahouri-G\'erard and Kenig-Merle, we investigate behavior of general sequences of solutions, by asymptotic expansion into a series of transformation sequences of fixed space-time functions, called profiles. 
This is the fundamental part for the construction of {\it a critical element} in the next section. 

\subsection{Linear profile decomposition} \label{ss:lin prof}
Here we give the Klein-Gordon version of Bahouri-G\'erard's profile decomposition for the massless free wave equation. 
The only essential difference is that the massive equation does not commute with the scaling transforms, but the proof goes almost the same. 
 
For simple presentation, we introduce the following notation. 
For any triple $(t_\heartsuit^\diamondsuit,x_\heartsuit^\diamondsuit,h_\heartsuit^\diamondsuit)\in\R^{1+d}\times(0,\I)$ with arbitrary suffix $\heartsuit$ and $\diamondsuit$, let $\tau_\heartsuit^\diamondsuit$, $T_\heartsuit^\diamondsuit$ and $\Da_\heartsuit^\diamondsuit$ respectively denote the scaled time shift, the unitary and the self-adjoint operators in $L^2(\R^d)$, defined by 
\EQ{ \label{def Tnj}
 \tau_\heartsuit^\diamondsuit=-\frac{t_\heartsuit^\diamondsuit}{h_\heartsuit^\diamondsuit}, \pq T_\heartsuit^\diamondsuit \fy (x) = (h_\heartsuit^\diamondsuit)^{-d/2} \fy\left(\frac{x-x_\heartsuit^\diamondsuit}{h_\heartsuit^\diamondsuit}\right), 
 \pq \Da_\heartsuit^\diamondsuit = \sqrt{-\De+(h_\heartsuit^\diamondsuit)^2}.}
We denote the set of Fourier multipliers on $\R^d$: 
\EQ{ \label{def MP}
 \MP=\{\mu=\F^{-1}\ti\mu \F \mid \ti\mu\in C(\R^d),\ \exists\lim_{|x|\to\I}\ti\mu(x)\in\R\}.}
(practically we need only $1$ and $|\na|\LR{\na}^{-1}$ in $\MP$). 
Also recall the correspondence $u\leftrightarrow \V u$ defined in Section \ref{rdc 1st}. 
\begin{lem}[Linear profile decomposition] \label{LPD crit}
Let $\V v_n=e^{\iD t}\V v_n(0)$ be a sequence of free Klein-Gordon solutions with bounded $L^2_x$ norm. 
Then after replacing it with some subsequence, 
there exist $K\in\{0,1,2\dots,\I\}$ and, for each integer $j\in[0,K)$, $\fy^j \in L^2(\R^d)$ and $\{(t_n^j,x_n^j,h_n^j)\}_{n\in\N}\subset \R\times\R^d\times(0,1]$ satisfying the following. Define $\V v_n^j$ and $\V w_n^k$ for each $j<k\le K$ by
\EQ{
  \V{v}_n^j = e^{\iD(t-t_n^j)}T_n^j \fy^j, \pq \V{v}_n = \sum_{j=0}^{k-1} \V{v}_n^j + \V w_n^k,}
then we have 
\EQ{ 
 \lim_{k\to K} \limsup_{n\to\I} \|\V w_n^k\|_{L^\I_t(\R;B^{-d/2}_{\I,\I}(\R^d))}=0,}
and for any Fourier multiplier $\mu\in\MP$, any $l<j<k\le K$ and any $t\in\R$, 
\EQ{ \label{orth}
 \lim_{n\to\I} |\log(h_n^l/h_n^j)| + \frac{|t_n^l-t_n^j| + |x_n^l-x_n^j|}{h_n^l} = \I,}
\EQ{ \label{L2 orth}
 \lim_{n\to\I} \LR{\mu \V v_n^l(t)|\mu \V v_n^j(t)}_{L^2_x} = 0 
 =  \lim_{n\to\I} \LR{\mu \V v_n^j(t)|\mu \V w_n^k(t)}_{L^2_x}.}
Moreover, each sequence $\{h_n^j\}_{n\in\N}$ is either going to $0$ or identically $1$ for all $n$. 
\end{lem}
We call such a sequence $\{\V v_n^j\}_{n\in\N}$ \emf{a free concentrating wave} for each $j$, and $\V w_n^k$ \emf{the remainder}. 
We say that $\{(t_n^j,x_n^j,h_n^j)\}_{n}$ and $\{(t_n^k,x_n^k,h_n^k)\}$ are \emf{orthogonal} when \eqref{orth} holds. Note that \eqref{L2 orth} implies 
\EQ{ \label{L2 decop}
 \lim_{n\to\I}\left[\|\V v_n(t)\|_{L^2_x}^2-\sum_{j<k}\|\V v_n^j(t)\|_{L^2_x}^2-\|\V w_n^k\|_{L^2_x}^2\right]=0.}
We remark that the case $h_n^j\to\I$ is excluded by the presence of the mass, or more precisely by the use of inhomogeneous Besov norm for the remainder. 
\begin{proof}
We introduce a Littlewood-Paley decomposition for the Besov norm. Let $\La_0(x)\in \S(\R^d)$ such that its Fourier transform $\ti\La_0(\x)=1$ for $|\x|\le 1$ and $\ti\La_0(\x)=0$ for $|\x|\ge 2$. Then we define $\La_k(x)$ for any $k\in\N$ and $\La_{(0)}(x)$ by the Fourier transforms 
\EQ{
 \ti\La_k(\x) = \ti\La_0(2^{-k}\x) - \ti\La_0(2^{-k+1}\x),
 \pq \ti\La_{(0)} = \ti\La_0(\x) - \ti\La_0(2\x).}
Let 
\EQ{ 
 \nu:=\limsup_{n\to\I}\|\V v_n\|_{L^\I_t B^{-d/2}_{\I,\I}}
 \sim \limsup_{n\to\I} \sup_{t\in\R,\ x\in\R^d,\ k\ge 0} 2^{-k d/2}|\La_k*\V{v}_n(t,x)|.} 
If $\nu=0$, then we are done with $K=0$. 
Otherwise, there exists a sequence $(t_n,x_n,k_n)$ such that for large $n$
\EQ{ 
 2^{-k_n d/2}|\La_{k_n}*\V{v}_n(t_n,x_n)| \ge \nu/2.} 
Now we define $h_n$ and $\psi_n$ by 
\EQ{
 h_n = 2^{-k_n}, \pq \V{v}_n(t_n,x) = T_n \psi_n.}
Since $\psi_n$ is bounded in $L^2_x$, it converges weakly to some $\psi$ in $L^2_x$, up to an extraction of a subsequence. Moreover, 
\EQ{
 2^{-k_n d/2}|\La_{k_n}*\V{v}_n(t_n,x_n)|=\CAS{|\La_0*\psi_n(0)| &(k_n=0) \\ |\La_{(0)}*\psi_n(0)| &(k_n \ge 1),}}
and hence by the weak convergence and by Schwarz 
\EQ{
 \|\psi\|_{L^2_x} \gec |\LR{\La_0|\psi}|+|\LR{\La_{(0)}|\psi}| \ge \nu/2.} 
If $h_n\to 0$, then we put $(t_n^0,x_n^0,h_n^0)=(t_n,x_n,h_n)$ and $\fy^0=\psi$. Otherwise, we may assume that $h_n\to\exists h_\I>0$, by extracting a subsequence, and we put 
\EQ{
 (t_n^0,x_n^0,h_n^0)=(t_n,x_n,1), \pq \fy^0=h_\I^{-d/2}\psi(x/h_\I).} 
Then we have $T_n\psi-T_n^0\fy^0\to 0$ strongly in $L^2_x$. 
Now we define $\V v_n^0$ and $\V w_n^1$ by 
\EQ{
 \V v_n^0=e^{\iD(t-t_n^0)}T_n^0\fy^0, \pq \V w_n^1=\V v_n-\V v_n^0.} 
Then $(T_n^0)^{-1}\V w_n^1(t_n^0)=(T_n^0)^{-1}T_n\psi_n-\fy^0\to 0$ weakly in $L^2$, and $\mu T_n^0 = T_n^0 \mu_n^0$, where $\mu_n^0$ denotes the Fourier multiplier whose symbol is the rescaling of $\mu$'s, that is $\ti\mu(\x/h_n^0)$. By the definition of $\MP$, the symbol of $\mu_n^0$ converges including the case $h_n^0\to 0$, so $\mu_n^0\to\exists\mu_\I^0$ converges strongly on $L^2(\R^d)$. Hence 
\EQ{
\LR{\mu\V v_n^0(t_n^0)|\mu\V w_n^1(t_n^0)}_{L^2_x}
 = \LR{\mu_n^0\fy^0|\mu_n^0(T_n^0)^{-1}\V w_n^1(t_n^0)}_{L^2_x}
 \to 0.}
The left hand side is preserved in $t$, hence the above holds at any $t$. 
This is the decomposition for $k=1$. 

Next we apply the above procedure to the sequence $\V w_n^1$ in place of $\V v_n$. Then either the Besov norm goes to $0$ and $K=1$, or otherwise we find the next concentrating wave $\V v_n^1$ and the remainder $\V w_n^2$, such that for some $(t_n^1,x_n^1,h_n^1)$ and $\fy^1\in L^2(\R^d)$,
\EQ{ 
 \pt\V w_n^1 = \V v_n^1 + \V w_n^2, \pq \V v_n^1=e^{\iD(t-t_n^1)}T_n^1\fy^1,
 \pq \LR{\mu\V v_n^1(t)|\mu\V w_n^2(t)}_{L^2_x} \to 0,}
 $(T_n^1)^{-1}\V w_n^2(t_n^1)\to 0$ weakly in $L^2_x$ as $n\to\I$, and 
\EQ{ \label{w B bd}
 \pt \limsup_{n\to\I} \|\V w_n^1\|_{L^\I_t B^{-d/2}_{\I,\I}} \lec \|\fy^1\|_{L^2}.}

Iterating the above procedure, we obtain the desired decomposition. 
The $L^2$ orthogonality implies that $\|\fy^k\|_{L^2_x}\to 0$ as $k\to\I$, and then \eqref{w B bd} (for general $k$) gives the decay of the remainder in the Besov norm. 

It remains to prove the orthogonality \eqref{orth} as well as \eqref{L2 orth}. 
First we have 
\EQ{ \label{vnlj prod}
 \LR{\mu\V v_n^l(0)|\mu\V v_n^j(0)}
 = \LR{e^{-\iD t_n^l}T_n^l \mu_n^l \fy^l|e^{-\iD t_n^j}T_n^j \mu_n^j\fy^j}
 = \LR{S_n^{j,l}\mu_n^l\fy^l|\mu_n^j\fy^j},}
where $\ti\mu_n^l=\ti\mu(\x/h_n^l)$ as before, and $S_n^{j,l}$ is defined by
\EQ{ \label{def Snjl}
 S_n^{j,l} := (T_n^j)^{-1}e^{\iD(t_n^j-t_n^l)} T_n^l 
 = e^{-i\Da_n^j t_n^{j,l}}(T_n^j)^{-1}T_n^l = e^{-i\Da_n^j t_n^{j,l}} T_n^{j,l},}
with the sequence 
\EQ{ \label{def njl}
 (t_n^{j,l},x_n^{j,l},h_n^{j,l}):= (t_n^l-t_n^j,x_n^l-x_n^j,h_n^l)/h_n^j.}
Using the last formula in \eqref{def Snjl}, \eqref{orth} and uniform time decay of $e^{i\Da_n^j t}:\S\to\S'$, it is easy to observe that $S_n^{j,l}\to 0$ weakly on $L^2_x$ as $n\to\I$ for all $j<l$. 
Since $\ti\mu_n^l=\ti\mu(\x/h_n^l)$ and $\ti\mu_n^j$ are convergent, \eqref{vnlj prod} also tends to $0$. Then we have also 
\EQ{
 \LR{\mu\V v_n^j(t)|\mu\V w_n^k(t)}_{L^2_x}
 = \LR{\mu\V v_n^j(t)|\mu\V w_n^{j+1}(t)-\sum_{m=j+1}^{k-1}\mu\V v_n^m(t)}_{L^2_x}
 \to 0,}
thus we obtain \eqref{L2 orth}. 
Now suppose that \eqref{orth} fails, then there exists a minimal $(l,j)$ breaking \eqref{orth}, with respect to the natural order 
\EQ{
 (l_1,j_1)\le (l_2,j_2) \iff l_1\le l_2 \text{ and } j_1\le j_2.}
Then by extracting a subsequence, we may assume that $h_n^l\to h_\I^l$, $\log(h_n^l/h_n^j)$, $(t_n^l-t_n^j)/h_n^l$ and $(x_n^l-x_n^j)/h_n^l$ all converge. Now we inspect
\EQ{ \label{w expand}
 (T_n^l)^{-1}\V w_n^{l+1}(t_n^l)
 = \sum_{m=l+1}^j S_n^{l,m} \fy^m + S_n^{l,j} (T_n^j)^{-1}\V w_n^{j+1}(t_n^j).}
where $S_n^{l,j}$ converges strongly to a unitary operator, due to the convergence of $(t_n^{l,j},x_n^{l,j},h_n^{l,j})$ and $h_n^l$. 
Since $S_n^{l,m}\to 0$ for $m<j$ and $(T_n^j)^{-1}\V w_n^{j+1}(t_n^j)\to 0$ weakly in $L^2_x$, we deduce from the weak limit of \eqref{w expand} that $\fy^k=0$, a contradiction. This proves the orthogonality \eqref{orth}. 
\end{proof}

Those free concentrating waves with scaling going to $0$ are vanishing in any Besov space with less regularity. Hence in the subcritical case, we may freeze the scaling to $1$ by regarding them as a part of remainder. 
Hence we have 
\begin{cor} \label{LPD sub}
Let $\V v_n$ be a sequence of free Klein-Gordon solutions with bounded $L^2_x$ norm.  
Then after replacing it with some subsequence, 
there exist $K\in\{0,1,2\dots,\I\}$ and, for each integer $j\in[0,K)$, $\fy^j \in L^2(\R^d)$ and $\{(t_n^j,x_n^j)\}_{n\in\N}\subset \R\times\R^d$ satisfying the following. Define $\V v_n^j$ and $\V w_n^k$ for each $j<k\le K$ by
\EQ{
  \V{v}_n^j = e^{\iD(t-t_n^j)}\fy^j(x-x_n^j), \pq \V{v}_n = \sum_{j=0}^{k-1} \V{v}_n^j + \V w_n^k,}
then for any $s<-d/2$, we have
\EQ{ 
 \lim_{k\to K} \limsup_{n\to\I} \|\V w_n^k\|_{L^\I(\R;B^s_{\I,1}(\R^d))}=0,}
and for any $\mu\in\MP$, any $l<j<k\le K$ and any $t\in\R$, 
\EQ{ \label{L2 orth M}
 \lim_{n\to\I} \LR{\mu \V v_n^l|\mu \V v_n^j}_{L^2_x}^2 = 0
 = \lim_{n\to\I} \LR{\mu \V v_n^j|\mu \V w_n^k}_{L^2_x},}
\EQ{ \label{orth sub}
 \lim_{n\to\I} |t_n^j-t_n^k| + |x_n^j-x_n^k| = \I.}
\end{cor}

The orthogonality holds also for the nonlinear energy, which implies that the decomposition is closed in $\ti\K^+$. 
Recall the vector notation for the energy given in Section \ref{rdc 1st}. 
We will also use the following estimates for $1<p<\I$, 
\EQ{ \label{diff conc D}
 \pt \|[|\na|-\LR{\na}_n]\fy\|_{L^p_x} \lec h_n\|\LR{\na/h_n}^{-1}\fy\|_{L^p_x}, 
 \pr \|[|\na|^{-1}-\LR{\na}_n^{-1}]\fy\|_{L^p_x} \lec \|\LR{\na/h_n}^{-2}|\na|^{-1}\fy\|_{L^p_x},}
which hold uniformly for $0<h_n\le 1$, by Mihlin's theorem on Fourier multipliers.
\begin{lem} \label{F orth}
Assume that $f$ satisfies \eqref{asm f}. Let $\V v_n$ be a sequence of free Klein-Gordon solutions satisfying $\V v_n(0)\in\ti\K^+$ and $\limsup_{n\to\I} \ti E(\V v_n(0)) < m$. 
Let $\V v_n=\sum_{j<k}\V v_n^j + \V w_n^k$ be the linear profile decomposition given by Lemma \ref{LPD crit}. Except for the $H^1$ critical case \eqref{f crit}, it may be given by Lemma \ref{LPD sub} too. 
Then we have $\V v_n^j(0)\in\ti\K^+$ for large $n$ and all $j<K$, and 
\EQ{ \label{E decop}
 \lim_{k\to K}\limsup_{n\to\I} \Bigl|\ti E(\V v_n(0)) - \sum_{j<k} \ti E(\V v_n^j(0)) - \ti E(\V w_n^k(0)) \Bigr| = 0.}
Moreover we have for all $j<K$ 
\EQ{
 0\le \liminf_{n\to\I} \ti E(\V v_n^j(0)) \le \limsup_{n\to\I} \ti E(\V v_n^j(0)) \le \limsup_{n\to\I} \ti E(\V v_n(0)),}
where the last inequality becomes equality only if $K=1$ and $\V w_n^1\to 0$ in $L^\I_t L^2_x$. 
\end{lem}
\begin{proof}
First we see that in the exponential case \eqref{f exp}, all the profiles and remainders are in the subcritical regime. Since $\V v_n(0)\in\ti\K^+$, Lemma \ref{E bd exp} implies 
\EQ{
 \|\na\LR{\na}^{-1}\Re\V v_n(0)\|_{L^2_x}^2 + \|\Im\V v_n(0)\|_{L^2_x}^2 < 2m \le 4\pi/\ka_0.}
For any $(\th_0,\dots,\th_k)\in\C^{1+k}$ satisfying $\|\th\|_{L^\I}=\max_j|\th_j|\le 1$, let 
\EQ{
 v_n^\th = \sum_{j<k}\th_j v_n^j + \th_k w_n^k.}
Then choosing $\mu=|\na|\LR{\na}^{-1}\in\MP$ in \eqref{L2 orth M}, we get 
\EQ{
 \limsup_{n\to\I}\sup_{t\in\R} \|\na v_n^\th\|_{L^2_x}^2 
 \le \limsup_{n\to\I} \|\na\LR{\na}^{-1}\V v_n\|_{L^2_x}^2=:M < 4\pi/\ka_0.}
Hence there exist $\ka>\ka_0$ and $q\in(1,2)$ such that $q\ka M < 4\pi$. 

Now we start proving \eqref{E decop} in all the cases. Since the linear version of \eqref{E decop} is given by Lemma \ref{LPD crit}, it suffices to show the orthogonality in $F$, i.e. 
\EQ{ \label{F decop}
  \lim_{k\to K}\limsup_{n\to\I} \Bigl|F(v_n(0)) - \sum_{j<k} F(v_n^j(0)) - F(w_n^k(0))\Bigr| = 0.}
For this we may neglect $w_n^k$, because by the decay in $B^{1-d/2}_{\I,\I}$ and interpolation with the $H^1$ bound we have
\EQ{
 \lim_{k\to K}\limsup_{n\to\I}\|w_n^k(0)\|_{L^{p}_x} = 0 \pq(2<p\le 2^\star).}
In the exponential case, we deal with it as follows. Let $v_n^{<k+\th}=v_n-(1-\th)w_n^k$ for $0\le\th\le 1$. Using the H\"older and Trudinger-Moser inequalities, we get 
\EQ{
 \pt|F(v_n)-F(v_n^{<k})| \le \int_0^1\int |f'(v_n^{<k+\th})w_n^k| dxd\th
 \pr\le\int_0^1d\th \|e^{q\ka|v_n^{<k+\th}|^2}-1\|_{L^1_x}^{1/q} \|w_n^k\|_{L^{q'}_x}
 \pn\le \int_0^1d\th \left[\frac{\|v_n^{<k+\th}\|_{L^2_x}^2}{4\pi-q\ka M}\right]^{1/q}\|w_n^k\|_{L^{q'}_x}.}
In the subcritical/exponential cases, it suffices to have the decay in $B^s_{\I,1}$ for all $s<1-d/2$, which is given by Lemma \ref{LPD sub}. 
Thus in any case we may replace $v_n(0)$ by $v_n^{<k}(0)$ in \eqref{F decop}. 

Next we may discard those $j$ for which $\tau_n^j=-t_n^j/h_n^j\to\pm\I$, since for any $p\in(2,2^\star]$ satisfying $1/p=1/2-s/d$ with $s\in(0,1]$, we have 
\EQ{
 \|v_n^j(0)\|_{L^{p}_x} \lec \|e^{-i\LR{\na}_n^j \tau_n^j}|\na|^{-s}\fy^j\|_{L^{p}_x} \to 0\pq (n\to\I),}
by the decay of $e^{i\LR{\na}_n^jt}$ in $\S\to L^p$ as $|t|\to\I$, which is uniform in $n$, and the Sobolev embedding $\dot H^s_x\subset L^p_x$. 

So extracting a subsequence, we may assume that $\tau_n^j\to\exists\tau_\I^j\in\R$ for all $j$. Let 
\EQ{
 \psi^j := \Re e^{-i\LR{\na}_\I^j\tau_\I^j}\fy^j \in L^2_x(\R^d)}
Then $v_n^j(0)-\LR{\na}^{-1}T_n^j\psi^j \to 0$ strongly in $H^1_x$, thus \eqref{F decop} has been reduced to 
\EQ{ \label{F decop red}
 |F(\sum_{j<k}\LR{\na}^{-1}T_n^j\psi^j)-\sum_{j<k}F(\LR{\na}^{-1}T_n^j\psi^j)| \to 0.}

In the subcritical/exponential cases, 
if $h_n^j\to 0$ then $\LR{\na}^{-1}T_n^j\psi^j\to 0$ strongly in $L^p_x$ for $2\le p<2^\star$, so it can be neglected. Hence we may assume that $h_n^j\equiv 1$. Then each $T_n^j\LR{\na}^{-1}\psi^j$ is getting away from the others as $n\to\I$, and so \eqref{F decop red} follows. 

In the critical case, if $h_n^j\to 0$ then we have by \eqref{diff conc D}, 
\EQ{
 \|\LR{\na}^{-1}T_n^j\psi^j-h_n^jT_n^j|\na|^{-1}\psi^j\|_{L^{2^\star}_x}
 \lec \|\LR{\na/h_n^j}^{-2}|\na|^{-1}\psi^j\|_{L^{2^\star}_x}\to 0.}
Hence we may replace $\LR{\na}^{-1}T_n^j\psi^j$ in \eqref{F decop red} by $h_n^jT_n^j\hat\psi^j$ for some $\hat\psi^j\in L^{2^\star}$, including the case $h_n^j\equiv 1$.  
Then we may further replace each $\hat\psi^j$ by 
\EQ{
 \check\psi_n^j(x):=\hat\psi^j(x)\times\CAS{0 &\text{$\exists l<j$ s.t. $h_n^l<h_n^j$ and $(x-x_n^{j,l})/h_n^{j,l}\in\supp\hat\psi^l$},\\
 1 &\text{otherwise},}}
where $(x_n^{j,l},h_n^{j,l})$ is defined in \eqref{def njl}, 
because \eqref{orth} after the above reduction implies either $h_n^{j,l}\to 0$ or $|x_n^{j,l}|\to\I$, and so $\check\psi^j\to\hat\psi^j$ at almost every $x\in\R^d$ as $n\to\I$, and strongly in $L^{2^\star}_x$ by the dominated convergence theorem. Now the decomposition is trivial 
\EQ{
 F(\sum_{j<k}h_n^jT_n^j\check\psi^j_n) = \sum_{j<k}F(h_n^jT_n^j\check\psi_n^j),}
by the support property of $\check\psi_n^j$. Thus we have obtained \eqref{F decop} and \eqref{E decop}. 

By exactly the same argument, we obtain also
\EQ{
 \lim_{k\to K}\limsup_{n\to\I}\Bigl|\ti K_{\al,\be}(\V v_n(0))-\sum_{j<k}\ti K_{\al,\be}(\V v_n^j(0)) - \ti K_{\al,\be}(\V w_n^k(0))\Bigr| = 0.} 
The remaining conclusions follow from the next lemma. 
\end{proof}

\begin{lem}[Decomposition in $\tilde\K^+$] \label{Korth}
Assume that $f$ satisfies \eqref{asm f}. Let $k\in\N$ and $\fy_0,\dots,\fy_k\in H^1(\R^d)$. Assume that 
\EQ{
 \pt \ti E(\sum_{j=0}^k \fy_j) \le m-\de,
 \pq \ti K_{\al,\be}(\sum_{j=0}^k \fy_j) \ge -\e, 
 \pr \ti E(\sum_{j=0}^k \fy_j) \ge \sum_{j=0}^k \ti E(\fy_j) -\e,
 \pq \ti K_{\al,\be}(\sum_{j=0}^k \fy_j) \le \sum_{j=0}^k \ti K_{\al,\be}(\fy_j)+\e,}
for some $(\al,\be)$ in \eqref{range albe} and some $\de,\e>0$ satisfying $\e(1+2/\bar\mu)<\de$. Then $\ti\fy_j\in\ti\K^+$ for all $j=0,\dots,k$, i.e. $0\le \ti E(\fy_j)<m$ and $\ti K_{\al,\be}(\fy_j)\ge 0$ for all $(\al,\be)$ in \eqref{range albe}. 
\end{lem}
\begin{proof}
Let $\psi_j=\Re\LR{\na}^{-1}\fy_j$ and suppose that $\ti K(\fy_l)<0$ for some $l$. Then $K(\psi_l)\le \ti K(\fy_l)<0$ and so $H(\psi_l)\ge m$. 
Since $H$ is non-negative, 
\EQ{
 m \pn\le \sum_{j=0}^k H(\psi_j) \pt\le \sum_{j=0}^k[H(\psi_j)+H^Q(\Im\LR{\na}^{-1}\fy_j)]= \sum_{j=0}^k [\ti E(\fy_j)-\ti K(\fy_j)/\bar\mu] 
 \pr\le \ti E(\sum_{j=0}^k \fy_j) - \ti K(\sum_{j=0}^k \fy_j)/\bar\mu + \e(1+1/\bar\mu) < m,}
where $H^Q$ denotes the quadratic part of $H$. 
Hence $K(\psi_j)\ge 0$ for all $j$, and so $\ti E(\fy_j)\ge J(\psi_j)=H(\psi_j)+K(\psi_j)/\bar\mu\ge 0$. 
\end{proof}

\subsection{Nonlinear profile decomposition}
The next step is to construct a similar decomposition for the nonlinear solutions with the same initial data. 

First we construct a nonlinear profile corresponding to a free concentrating wave. Let $\V{v}_n$ be a free concentrating wave for a sequence $(t_n,x_n,h_n)\in\R\times\R^d\times(0,1]$, 
\EQ{
 (i\p_t + \Da)\V{v}_n = 0, \pq \V{v}_n(t_n) = T_n \psi, \pq \psi(x)\in L^2,}
satisfying $\V{v}_n(0)\in\ti\K^+$. 
Here we use Lemma \ref{LPD crit} only in the $H^1$ critical case, and Lemma \ref{LPD sub} in the subcritical/exponential cases. Hence $h_n\to 0$ can happen only in the critical case, otherwise $h_n\equiv 1$. 
Let $u_n$ be the nonlinear solution with the same initial data 
\EQ{
 (i\p_t + \Da)\V{u}_n = f'(u_n), \pq \V{u}_n(0)=\V{v}_n(0) \in\ti\K^+,}
which may be local in time. 
Next we define $\V{V}_n$ and $\V{U}_n$ by undoing the transforms 
\EQ{
 \V{v}_n = T_n \V{V}_n((t-t_n)/h_n), \pq \V{u}_n = T_n \V{U}_n((t-t_n)/h_n).}
Then they satisfy the rescaled equations 
\EQ{
 \V{V}_n = e^{it\Da_n}\psi, \pq \V{U}_n = \V{V}_n -i \int_{\tau_n}^t e^{i(t-s)\Da_n}f'(\Re\Da_n^{-1}\V U_n)ds,}
where $\tau_n=-t_n/h_n$. 
Extracting a subsequence, we may assume convergence 
\EQ{
 h_n\to \exists h_\I\in[0,1], \pq \tau_n\to\exists \tau_\I\in[-\I,\I].}
Then the limit equations are naturally given by 
\EQ{ \label{NP eq}
  \V{V}_\I = e^{it\Da_\I}\psi, \pq \V{U}_\I = \V{V}_\I -i \int_{\tau_\I}^t e^{i(t-s)\Da_\I}f'(\hat U_\I)ds,}
where $\hat U_\I$ is defined by 
\EQ{ \label{def hat U}
 \hat U_\I := \Re\LR{\na}_\I^{-1}\V U_\I 
 = \CAS{\Re\LR{\na}^{-1}\V U_\I &(h_\I=1), \\ 
  \Re|\na|^{-1}\V U_\I &(h_\I=0).}}
The unique existence of a local solution $\V{U}_\I$ around $t=\tau_\I$ is known in all cases, including $h_\I=0$ and $\tau_\I=\pm\I$ (the latter corresponding to the existence of the wave operators), by using the standard iteration with the Strichartz estimate. 
In the exponential case, it requires that $\V{U}_\I$ is in the subcritical regime in the Trudinger-Moser inequality. 
It is guaranteed by Lemma \ref{F orth}, because $\V V_\I(t)\in\ti\K^+$ for $t$ close to $\tau_\I$, and so $\V U_\I(t)\in\ti\K^+$ for all $t$ in its existence interval. 

$\V U_\I$ on the maximal existence interval is called the \emf{nonlinear profile} associated with the free concentrating wave $\V{v}_n$. 
The \emf{nonlinear concentrating wave} $\V u_{(n)}$ associated with $\V{v}_n$ is defined by
\EQ{ \label{def u(n)}
 \V{u}_{(n)} = T_{n} \V{U}_\I((t-t_n)/h_n).}
If $h_\I=1$ then $u_{(n)}$ solves NLKG. 
If $h_\I=0$ then it solves
\EQ{ \label{eq waveNP}
 (\p_t^2 - \De + 1)u_{(n)} \pt= (i\p_t+\LR{\na})\V u_{(n)}
 \pr= (\LR{\na}-|\na|)\V u_{(n)} + f'(|\na|^{-1}\LR{\na}u_{(n)}).}
The existence time of $u_{(n)}$ may be finite and even go to $0$, but at least we have
\EQ{
 \pt\|\V{u}_n(0)-\V{u}_{(n)}(0)\|_{L^2_x} =\|\V{V}_n(\tau_n)-\V{U}_\I(\tau_n)\|_{L^2_x}
 \prq\le \|\V{V}_n(\tau_n)-\V{V}_\I(\tau_n)\|_{L^2_x}
  + \|\V{V}_\I(\tau_n)-\V{U}_\I(\tau_n)\|_{L^2_x} \to 0.}

Let $u_n$ be a sequence of (local) solutions of NLKG in $\K^+$ around $t=0$, and let $v_n$ be the sequence of the free solutions with the same initial data. We consider the linear profile decomposition given by Lemma \ref{LPD crit} or \ref{LPD sub}, 
\EQ{
 \V v_n = \sum_{j=0}^{k-1} \V{v}_n^j + \V w_n^k, \pq \V{v}_n^j = e^{\iD(t-t_n^j)}T_n^j \fy^j.}
With each free concentrating wave $\{\V v_n^j\}_{n\in\N}$, we associate the nonlinear concentrating wave $\{\V u_{(n)}^j\}_{n\in\N}$. 
A \emf{nonlinear profile decomposition} of $u_n$ is given by 
\EQ{ \label{def NPD}
 \V u^{<k}_{(n)} := \sum_{j=0}^{k-1} \V u_{(n)}^j.}
We are going to prove that $\V u_{(n)}^{<k}$ is a good approximation for $\V u_n$, provided that each nonlinear profile has finite global Strichartz norm (in Lemma \ref{NP Stz}). 
Now we define the Strichartz norms for the profile decomposition, using the notation in Section \ref{Stz notation}. Let $ST$ and $ST^*$ be the function spaces on $\R^{1+d}$ defined by
\EQ{ \label{def ST}
 ST=[W]_2\cap[K]_2, \pq ST^* = [W^{*(1)}]_2+[K^{*(1)}]_2+L^1_t L^2_x,}
where the exponents $W$ and $K$ as well as their duals are as defined in \eqref{def WK} and \eqref{exp change}. The Strichartz norm for the nonlinear profile depends on the scaling $h_\I^\diamondsuit$ for any suffix $\diamondsuit$;
\EQ{ \label{def STI}
  ST_\I^\diamondsuit := \CAS{[W]_2\cap[K]_2 &(h_\I^\diamondsuit=1),\\ [W]_2^\bul &(h_\I^\diamondsuit=0).}}
In other words, we take the scaling invariant part if $h_n^\diamondsuit\to+0$, which can happen only in the $H^1$ critical case. The following estimate is convenient to treat the concentrating case: 
For any $S\in[0,1]\times[0,1/2]\times[0,1]$ we have 
\EQ{ \label{hat U bd u}
 \|u_{(n)}\|_{[S]_2(\R)} \lec (h_n)^{1-\reg^0(S)}\|\hat U_\I\|_{[S]_2^\bul(\R)},}
where $\hat U_\I$ is as defined in \eqref{def hat U}. 
Indeed, using $\dot B^0_{p,2}\subset L^p$ with $p=1/S_2\ge 2$ in the lower frequencies, we have 
\EQ{
 \|u_{(n)}\|_{[S]_2} \pt\lec \||\na|^{-S_3}\LR{\na}^{S_3}u_{(n)}\|_{[S]_2^\bul}
 \pr\sim (h_n)^{1-\reg^0(S)}\|\Re|\na|^{-S_3}\LR{\na}_n^{S_3-1}\V U_\I^j\|_{[S]_2^\bul}
 \pn\lec (h_n)^{1-\reg^0(S)}\|\hat U_\I^j\|_{[S]_2^\bul}.}

Concerning the orthogonality in the Strichartz norms, we have 
\begin{lem}
Assume that $f$ satisfies \eqref{asm f}. 
Suppose that in the nonlinear profile decomposition \eqref{def NPD} we have
\EQ{
 \|\hat U_\I^j\|_{ST_\I^j(\R)}+\|\V U_\I^j\|_{L^\I_tL^2_x(\R)}<\I} 
for each $j<K$. Then we have for any finite interval $I$, any $j<K$ and any $k\le K$,  
\EQ{ \label{ST lim}
 \limsup_{n\to\I} \|u_{(n)}^j\|_{ST(I)} \lec \|\hat U_\I^j\|_{ST_\I^j(\R)},}
\EQ{ \label{ST orth} 
 \limsup_{n\to\I} \|u_{(n)}^{<k}\|_{ST(I)}^2 \lec \limsup_{n\to\I}\sum_{j<k}\|u_{(n)}^j\|_{ST(I)}^2,}
where the implicit constants do not depend on $I$, $j$ or $k$. We have also
\EQ{ \label{f ST decop}
 \lim_{n\to\I} \|f'(u_{(n)}^{<k})-\sum_{j<k}f'((\LR{\na}_\I^j)^{-1}\LR{\na}u_{(n)}^j)\|_{ST^*(I)} = 0.}
\end{lem}
\begin{proof}
First note that if $h_\I^j=1$ then $u_{(n)}^j$ is just a sequence of space-time translations of $\hat U_\I^j$. In particular, \eqref{ST lim} is trivial in that case. 

Next we prove \eqref{ST lim} in the case $h_\I^j=0$, which is only in the $H^1$ critical case.  For the moment we drop the superscript $j$. 
For the $[W]_2$ part, \eqref{hat U bd u} gives us 
\EQ{
 \|u_{(n)}\|_{[W]_2(I)} \pt\lec \|\hat U_\I\|_{[W]_2^\bul(\R)}
 =\|\hat U_\I\|_{ST_\I^j(\R)}.} 
For the $[K]_2$ part, let $V$ be the following interpolation between $H$ and $W$  
\EQ{ \label{def V}
 V := \frac{1}{d+2} H + \frac{d+1}{d+2} W = K + \frac{\paren{-1,0,1}}{2(d+2)}.}
Then using H\"older in $t$ and \eqref{hat U bd u} together with $\reg^0(K)=(d+1)/(d+2)$, we get 
\EQ{ \label{vanish K}
 \|u_{(n)}\|_{[K]_2(I)} 
 \pt\lec \|u_{(n)}\|_{[V^{\frac 12}]_2(I)}|I|^{\frac{1}{2(d+2)}}
 \pn\lec (h_n)^{\frac{1}{2(d+2)}}\|\hat U_\I\|_{[V]_2^\bul(\R)}|I|^{\frac{1}{2(d+2)}} \to 0,}
as $n\to\I$. 
Thus we have proved \eqref{ST lim}. 

Next we prove \eqref{ST orth} in the subcritical/exponential cases. 
Define $\hat U_{\I,R}^j$, $u_{(n),R}^j$ for $R\gg 1$ and $u_{(n),R}^{<k}$ by 
\EQ{
 \hat U_{\I,R}^j = \chi_R(t,x) \hat U_\I^j, \pq u_{(n),R}^j = T_n^j \hat U_{\I,R}^j(t-t_n^j),  \pq u_{(n),R}^{<k} = \sum_{j<k} u_{(n),R}^j,}
where $\chi_R$ is the cut-off defined in \eqref{def chiR}. Then we have
\EQ{
 \|u_{(n)}^{<k}-u_{(n),R}^{<k}\|_{ST(\R)}
 \le \sum_{j<k} \|(1-\chi_R(t,x))\hat U_\I^j\|_{ST(\R)} \to 0,\pq(R\to+0)}
so we may replace $u_{(n)}^{<k}$ by $u_{(n),R}^{<k}$. 
Let $\de_m^l$ denote the difference operator
\EQ{ \label{def diff}
 \de_m^l \fy(x) = \fy(x-2^{-m}e_l)-\fy(x),}
where $e_l$ denotes the $l$-th unit vector in $\R^d$. 
Each Besov norm in $ST$ is equivalent to 
\EQ{ \label{Besov diff}
 \sum_{l=1}^d \Bigl\|\sum_{j<k}2^{sm} \de_m^l u_{(n),R}^{j} \Bigr\|_{L^p_t \ell^2_{m\ge 0} L^q_x} + \|\sum_{j<k}u_{(n),R}^{j}\|_{L^p_t L^q_x},}
where $(1/p,1/q,s)=W$ or $K$. \eqref{orth sub} implies that each $\supp u_{(n),R}^j$ is away from the others at least by distance $2$ for large $n$, and then $\supp\de_m^l u_{(n),R}^j$ are also disjoint for $j<k$ at each $l,m$. Hence the first norm in \eqref{Besov diff} equals
\EQ{ \label{Minkow}
 \|2^{sm}\de_m^l u_{(n),R}^{j}\|_{L^p_t \ell^2_{m\ge 0} L^q_x \ell^2_{j<k}}
 \pt\le \|2^{sm}\de_m^l u_{(n),R}^{j}\|_{\ell^2_{j<k} L^p_t \ell^2_{m\ge 0} L^q_x} 
 \pn\lec \|u_{(n),R}^j\|_{\ell^2_{j<k} L^p_t B^s_{q,2}},}
where the first inequality is by Minkowski. 
Thus we have obtained \eqref{ST orth} in the subcritical/exponential cases. 

Next we prove \eqref{ST orth} in the $H^1$ critical case. 
For the nonlinear concentrating waves with $h_\I^j=1$, the above argument works. For those with $h_\I^j=0$, the $K$ component is vanishing by \eqref{vanish K}. 
Hence it suffices to estimate $[W]_2$ in the case all $h_n^j$ tend to $0$ as $j\to\I$. Using that $W_3=1/2\in(0,1)$, we have
\EQ{ \label{u by check}
 \|u_{(n)}^{<k}\|_{[W]_2(\R)} \pn\lec\||\na|^{-1}\LR{\na}u_{(n)}^{<k}\|_{[W]_2^\bul(\R)}
 \pt= \|\Re |\na|^{-1}\V u_{(n)}^{<k}\|_{[W]_2^\bul(\R)}
 \pr\sim \|\sum_{j<k}\check u_{n,m}^{j,l}\|_{L^p_t \ell^2_{m\in\Z} L^q_x},}
where we put $(1/p,1/q,s)=W$ and 
\EQ{
 \check u_{n,m}^{j,l}:= 2^{sm}\de_m^l h_n^j T_n^j \hat U_\I^j((t-t_n^j)/h_n^j),}where $\de_m^l$ is the difference operator defined in \eqref{def diff}. 
For $R\gg 1$, let 
\EQ{
 \check u_{n,m,R}^{j,l}(t,x):=\CAS{\chi_{h_n^jR}(t-t_n^j,x-x_n^j)\check u_{n,m}^{j,l}(t,x) &(|m-\log_2h_n^j|\le R)\\
 0 &(|m-\log_2h_n^j|>R),}}
where $\chi_{*}$ is as in \eqref{def chiR}. 
Then by the same computation as for \eqref{hat U bd u}, we have
\EQ{
 \|\check u_{n,m}^{j,l}-\check u_{n,m,R}^{j,l}\|_{L^p_t \ell^2_{m\in\Z} L^q_x}
 \lec \|2^{sm}\de_m^l\hat U_\I^j\|_{L^p_{t} \ell^2_m L^q_x(|t|+|m|+|x|>R)} \to 0,}
as $R\to\I$ uniformly in $n$. Hence we may replace $\check u_{n,m}^{j,l}$ by $\check u_{n,m,R}^{j,l}$ in \eqref{u by check}. 
The orthogonality \eqref{orth} implies that $\{\supp_{(t,m,x)}\check u_{n,m,R}^{j,l}\}_{j<k}$ becomes mutually disjoint for large $n$. 
Then arguing as in \eqref{Minkow}, we obtain \eqref{ST orth}. 

To prove \eqref{f ST decop} in the subcritical/exponential cases is easier than \eqref{ST orth}, because after the smooth cut-off, we have for large $n$ 
\EQ{
 f'(u_{(n),R}^{<k})=\sum_{j<k}f'(u_{(n),R}^j).}
Note that the $u_{(n)}^j\in ST$ implies that the full Strichartz norms are finite by Lemma \ref{Stz extend}. The error for $f'(u_{(n)}^{<k})$ coming from the cut-off is small in $ST^*$ by \eqref{fS diff lD}--\eqref{fL diff exp} if $d\le 4$. When $d\ge 5$, the difference estimates in the proof of Lemma \ref{PerturLem} are not sufficient because they control only the exotic norm $Y$. In order to estimate the difference in the admissbile dual norm $ST^*(I)$, we introduce the following new exponents:
\EQ{
 H_\e:=(\e^2,\frac{1-\e}{2},0), \pq W_\e:=W-p_2\e(d,-1,0), \pq M^\sharp_\e:=M^\sharp+\e(d,-1,0),}
where $W$ and $M^\sharp$ were defined in \eqref{def WK} and \eqref{def MsharpX}, and $\e\in(0,p_1)$ is fixed small enough to have 
\EQ{ \label{def norm_eps}
 \pt \str^0(H_\e),\str^0(M^\sharp_\e),\str^0(W_\e)<0, \pq \reg^0(H_\e)<1,
 \pr \reg^0(W_\e)=\reg^0(W)=1,\pq \reg^0(M^\sharp_\e)=\reg^0(M^\sharp)\le 1,
 \pr W_\e+p_2M^\sharp_\e=W+p_2M^\sharp=W^{*(1)}.}
Then we have for any $u$ and $v$, 
\EQ{ \label{fS diff ad}
 \|f_S'(u)-f_S'(v)\|_{L^1_tL^2_x(I)} \lec |I|^{1-\e^2}\|u-v\|_{[H_\e]_0(I)}(\|u\|_{L^\I_tL^2_x(I)}+\|v\|_{L^\I_tL^2_x(I)})^\e,}
because $|f_S'(u)-f_S'(v)|\lec |u-v|(|u|+|v|)^\e$. For large $u$, we have if $p_2\ge 1$, 
\EQ{ \label{fL diff adLip}
 \pt\|f_L'(u)-f_L'(v)\|_{[W^{*(1)}]_2} 
 \pr\lec \|u\|_{[M^\sharp_\e]_0}^{p_2}\|u-v\|_{[W_\e]_2} + \|u-v\|_{[M^\sharp_\e]_0}(\|u\|_{[M^\sharp_\e]_0}+\|v\|_{[M^\sharp_\e]_0})^{p_2-1}\|v\|_{[W_\e]_2},}
and if $p_2<1$, 
\EQ{ \label{fL diff adHol}
 \|f_L'(u)-f_L'(v)\|_{[W_\e^{*(1)}]_2} \lec \|u\|_{[M_\e^\sharp]_0}^{p_2}\|u-v\|_{[W_\e]_2} + \|u-v\|_{[M_\e^\sharp]_0}^{p_2}\|v\|_{[W_\e]_2}.}
The latter estimate is not Lipschitz in $u-v$, but sufficient for our purpose here.\footnote{The situation is different from the long-time iteration in the previous section, where we needed the exotic Strichartz estimate in order to get the Lipschitz estimate for the iteration along the numerous time intervals.} 
Thus we obtain \eqref{f ST decop} in the subcritical/exponential cases.

It remains to prove \eqref{f ST decop} in the $H^1$ critical case, where we need further cut-off to get a disjoint sum. 
First we see that each $u_{(n)}^j$ in $u_{(n)}^{<k}$ may be replaced with 
\EQ{
 u_{\LR{n}}^j :=(\LR{\na}_\I^j)^{-1}\LR{\na}u_{(n)}^j= h_n^j T_n^j\hat U_\I^j((t-t_n^j)/h_n^j).}
For the moment we drop the superscript $j$. Let $p_2=4/(d-2)$ and $h_\I=0$. 
If $d\le 4$, then we have by using \eqref{fL diff lD} and \eqref{diff conc D}
\EQ{
 \pn\|f'(u_{(n)})-f'(u_{\LR{n}})\|_{L^1_tL^2_x(\R)}
 \pt\lec \|u_{\LR{n}}\|_{[L]_0(\R)}^{p_2}\|u_{(n)}-u_{\LR{n}}\|_{[L]_0(\R)}
 \pr\sim \|\hat U_\I\|_{[L]_0(\R)}^{p_2}\|[|\na|\LR{\na}_n^{-1}-1]\hat U_\I\|_{[L]_0(\R)}
 \pr\lec \|\hat U_\I\|_{[L]_0(\R)}^{p_2}\|\LR{\na/h_n}^{-2}\hat U_\I\|_{[L]_0(\R)}\to 0,}
since $\hat U_\I\in[H]_2^\bul\cap[W]_2^\bul\subset[L]_0$ by the homogeneous version of Lemma \ref{lem:trinterpole}(1). 

If $d\ge 5$, we introduce a new exponent 
\EQ{ \label{def G}
 G:=\frac{d-2}{d+2}\paren{\frac{1}{d+1},\frac{d+3}{2(d+1)},0}.}
Then $\reg^0(G)=1$, $\str^0(G)<0$ and 
\EQ{
 (2^\star-1)G = W^{*(1)} - \frac{(1,0,1)}{2}.} 
Hence 
\EQ{ \label{u(n)2U}
 \|f'(u_{(n)})-f'(u_{\LR{n}})\|_{[W^{*(1)}]_2(I)}
 \pt\lec \|f'(u_{(n)})-f'(u_{\LR{n}})\|_{[W^{*(1)}]_2^\bul(\R)} 
 \prq+|I|^{1/2} \|f'(u_{(n)})-f'(u_{\LR{n}})\|_{[(2^\star-1)G]_0(I)},}
where the first term on the right is dominated by (the homogeneous version of \eqref{fL diff adLip}--\eqref{fL diff adHol})
\EQ{
 \pt \|u_{\LR{n}}\|^{p_2}_{[M^\sharp_\e]_0(\R)} \|u_{(n)}-u_{\LR{n}}\|_{[W_\e]_2^\bul(\R)}
 \pn+ \|u_{(n)}-u_{\LR{n}}\|^{\th}_{[M^\sharp_\e]_0(\R)} \|(u_{\LR{n}},u_{(n)})\|_{[W_\e]_2^\bul(\R)}^{p_2-\th}
 \pr\lec \|\hat U_\I\|^{p_2}_{[M^\sharp_\e]_0(\R)} \|\LR{\na/h_n}^{-2}\hat U_\I\|_{[W_\e]_2^\bul(\R)}
 \pn+ \|\LR{\na/h_n}^{-2}\hat U_\I\|^{\th}_{[M^\sharp_\e]_0(\R)} \|\hat U_\I\|_{[W_\e]_2^\bul(\R)}^{p_2-\th},}
where $\th:=\min(p_2,1)$. 
The right hand side goes to $0$, since $\hat U_\I\in[H]_2^\bul\cap[W_\e]_2^\bul\subset[M^\sharp_\e]_0$ by the homogeneous version of Lemma \ref{lem:trinterpole}(1). 
Similarly, the last term in \eqref{u(n)2U} is bounded by 
\EQ{
 \|u_{\LR{n}}\|_{[G]_0(\R)}^{p_2} \|u_{(n)}-u_{\LR{n}}\|_{[G]_0(\R)}
 \sim \|\hat U_\I\|_{[G]_0(\R)}^{p_2} \|\LR{\na/h_n}^{-2}\hat U_\I\|_{[G]_0(\R)}\to 0.}
Thus it suffices to show 
\EQ{
 \|f'(\sum_{j<k}u_{\LR{n}}^{j})-\sum_{j<k}f'(u_{\LR{n}}^j)\|_{ST^*(I)}\to 0.}
Now we define $\hat U_{n,R}^j$ for any $R\gg 1$ by 
\EQ{
 \hat U_{n,R}^j(t,x) \pt= \chi_R(t,x) \hat U_\I^j(t,x)
 \pr\times \prod\{(1-\chi_{h_n^{j,l}R})(t-t_n^{j,l},x-x_n^{j,l}) \mid 1\le l<k,\ h_n^lR<h_n^j\},}
where $\chi_R$ and $(t_n^{j,l},x_n^{j,l},h_n^{j,l})$ are as defined respectively in \eqref{def chiR} and \eqref{def njl}. 
Then it is uniformly bounded in $[H]_2^\bul(\R)\cap[W]_2^\bul(\R)$, and 
$\hat U_{n,R}^j\to\chi_R \hat U_\I^j$ in $[M^\sharp]_0(\R)$ as $n\to\I$, because either $h_n^{j,l}\to 0$ or $|t_n^{j,l}|+|x_n^{j,l}|\to\I$ by the orthogonality \eqref{orth}. 
Then by the homogeneous version of Lemma \ref{lem:trinterpole}(2), it converges also in $[L]_0(\R)$ (if $d\le 4$), $[W_\e]_2^\bul(\R)$ and $[M_\e^\sharp]_0(\R)$. 
Moreover, we have $\chi_R\hat U_\I^j\to \hat U_\I^j$ as $R\to\I$ in the same spaces. 

Hence we may replace $u_{\LR{n}}^j$ by $u_{\LR{n},R}^j := h_n^jT_n^j \hat U_{n,R}^j((t-t_n^j)/h_n^j)$, 
and then we get the desired result, since $\{\supp_{(t,x)}u_{\LR{n},R}^j\}_{j<k}$ are mutually disjoint for large $n$, and so
\EQ{
 f'(\sum_{j<k} u_{\LR{n},R}^j) = \sum_{j<k} f'(u_{\LR{n},R}^j),}
which concludes the proof of \eqref{f ST decop}. 
\end{proof}

The next lemma is the conclusion of this section. 
\begin{lem} \label{NP Stz}
Assume that $f$ satisfies \eqref{asm f}. Let $u_n$ be a sequence of local solutions of NLKG around $t=0$ in $\K^+$ satisfying $\limsup_{n\to\I}E(u_n)<m$. 
Suppose that in its nonlinear profile decomposition \eqref{def NPD}, every nonlinear profile $\V U_\I^j$ has finite global Strichartz and energy norms, i.e. 
\EQ{
 \|\hat U_\I^j\|_{ST_\I^j(\R)} + \|\V U_\I^j\|_{L^\I_t L^2_x(\R)} <\I.}
Then $u_n$ is bounded for large $n$ in the Strichartz and the energy norms, i.e.
\EQ{
 \limsup_{n\to\I} \|u_n\|_{ST(\R)} + \|\V u_n\|_{L^\I_t L^2_x(\R)} < \I.}
\end{lem}
\begin{proof}
We will apply the perturbation lemma to $u_{(n)}^{<k}+w_n^k$ as an approximate solution. First observe that 
\EQ{ \label{init app}
 \|\V u_n(0)-\V u_{(n)}^{<k}(0)-w_n^k(0)\|_{L^2_x} 
 \pt\le \sum_{j<k}\|\V v_n^j(0)-\V u_{(n)}^{j}(0)\|_{L^2_x} = o(1),}
and
\EQ{ \label{NP E bd}
 \|\V u_n(0)\|_{L^2}^2 \pt= \|\V v_n\|_{L^2_x}^2
 \ge \sum_{j<k} \|\V v_n^j\|_{L^2_x}^2 + o(1)
 \pn= \sum_{j<k} \|\V u_{(n)}^j(0)\|_{L^2_x}^2 + o(1),}
where $o(1)\to 0$ as $n\to\I$. 
Hence except for a finite set $J\subset\N$, the energy of $u_{(n)}^j$ with $j\not\in J$ is smaller than the iteration threshold, which implies  
\EQ{ \label{small prof}
 \|u_{(n)}^j\|_{ST(\R)} \lec \|\V u_{(n)}^j(0)\|_{L^2_x} \pq(j\not\in J).}
Combining \eqref{ST orth}, \eqref{ST lim}, \eqref{small prof} and \eqref{NP E bd}, we obtain for any finite interval $I$, 
\EQ{
 \sup_k \limsup_{n\to\I} \|u_{(n)}^{<k}\|_{ST(I)}^2
 \pt\lec \sum_{j\in J}\|\hat U_{\I}^j\|_{ST_\I^j}^2 + \limsup_{n\to\I} \|\V u_n(0)\|_{L^2_x}^2
 \pn<\I.}
The equation of $u_{(n)}^{<k}$ is given by 
\EQ{
 eq(u_{(n)}^{<k}) = \sum_{j<k}(\LR{\na}-\LR{\na}_\I^j)\V u_{(n)}^j 
 \pn+ f'(u_{(n)}^{<k}) - \sum_{j<k}f'(u_{\LR{n}}^j),}
where $u_{\LR{n}}^j=(\LR{\na}_\I^j)^{-1}\LR{\na}u_{(n)}^j$ as before. 
The nonlinear part goes to $0$ by \eqref{f ST decop}, while the linear part vanishes if $h_\I^j=1$, and is dominated if $h_\I^j=0$ by 
\EQ{
 \|(\LR{\na}-|\na|)\V u_{(n)}^j\|_{L^1_tL^2_x(I)}
 \pt\lec |I|\|\LR{\na}^{-1}\V u_{(n)}^j\|_{L^\I_t L^2_x(\R)}
 \pr\sim |I|\|\LR{\na/h_n^j}^{-1}\V U_\I^j\|_{L^\I_t L^2_x(\R)}\to 0\pq(n\to\I),}
by continuity in $t$ for bounded $t$, and by the scattering of $\hat U_\I^j$ for $|t|\to\I$, which follows from $\|\hat U_\I^j\|_{[W]_2^\bul(\R)}<\I$. 
Hence Lemma \ref{Stz extend} gives for any $1$-admissible $Z$ 
\EQ{
 \sup_k \limsup_{n\to\I} \|u_{(n)}^{<k}\|_{[Z]_2(\R)} <\I.}

On the other hand, by Lemma \ref{lem:trinterpole} we can extend the smallness of $w_n^k$ from $L^\I_t B^{s}_{\I,\I}$ to the other spaces that we need for the nonlinear difference estimates, i.e. $[S]_0$, $[L]_0$, $[X]_2$, $[H_\e]_0$, $[M^\sharp_\e]_0$, and $[W_\e]_2$, depending on $d$ and $f$. 
In addition, in the exponential case \eqref{f exp}, Lemmas \ref{F orth} and \ref{E bd exp} imply that $u_{(n)}^{<k}$ and $w_n^k$ are both in the subcritical regime for the Trudinger-Moser inequality.  
Putting them together with the above bounds on $u_{(n)}^{<k}$ in the nonlinear difference estimates \eqref{fS diff lD}--\eqref{fL diff exp} or \eqref{fS diff ad}--\eqref{fL diff adHol}, we get 
\EQ{ \label{rest in f}
 \lim_{k\to K}\limsup_{n\to\I}\|f'(u_{(n)}^{<k}+w_n^k)-f'(u_{(n)}^{<k})\|_{ST^*(I)}= 0,}
and so 
\EQ{
 \lim_{k\to K}\limsup_{n\to\I}\|eq(u_{(n)}^{<k}+w_n^k)\|_{ST^*(I)}=0.}

Hence for $k$ sufficiently close to $K$ and $n$ large enough, 
the true solution $u_n$ and the approximate solution $u_{(n)}^{<k}+w_n^k$ satisfy all the assumptions of the perturbation Lemma \ref{PerturLem}. 
Hence $u_n$ is bounded in global Strichartz norms for large $n$. 
\end{proof}

\section{Extraction of a critical element}
In this section, we prove that if uniform global Strichartz bound fails strictly below the variational threshold $m$, then we have a global solution in $\K^+$ with infinite Strichartz norm and with the minimal energy, which is called {\it a critical element}. 

Let $E^\star$ be the threshold for the uniform Strichartz bound. 
More precisely, 
\EQ{ \label{def E*}
 E^\star := \sup\{A>0 \mid S(A)<\I\},}
where $S(A)$ denotes the supremum of $\|u\|_{ST(I)}$ for any strong solution $u$ in $\K^+$ on any interval $I$ satisfying $E(u)\le A$. 

The small energy scattering tells us $E^\star>0$, and the presence of the ground state tells us $E^\star\le m$, at least in the subcritical case, and also in the other cases if we allow complex-valued solutions, because the stationary solutions with different masses yield standing wave solutions of the original NLKG. 
Anyway, we are going to prove $E^\star\ge m$ by contradiction. 

We remark that there is an alternative threshold: 
\EQ{
 \pt E^\star_{FS} := \sup \left\{A>0 \middle| 
  \begin{aligned}
   \pt\text{If $u$ is a solution in $\K^+$ of NLKG}
   \pr\text{with $E(u)\le A$, then $\|u\|_{ST(\R)}<\I$ }
  \end{aligned} 
 \right\}.}
Obviously $E^\star\le E^\star_{FS}$. 
Kenig-Merle \cite{KM1} chose this definition. 
The advantage of using $E^\star$ is that $E^\star\ge m$ implies uniform bound on the global Strichartz norms below $m$, which is very important in applications where we want to perturb the equation.  

The next lemma is the conclusion of this section.
\begin{lem} \label{exist CE}
Assume that $f$ satisfies \eqref{asm f}, and 
let $u_n$ be a sequence of solutions of NLKG in $\K^+$ on $I_n\subset\R$ satisfying 
\EQ{
  E(u_n)\to E^\star<m, \pq \|u_n\|_{ST(I_n)} \to \I \pq(n\to\I).}
Then there exists a global solution $u_*$ of NLKG in $\K^+$ satisfying 
\EQ{ 
  E(u_*)=E^\star, \pq \|u_*\|_{ST(\R)}=\I.}
In addition, there are a sequence $(t_n,x_n)\in\R\times\R^d$ and $\fy\in L^2(\R^d)$ such that along some subsequence, 
\EQ{ \label{CE init}
 \|\V u_n(0,x) - e^{-\iD t_n}\fy(x-x_n)\|_{L^2_x} \to 0.}
\end{lem}
We call such a global solution $u_*$ \emf{a critical element}. 
Observe that by the definition of $E^\star$, we can find such a sequence $u_n$, once we have $E^\star<m$. 

\begin{proof}
We can translate $u_n$ in $t$ so that $0\in I_n$ for all $n$. 
Then we consider the linear and nonlinear profile decompositions of $u_n$, using Lemma \ref{LPD crit} in the $H^1$ critical case \eqref{f crit} and Lemma \ref{LPD sub} in the subcritical/exponential cases. 
\EQ{
 \pt e^{\iD t}\V u_n(0)=\sum_{j<k}\V v_n^j + \V w_n^k, \pq \V v_n^j = e^{\iD (t-t_n^j)}T_n^j\fy^j,
 \pr u_{(n)}^{<k} = \sum_{j<k} u_{(n)}^j, \pq \V u_{(n)}^j = T_n^j \V U_\I^j((t-t_n^j)/h_n^j),
 \pr \|\V v_n^j(0)-\V u_{(n)}^j(0)\|_{L^2_x} \to 0 \pq (n\to\I).} 
Lemma \ref{NP Stz} precludes that all the nonlinear profiles $\V U_\I^j$ have finite global Strichartz norm. 
On the other hand, every solution of NLKG in $\K^+$ with energy less than $E^\star$ has global finite Strichartz norm by the definition of $E^\star$. 
Hence by Lemma \ref{F orth} we deduce that there is only one profile i.e. $K=1$, and moreover 
\EQ{
 \ti E(\V u_{(n)}^0)=E^\star, \pq \V u_{(n)}^0(0)\in\ti\K^+, \pq \|\hat U_\I^0\|_{ST_\I^0(\R)}=\I, \pq \lim_{n\to\I}\|\V w_n^1\|_{L^\I_t L^2_x}= 0.}
If $h_n^0\to 0$ in the critical case, then $\hat U_\I^0=|\na|^{-1}\Re\V U_\I^0$ solves the massless equation 
\EQ{
 (\p_t^2-\De)\hat U_\I^0 = f'(\hat U_\I^0),}
and satisfies 
\EQ{
 E^0(\hat U_\I^0)=E^\star<m=J^{(0)}(Q), \pq K^w(\hat U_\I^0(0))\ge 0, \pq \|\hat U_\I^0\|_{[W]_2^\bul}=\I,}
where $Q$ is the massless ground state and $K^w$ is the massless version of $K$.  
However, Kenig-Merle \cite{KM1} has proven that there is no such solution.\footnote{\cite{KM1} is restricted to the dimensions $d\le 5$ for simplicity of the perturbation argument, but the elimination of critical elements works in any higher dimensions.} 
Hence $h_n^0\equiv 1$ in all cases, and we obtain \eqref{CE init}. 

Hence $h_n^0\equiv 1$ in all cases, and we obtain \eqref{CE init}. 

It remains to prove that $\hat U_\I^0=\LR{\na}^{-1}\Re\V U_\I^0$ is a global solution. 
Suppose not. Then we can choose a sequence $t_n\in\R$ approaching the maximal existence time. 
Since the sequence of solutions $\hat U_\I^0(t+t_n)$ satisfies the assumption of this lemma, we may apply the above argument to it. In particular, from \eqref{CE init} we obtain 
\EQ{ \label{init CE}
 \|\V U_\I^0(t_n) - e^{-\iD t_n'}\psi(x-x_n')\|_{L^2_x} \to 0,}
for some $\psi\in L^2_x$ and another sequence $(t_n',x_n')\in\R\times\R^d$. 
Let $\V v:=e^{\iD t}\psi$. Since it is a free solution, for any $\e>0$ there is $\de>0$ such that for any interval $I$ satisfying $|I|\le2\de$, we have $\|\LR{\na}^{-1}\V v\|_{ST(I)}\le\e$, where $ST=[W]_2\cap[K]_2$ as in \eqref{def ST}. Then \eqref{init CE} implies that
\EQ{
 \limsup_{n\to\I} \|\LR{\na}^{-1}e^{\iD t}\V U_\I^0(t_n)\|_{ST(-\de,\de)}\le\e.}
If $\e>0$ is small enough, this implies that the solution $\hat U_\I^0$ exists on $(t_n-\de,t_n+\de)$, by the iteration argument, for large $n$. This contradicts the choice of $t_n$. Hence $\hat U_\I^0$ is global and so a critical element. 
\end{proof}

\section{Extinction of the critical element} 
In this section, we prove that the critical element can not exist by deriving a contradiction from a few properties of it. The main idea follows \cite{KM1,KM2}. 
Let $u_c$ be a critical element given by Lemma \ref{exist CE}. 
Since NLKG is symmetric in $t$, we may assume that $\|u_c\|_{ST(0,\I)}=\I$. 
We call such $u$ a \emf{forward critical element}. 
Note that since the critical element is in $\K^+$, we have $E^Q(u;t)\sim E(u)$ uniformly, by Lemma \ref{E0 equiv K+}. 

\subsection{Compactness}
First we show that the trajectory of a forward critical element is precompact for positive time in the energy space modulo spatial translations. 
\begin{lem} \label{precompact}
Assume that $f$ satisfies \eqref{asm f}, and 
let $u_c$ be a forward critical element. 
Then there exists $c:(0,\I)\to\R^d$ such that the set 
\EQ{ 
  \{(u,\dot u)(t,x-c(t)) \mid 0<t<\infty\} } 
is precompact in $H^1(\R^d)\times L^2(\R^d)$. 
\end{lem}
\begin{proof}
The proof of Kenig-Merle \cite{KM1} can be adapted verbatim, but we give a sketch for the sake of completeness. Recall the convention $u\leftrightarrow \V u$ defined in Section \ref{rdc 1st}. 

It suffices to prove precompactness of $\{\V u(t_n)\}$ in $L^2_x$ for any $t_1,t_2,\dots>0$. 
If $t_n$ converges, then it is trivial from the continuity in $t$. 
Hence we may assume that $t_n\to\I$. 
Applying Lemma \ref{exist CE} to the sequence of solutions $u(t+t_n)$, we get another sequence $(t_n',x_n')\in\R^{1+d}$ and $\fy \in L^2$ such that 
\EQ{ \label{decop for utn}
 \V{u}(t_n,x) - e^{-\iD t_n'} \fy(x-x_n') \to 0 \IN{L^2_x}\pq(n\to\I).}

If $t_n'\to-\I$, then we have 
\EQ{
 \|e^{\iD t}\V{u}(t_n)\|_{ST(0,\I)} = \|e^{\iD t}\fy\|_{ST(-t_n',\I)} + o(1)  \to 0,}
so that we can solve NLKG of $u$ for $t>t_n$ with large $n$ globally by iteration with small Strichartz norms, contradicting its forward criticality. 

If $t_n'\to+\I$, then we have 
\EQ{
 \|e^{\iD t}\V{u}(t_n)\|_{ST(-\I,0)} = \|e^{\iD t}\fy\|_{ST(-\I,-t_n')} + o(1) \to 0,}
so that we can solve NLKG of $u$ for $t<t_n$ with large $n$ with diminishing Strichartz norms, which implies $u=0$ by taking the limit, a contradiction. 

Thus $t_n'$ is precompact, so is $\V{u}(t_n,x+x_n')$ in $L^2_x$ by \eqref{decop for utn}. 
\end{proof}

As a consequence, the energy of $u$ stays within a fixed radius for all positive time, modulo arbitrarily small rest. More precisely, we define the exterior energy by 
\EQ{ \label{def ER}
 E_{R,c}(u;t) = \int_{|x-c|\ge R} |u_t|^2 + |\na u|^2 + |u|^2 + |f(u)| + |uf'(u)| dx,}
for any $R>0$ and $c\in\R^d$. Then we have 
\begin{cor}
Let $u$ be a forward critical element. Then for any $\e>0$, there exist $R_0(\e)>0$ and $c(t):(0,\I)\to\R^d$ such that at any $t>0$ we have
\EQ{ \label{R_0 ext bd}
 E_{R_0,c(t)}(u;t) \le \e E(u).}
\end{cor}

\subsection{Zero momentum and non-propagation} 
Next we observe that the critical element can not move with any positive speed in the sense of energy. 
For that we first need to see that the (conserved) momentum 
\EQ{ \label{def P}
 P(u) := \int_{\R^d} u_t \na u dx \in \R^d}
is zero for any critical element $u$.
\begin{lem}
For any critical element $u$, we have $P(u)=0$.
\end{lem}
\begin{proof}
For $j=1,\dots,d$ and $\la\in\R$, we define the operator $L_j^\la$ of Lorentz transform
\EQ{
 \pt L_j^\la u(x_0,\dots,x_d) = u(y_0,\dots,y_d), 
 \pr y_0 = x_0 \cosh\la + x_j \sinh\la, \pq y_j = x_0 \sinh\la + x_j\cosh\la, \pq y_k=x_k\ (k\not=0,j),}
then we have $L_j^\al L_j^\be = L_j^{\al+\be}$. 
Since $\p_\la y_0 = y_j$ and $\p_\la y_j = y_0$, we have
\EQ{
 \p_\la L_j^\la u = L_j^\la[(x_j\p_t + t\p_j)u].}
Also we have
\EQ{
 \pt \p_t L_j^\la = L_j^\la(s\p_t  + c \p_j), \pq \p_{tt} L_j^\la= L_j^\la(s^2 \p_{tt} + 2sc \p_{tj} + c^2 \p_{jj}),
 \pr \p_j L_j^\la = L_j^\la(c \p_t  + s\p_j ), \pq \p_{jj} L_j^\la= L_j^\la(c^2 \p_{tt} + 2sc \p_{tj} + s^2 \p_{jj}),}
where $s:=\sinh\la$ and $c:=\cosh\la$. In particular $[\p_t^2 - \De, L_j^\la]= 0$, and so $L_j^\la$ maps global solutions to themselves.
For the space-time norm, we have
\EQ{
 \iint L_j^\la v dtdx_j = \iint v \left|\mat{c & s \\ s & c}\right| dtdx_j
  = \iint v dt dx_j,}
hence $L_j^\la$ preserves all $L^p_{t,x}(\R^{1+d})$ norm. 
For any solution $u$, we have 
\EQ{
 \pt \p_\la^0 E(L_j^\la u) = \LR{u_t|\p_\la^0\p_t L_j^\la u} + \LR{\na u|\p_\la^0\na L_j^\la u} + \LR{u-f'(u)|\p_\la^0 L_j^\la u}
 \pr=\LR{u_t|x_ju_{tt}+tu_{tj}+u_j} + \LR{u_k|x_ju_{kt}+tu_{kj}+\de_{kj}u_t} 
 \pr\pq\pq\pq +\LR{u-f'(u)|x_ju_t+tu_j}
 \pr=\LR{x_ju_t|\De u} + 2\LR{u_t|u_j} -\LR{x_ju_{kt}|u_k} = \LR{u_t,u_j}=P(u),}where $\p_\la^0:=\p_\la|_{\la=0}$. 
If $P_j(u)\not=0$ for some $j$, then we obtain another global solution $L_j^\la u$, which has smaller energy and infinite Strichartz norm. It also belongs to $\K^+$, by continuity. More precisely, the continuity of $L_j^\la u$ in $\la$ in the energy space easily follows from the local wellposedness if $u$ has compactly supported initial data. Then the original solution is approximated by smooth cut-off using the finite propagation property. 
Thus we obtain another critical element with less energy, a contradiction. 
Hence $P(u)=0$. 
\end{proof}

Next we see stillness of critical elements in terms of the energy propagation. 
For any $R>0$, we define the localized center of energy $X_R(t)\in\R^d$ by 
\EQ{ \label{def XR}
 X_R(u;t) := \int \chi_R(x) xe(u)(t,x) dx,}
where $\chi_R$ is as defined in \eqref{def chiR}, and 
$e(u)$ denote the energy density of $u$, namely
\EQ{ \label{def e}
 e(u)=(|u_t|^2+|\na u|^2+|u|^2)/2-f(u).} 
From the energy identity $\dot e(u)=\na\cdot(u_t\na u)$, we get for any solution $u$
\EQ{
 \frac{d}{dt} X_R(u;t) = -d P(u) + \int[d(1-\chi_R(x))+(r\p_r)\chi_R(x)]u_t\na u.}
If $u$ is a critical element, the first term disappears by the above lemma, so we have
\EQ{ \label{propa est}
 \left|\frac{d}{dt} X_R(u;t)\right| \lec E_{R,0}(u;t).} 
Moreover, since $u$ is in ${\mathcal K}^+$, by Lemma \ref{lem:K bd} there exists $\de_0\in (0,1)$ such that
\EQ{ \label{de inf}
 K_{1,0}(u(t)) \ge \de_0 \|u(t)\|_{H^1}^2}
for all $t\in\R$. 
\begin{lem} \label{lem:non-propa}
Let $u$ be a forward critical element, and let $R_0(\e)>0$, $c(t)\in\R^d$ and $\de_0>0$ be as in \eqref{R_0 ext bd} and \eqref{de inf}. If $0<\e\ll\de_0$ and $R\gg R_0(\e)$ then we have 
\EQ{
 |c(t)-c(0)| \le R-R_0(\e),}
for $0<t<t_0$ till some $t_0\gec \de_0 R/\e$.
\end{lem}
\begin{proof}
By translation in $x$, we may assume that $c(0)=0$. Let $t_0$ be the final time for the above property
\EQ{
 t_0 = \inf\{t>0 \mid |c(t)|\ge R-R_0\}.}
Then the finite speed of propagation implies that $t_0>0$. For any $0<t<t_0$ we have $|c(t)|\le R-R_0$, hence by \eqref{R_0 ext bd} we have $E_{R,0}\le \e E(u)$, and so by \eqref{propa est} we get
\EQ{
 \left|\frac{d}{dt} X_R(u;t)\right| \lec \e E(u).}
Next we expand it around $c$:
\EQ{ \label{c expand} 
 c(t)\cdot X_R(u;t) = |c(t)|^2 \int \chi_R(x) e(u) dx + \int \chi_R(x) c\cdot(x-c) e(u) dx,}
where the first term on the right is bounded from below by
\EQ{
 \pt E(u) - \int (1-\chi_R(x))e(u) dx 
 \pr\ge \|\dot u(t)\|_{L^2_x}^2/2 + K_{1,0}(u(t)) - CE_{R,0}(t) 
 \ge \de_0 E(u) - C\e E(u)\gec \de_0 E(u),}
since $\e\ll\de_0$. 
The second term of \eqref{c expand} is dominated by splitting the integral into $|x-c|\le R_0$ and $|x-c|\ge R_0$. 
In the interior it is bounded by using the energy bound, and in the exterior it is bounded by using \eqref{R_0 ext bd}.
Thus we obtain 
\EQ{
 \left|\int \chi_R(x) c\cdot(x-c) e(u) dx\right| \lec (R_0 + R \e)E(u)|c|.}
In the same way we have 
\EQ{
 |X_R(u;0)| \lec (R_0 + R \e)E(u),}
since $c(0)=0$. Thus we get 
\EQ{
 \de_0 E(u) |c(t)| \lec (R_0  + R \e + \e t)E(u),}
and sending $t\to t_0$, we get 
\EQ{
 \de_0 R \lec \e t_0.}
\end{proof}

\subsection{Dispersion and contradiction}
Finally we use the localized virial identity to see dispersion of the critical element, which will contradict the above  non-propagation property. For any $R>0$, we define the localized virial $V_R(u;t)\in\R$ by 
\EQ{ \label{def VR}
 V_R(u;t) := \LR{\chi_R(x) u_t| (x\cdot\na +\na\cdot x)u},}
where $\chi_R$ is as defined in \eqref{def chiR}. 
Then we have for any solution $u$,
\EQ{
 \frac{d}{dt} V_R(u;t) \pt= -\int \chi_R(x)[2|\na u|^2-d(D-2)f(u)] + \frac{d}{2}|u|^2\De\chi_R(x) dx
 \prq - \int r\p_r\chi_R(x)[|u_t|^2+2|u_r|^2-|\na u|^2-|u|^2+2f(u)]  dx
 \pr\le -K_{d,-2}(u(t)) + CE_{R,0}(u;t).} 
If $u$ is a critical element, then $u\in{\mathcal K}^+$ and hence by Lemma \ref{lem:K bd}, there exists $\de_2\in (0,1)$ such that 
\EQ{
 K_{d,-2}(u(t))\ge \de_2\|\na u(t)\|_{L^2_x}^2} 
for all $t>0$. Thus we obtain, integrating in $t$, 
\EQ{ \label{loc vir}
 V_R(u;t_0) \le V_R(u;0) - \de_2 \int_0^{t_0} \|\na u(t)\|_{L^2_x}^2 dt + C\e E(u) t_0.}

Now by the compactness Lemma \ref{precompact}, we have 
\begin{lem}
Let $u$ be a forward critical element. Then for any $\e>0$ there exists $C>0$ such that 
\EQ{
 \|u(t)\|_{L^2_x}^2 \le C\|\na u(t)\|_{L^2_x}^2 + \e \|\dot u(t)\|_{L^2_x}^2,}
for all $t>0$. 
\end{lem}
\begin{proof}
Otherwise there exists a sequence $t_n>0$ such that 
\EQ{
 \|u(t_n)\|_{L^2_x}^2 > n\|\na u(t_n)\|_{L^2_x}^2 + \e \|\dot u(t_n)\|_{L^2_x}^2.}
Since $u$ is $L^2_x$ bounded, it follows $\|\na u(t_n)\|_{L^2_x}\to 0$. 
Then Lemma \ref{precompact} implies that, after passing to a subsequence, $u(t_n)\to 0$ strongly in $H^1_x$, then the above inequality implies that $\dot u(t_n)\to 0$ too. 
Hence $E^Q(u;t_n)\to 0$, which contradicts the energy equivalence, Lemma \ref{E0 equiv K+}. 
\end{proof}
Multiplying the equation with $u$, and then applying the above lemma with $\e=1/4$, we obtain  
\EQ{
 \p_t \LR{u|\dot u} \pt= \int_{\R^d} |\dot u|^2 - |\na u|^2 - |u|^2 + Df(u) dx.
  \pr\ge \int_{\R^d} |\dot u|^2/2 + |u|^2 - C|\na u|^2 dx,}
with some $C>0$. Hence 
\EQ{
 \int_0^{t_0} \|\dot u\|_{L^2_x}^2 + \|u\|_{L^2_x}^2 dt \lec E(u) + \int_0^{t_0}\|\na u\|_{L^2_x}^2 dt,} 
and so 
\EQ{ \label{int E bd H1}
 t_0 E(u) \le \int_0^{t_0} E^Q(u;t)dt \lec E(u) + \int_0^{t_0}\|\na u\|_{L^2_x}^2 dt.} 

Now we choose positive $\e\ll\de_2\de_0$ and $R\gg R_0(\e)$.
Then by Lemma \ref{lem:non-propa} there exists $t_0\sim \de_0 R/\e$ such that $E_{R,0}(u;t)\le\e E(u)$ for $0<t<t_0$.
Then from \eqref{loc vir} and \eqref{int E bd H1}, we have 
\EQ{
 -V_R(u;t_0)+V_R(u;0) \gec [\de_2 t_0 - C\e t_0 - C]E(u) \gec \de_2 t_0 E(u) \sim \frac{\de_2\de_0R}{\e}E(u),}
while the left hand side is dominated by $RE(u)$, which is a contradiction when $\e/\de_2\de_0$ is sufficiently small. \qedsymbol 

\appendix

\section{The range of scaling exponents}
In Section \ref{sect:var}, we have shown that $m_{\al,\be}$ in \eqref{min J} is positive and achieved (after modification of the mass in the critical/exponential cases) if $(\al,\be)$ satisfies \eqref{range albe}. 
Here we see that it is also necessary, modulo the obvious symmetry $(\al,\be)\to(-\al,-\be)$. 
For simplicity, we consider only the pure power nonlinearity. 

\begin{prop} \label{need range}
Assume that neither $(\al,\be)\in\R^2$ nor $(-\al,-\be)$ satisfies \eqref{range albe}. Then there exists $q\in(2_\star,2^\star)$ such that we have $m_{\al,\be}=-\I$ for $f(\fy)=|\fy|^{q}$. 
\end{prop}
\begin{proof}
By symmetry with respect to $(\al,\be)\to(-\al,-\be)$, we may assume that $\be>0$ and $\bar\mu=2\al+d\be>0$. 

First we consider the case $\al<0$ and $\U\mu>0$, which implies that $d\ge 2$. 
Let $(2_\star,2^\star)\ni q=2+p$, then we have 
\EQ{
 \al p + \bar\mu \ge d\U\mu/(d-2)>0.}
Decompose $K(\fy)$ by 
\EQ{
 K=K_1+K_2,\pq K_1(\fy)=\U\mu\frac{\|\na\fy\|_{L^2}^2}{2},
 \pq K_2(\fy)=\bar\mu\frac{\|\fy\|_{L^2}^2}{2}-(\al p+\bar\mu)F(\fy).}
Suppose that $0\not=\fy\in H^1(\R^d)$ satisfies $K_2(\fy)=0$. 
If there is no such $\fy$, then $K$ is positive definite and the minimization set in \eqref{min J} becomes empty. 
Let $1<\nu\to 1+0$, then we have 
\EQ{ \label{K asy nufy}
 0>K_2(\nu\fy)\to K_2(\fy)=0,\pq K_1(\nu\fy)\to K_1(\fy)>0.}
Now let $\la(\nu)>0$ solve 
\EQ{
 0=K(\nu\fy(x/\la)) = \la^{d-2}K_1(\nu\fy)+\la^dK_2(\nu\fy),}
in other words $\la(\nu)=[-K_2(\nu\fy)/K_1(\nu\fy)]^{1/2}$. 
Then $\la(\nu)\to\I$ as $\nu\to 1+0$ due to \eqref{K asy nufy}. 
Since 
\EQ{
 \bar\mu J(\psi)=K(\psi)+\be\|\na\psi\|_{L^2}^2+\al pF(\psi),}
we obtain 
\EQ{
 \bar\mu J(\nu\fy(x/\la)) = \be\nu^2\la^{d-2}\|\na\fy\|_{L^2}^2 + \al p\la^d F(\nu\fy) \to -\I,}
which implies that $m=-\I$. 

Next, if $\bar\mu=0>\al$, which implies $d\ge 2$, then for any nonzero $\fy\in H^1(\R^d)$ satisfying $K(\fy)=0$ we have 
\EQ{
 K(\fy(x/\la))=\la^d K(\fy)=0,}
and similarly as above, $J(\fy(x/\la))=O(-\la^d)\to\I$ as $\la\to\I$. 

Finally consider the case $\U\mu<0<\bar\mu$. 
Then $\al p+2\be=0$ has a solution $p\in(4/d,2^\star-2)$. 
Since $\al p+\bar\mu=\al p+2\be+\U\mu$, there exists $p\in(4/d,2^\star-2)$ such that 
\EQ{
 \al p + \bar\mu < 0 < \al p + 2 \be.}
Then $K^N(\fy)=-(\al p+\bar\mu)F(\fy)$ is positive and so for any $\fy\in H^1(\R^d)$, $K(\nu\fy)\ge 0$ if $\nu\gg 1$. 
Since the kinetic term in $K$ is negative, there exists $\xi(\nu)\in\R^d$ such that $K(e^{i\x x}\nu\fy)=0$. Since 
\EQ{
 -\U\mu J(\psi) = -K(\psi) + 2\be\frac{\|\fy\|_{L^2}^2}{2} - (\al p+2\be) F(\psi),}
we obtain 
\EQ{
 -\U\mu J(e^{i\x x}\nu\fy)=2\be\nu^2\frac{\|\fy\|_{L^2}^2}{2} - (\al p+2\be)F(\nu\psi) \to -\I,}
which implies that $m=-\I$. 
\end{proof}
The above proof shows that if $\al<0$ and $\U\mu\ge 0$ then $m=-\I$ for all $q\in(2,2^\star]$. The choice of $q$ was needed only in the other region. 

\section{Table of Notation}
The notation below applies to any $s\in\R$, $\nu\ge 0$, $(\al,\be)\in\R^2$, $j,k\in\Z$, $Z\in\R^3$, $I\subset\R$, $\fy,\psi\in H^1(\R^d)$, $u\in C_t(H^1_x(\R^d))$, any suffix $\diamondsuit,\heartsuit$, any sequence $\fy_n\in H^1(\R^d)$, and any functional $G$ on $H^1(\R^d)$. 
{
\small

\begin{longtable}{l|l}
   \tabhead{Dimension and scaling}
 $d\in\N$, $2_\star,2^\star>0$: space dimension and critical powers & \eqref{def 2star} \\
 $\al,\be\in\R$, $\bar\mu\ge\U\mu\ge0$: scaling exponents and their functions &  \eqref{def mu} \\
 $\fy_{\al,\be}^\la$, $\Dab G$: rescaled family and scaling derivative & \eqref{def scaling},\eqref{def Dab} \\ 
 \pq (subscript of the form $\spadesuit_{\al,\be}$ is often omitted as $\spadesuit$) \\  
   \tabhead{\it 1st order representation}
 $\V u\ \leftrightarrow\ u$: linked with each other by & \eqref{def vec} \\
   \tabhead{Nonlinearity} 
 $F(\fy),f(s)\ge 0$: nonlinear energy and its density & \eqref{def J} \\
 $f_S(s),f_L(s)\ge 0$: small and large parts of $f$ & \eqref{def fSL} \\
 $p_1,p_2>0$, $\ka_0\ge 0$: leading powers of $f_S$ and $f_L$ & \eqref{f_S}, \eqref{f sub}, \eqref{f exp} \\
   \tabhead{Functionals} 
 $J(\fy),J^{(\nu)}(\fy)\in\R$: static energy, with mass change & \eqref{def J}, \eqref{def Jc} \\
 $K_{\al,\be}(\fy),K^{(c)}_{\al,\be}(\fy)\in\R$, $H_{\al,\be}(\fy)\ge 0$: derivatives of $J$ & \eqref{def K},\eqref{def H} \\
 $K_{\al,\be}^Q(\fy), K_{\al,\be}^N(\fy)\in\R$:  quadratic and nonlinear parts of $K$ & \eqref{def KQN} \\
 $E(u;t),E(\fy,\psi),e(u)\in\R$:  total energy and its density & \eqref{def E},\eqref{def e} \\
 $E^Q(u;t),E^Q(\fy,\psi)\ge 0$:  linear energy & \eqref{def E^Q} \\
 $\ti E(\fy),\ \ti K_{\al,\be}(\fy)\in\R$:  vector versions of $E$ and $K$ & \eqref{def tiE} \\
 $P(u;t), E_{R,c}(u;t)\in\R$:  momentum and exterior energy & \eqref{def P},\eqref{def ER} \\
 $X_R(u;t), V_R(u;t) \in\R$:  localized energy center and virial & \eqref{def XR},\eqref{def VR} \\
   \tabhead{Variational splittings} 
 $m_{\al,\be},E^\star \ge 0$:  static and scattering energy thresholds & \eqref{min J},\eqref{def E*} \\
 $\K_{\al,\be}^\pm$, $\ti\K_{\al,\be}^+$:  splitting below the threshold & \eqref{def Kpm},\eqref{def tiK} \\
 $\CTM^\nu(G),\CTM^\star(G)\in[0,\I]$:  Trudinger-Moser ratio & \eqref{def CTMA},\eqref{def CTMS} \\
 $\mm(G)\in[0,\I]$:  Trudinger-Moser threshold on $\dot H^1$ & \eqref{def mm} \\
 $conc.G((\fy_n)_n)\in\R$:  concentration at $x=0$ & \eqref{def conc} \\
   \tabhead{Function spaces and exponents} 
 $[Z]_\nu(I),[Z]_0(I),[Z]_\nu^\bul(I)$:  Lebesgue-Besov spaces on $I\times\R^d$ & \eqref{def Yq} \\ 
 $Z^s, Z^{*(s)}\in\R^3$:  regularity change and dual of exponents & \eqref{exp change} \\
 $\reg^\th(Z),\str^\th(Z),\dec^\th(Z)\in\R$:  regularity and decay indexes & \eqref{def regstr} \\
 $H,W,K,M^\sharp,V\in\R^3$:  exponents for $d\in\N$ & \eqref{def WK},\eqref{def MsharpX},\eqref{def V}\\
 $X,S,L\in\R^3$:  exponents for $d\le 4$ & \eqref{def MsharpX},\eqref{def SL}\\
 $\ti M,M,\hat M,\ti N,N,Q,P,Y,R,G\in\R^3$:  exponents for $d\ge 5$ & \eqref{def M2X},\eqref{def hatM},\eqref{def G} \\
 $H_\e,W_\e,M^\sharp_\e\in\R^3$:  exponents for $d\ge 5$ & \eqref{def norm_eps} \\ 
 $H^1_\nu$, $\MP$:  $H^1(\R^2)$ and a set of Fourier multipliers on $\R^d$ & \eqref{def H1mu}, \eqref{def MP} \\
 $\X$, $\Y$, $\Y_0$, $\ti\Y$, $\Y_0^*$, $\Y^*$: Strichartz-type spaces & \eqref{def SL}, \eqref{def YY} \\
 $ST(I),ST^*(I),ST_\I^\diamondsuit(I)$:  Strichartz-type spaces on $I\times\R^d$& \eqref{def ST},\eqref{def STI} \\ 
   \tabhead{Profile decomposition}
 $(t_\heartsuit^\diamondsuit,x_\heartsuit^\diamondsuit,h_\heartsuit^\diamondsuit)\in\R^{1+d}\times[0,1]$: time-space-scale shift parameter & Section \ref{ss:lin prof}\\
 $\ga_\heartsuit^\diamondsuit=-t_\heartsuit^\diamondsuit/h_\heartsuit^\diamondsuit\in\R$: rescaled time shift \\
 $h_\I^\diamondsuit\in\{0,1\}$,$\ga_\I^\diamondsuit\in[-\I,\I]$: limit of $h_n^\diamondsuit$ and $\ga_n^\diamondsuit$ \\
 $T_\heartsuit^\diamondsuit \fy, \LR{\na}_\heartsuit^\diamondsuit \fy$: operators dependent on $(x_\heartsuit^\diamondsuit,h_\heartsuit^\diamondsuit)$ & \eqref{def Tnj} \\
 $(t^{jl}_n,x^{jl}_n,h^{jl}_n)$, $S^{jl}_n u$: relative shift and transform & \eqref{def njl},\eqref{def Snjl} \\
 $\tau_\heartsuit^\diamondsuit\in\R$, $\tau_\I^\diamondsuit\in[-\I,\I]$: scaled time shift and its limit & \eqref{def Tnj}\\
 $\V U_\I^\diamondsuit$, $\hat U_\I^\diamondsuit$, : nonlinear profiles (scaled limit) & \eqref{NP eq},\eqref{def hat U} \\
 $\V u_{(n)}^j$, $\V u_{(n)}^{<k}$: nonlinear profiles (in original scales) & \eqref{def u(n)},\eqref{def NPD} \\
 \hline
\end{longtable}
}

\section*{Acknowledgments}
The authors thank Guixiang Xu for pointing out several mistakes in the first manuscript. 
S.~Ibrahim is partially supported by NSERC\# 371637-2009 grant and a start up fund from University of Victoria.

\end{document}